\documentclass[reqno,11pt]{amsart}
\usepackage{amsmath,amssymb,latexsym,soul,cite,mathrsfs}

\usepackage{mathtools}

\usepackage{color,enumitem,graphicx}
\usepackage[colorlinks=true,urlcolor=blue,
citecolor=red,linkcolor=blue,linktocpage,pdfpagelabels,
bookmarksnumbered,bookmarksopen]{hyperref}
\usepackage[english]{babel}
\usepackage[left=2.6cm,right=2.6cm,top=2.9cm,bottom=2.9cm]{geometry}
\usepackage[hyperpageref]{backref}
\usepackage{comment}
\usepackage{enumerate}

\numberwithin{equation}{section}

\newtheorem{theorem}{Theorem}
\newtheorem{lemma}[theorem]{Lemma}
\newtheorem{corollary}[theorem]{Corollary}
\newtheorem{proposition}[theorem]{Proposition}
\newtheorem{remark}[theorem]{Remark}
\newtheorem{definition}[theorem]{Definition}

\newtheorem{lemmaletter}{Lemma}

\newcommand{\innerthmname}{}

\theoremstyle{definition}

\makeatletter
\def\namedlabel#1#2{\begingroup
	#2%
	\def\@currentlabel{#2}%
	\phantomsection\label{#1}\endgroup
}
\makeatother

\newcommand{\ud}{\hspace{0.05cm}\mathrm{d}}

\DeclareMathOperator{\Ric}{Ric}

\newcommand{\R}{\mathbb{R}}
\newcommand{\Ss}{\mathbb{S}}
\newcommand{\Div} {\operatorname{div}}
\newcommand{\eps}{\varepsilon}

\title[Quantitative stability near minimizing $Q$-curvature metrics]{Quantitative stability of the total 
$Q$-curvature near minimizing metrics}  
 
\author[J.H. Andrade]{Jo\~{a}o Henrique\ Andrade}
\author[T. K\"{o}nig]{Tobias K\"{o}nig}
\author[J. Ratzkin]{Jesse Ratzkin}
\author[J. Wei]{Juncheng Wei}

\address[J.H. Andrade]{
	Department of Mathematics,
	University of S\~ao Paulo
	\newline\indent 
	05508-090, S\~ao Paulo-SP, Brazil}
\email{\href{mailto:andradejh@ime.usp.br}{andradejh@ime.usp.br}}

\address[T. K\"{o}nig]{
	Institut für Mathematik, 
	Goethe-Universität Frankfurt, 
	\newline\indent 
	60325 Frankfurt am Main, Germany}
\email{\href{koenig@mathematik.uni-frankfurt.de}{koenig@mathematik.uni-frankfurt.de}}

\address[J. Ratzkin]{Department of Mathematics,
	Universit\"{a}t W\"{u}rzburg
	\newline\indent
	97070, W\"{u}rzburg-BA, Germany}
\email{\href{mailto:jesse.ratzkin@mathematik.uni-wuerzburg.de}{jesse.ratzkin@mathematik.uni-wuerzburg.de}}

\address[J. Wei]{
	Department of Mathematics,
	University of British Columbia
	\newline\indent 
	V6T 1Z2, Vancouver-BC, Canada
	\newline\indent
	and
	\newline\indent 
	Department of Mathematics,
	The Chinese University of Hong Kong
	\newline\indent 
	Shatin-NT, Hong Kong}
\email{\href{mailto:jcwei@math.ubc.ca}{jcwei@math.ubc.ca}}
\email{\href{mailto:wei@math.cuhk.edu.hk}{wei@math.cuhk.edu.hk}}

\thanks{This research was supported by the National Council for Scientific and Technological Development (CNPq) \#441922/2023-6, \#409764/2023-0 and \#443594/2023-6, Funda\c c\~ao de Amparo \`a Pesquisa do Estado de S\~ao Paulo (FAPESP) \#2021/07566-3 and \#2023/15934-0, and Natural Sciences and Engineering Research Council of Canada (NSERC) \#RGPIN-2018-03773}
\date{July 09, 2024}
\begin{document}
	
	\begin{abstract}
		Under appropriate positivity hypotheses, we prove quantitative estimates for the total $k$-th order 
        $Q$-curvature functional near minimizing metrics on any smooth, closed $n$-dimensional Riemannian 
        manifold for every integer $1 \leq k < \frac{n}{2}$.  
        More precisely, we show that on a generic closed Riemannian manifold the distance to the minimizing 
        set of metrics is controlled quadratically by the $Q$-curvature energy deficit, extending recent 
        work by Engelstein, Neumayer and Spolaor \cite{ENS} in the case $k=1$. Next we prove, for any integer $1 \leq k< \frac{n}{2}$, the existence of an 
        $n$-dimensional Riemannian manifold such that the $k$-th order $Q$-curvature deficit controls a higher power of the distance to the minimizing set. We believe that these degenerate examples are of independent interest and can be 
        used for further development in the field.
	\end{abstract}
	
	\maketitle
	
	
	\begin{center}
		\footnotesize
		\tableofcontents
	\end{center}
	
	\section{Introduction and main results}

\subsection{The $k$-th order $Q$-curvature problem}
    We consider a compact Riemannian manifold $(M,g)$ of dimension $n$ without boundary. 
    In 
    1992, Graham, Jenne, Mason and Sparling \cite{GJMS} constructed a conformally invariant 
    operator $P_{g,k}$ whose leading term is $(-\Delta_g)^k$ for each integer $1 \leq k < \frac{n}{2}$, 
    where $\Delta_g$ is the Laplace--Beltrami operator of $g$ (see Appendix~\ref{GJMSformulas} for more details). 
    
    The operator $P_{g,k}$ is now known as the GJMS operator of order $2k$.
    It is naturally constructed from curvature quantities of $g$ and satisfies the transformation law 
    \begin{equation} \label{conv_inv_law} 
     P_{\widetilde g,k} (\phi) 
    = u^{-\frac{n+2k}{n-2k}} P_{g,k}(\phi u) \quad {\rm when} \quad \widetilde g = u^{\frac{4}{n-2k}} g.
    \end{equation} 
    Subsequently, the authors defined the (scalar-valued) curvature quantity 
    $Q_{g,k} = \frac{2}{n-2k} P_{g,k} (1)$. Substituting 
    $\phi=1$ into \eqref{conv_inv_law}, we see 
    \begin{equation} \label{q_trans_law}  
    Q_{\widetilde g,k} = \frac{2}{n-2k} u^{-\frac{n+2k}{n-2k}} P_{g,k} (u) \quad {\rm when} \quad 
    \widetilde g = u^{\frac{4}{n-2k}} g.
    \end{equation} 

    Motivated by this conformal invariance and the analysis of the classical 
    Yamabe problem, one poses the $k$th--order Yamabe problem: given a compact 
    Riemannian manifold $(M,g)$ of dimension $n$ and an integer $1 \leq k <  \frac{n}{2}$, we 
    seek a conformal metric $\widetilde g = u^{{4}/{(n-2k)}} g$ 
    such that $Q_{\widetilde g,k}$ is constant. (Here and below we fix the 
    background metric $g$ and identify the conformal metric $\widetilde{g}=
    u^{{4}/{(n-2k)}} g$ with its conformal factor $u$.)
    By \eqref{q_trans_law} this is equivalent to solving the PDE 
    \begin{equation} \label{const_q_pde} 
    P_{g,k} (u) = c u^{\frac{n+2k}{n-2k}} \quad {\rm on } \quad M,
    \end{equation} 
    where $c$ is a constant. We denote the set of solutions by 
    \begin{eqnarray*} 
    \mathcal{CQC}_{g,k} & = & \left\{ \widetilde g \in [g] : Q_{\widetilde g,k} 
    \textrm{ is constant}\right\} \\
    & = & \left\{ u \in W^{k,2}(M) : u> 0 \; \textrm{a.e.} \; \text{and} \; P_{g,k}(u) 
    = \lambda u^{\frac{n+2k}{n-2k}} \textrm{ for some } \lambda \in \R \right\}.
    \end{eqnarray*} 
    Since $2_k^*:=\frac{2n}{n-2k}$, Eq. \eqref{const_q_pde} has critical growth in the sense of the embedding $W^{k,2}(M)\hookrightarrow L^{2^*_k}(M)$.
    
    To establish a variational setting, we introduce the functional 
    \begin{equation} \label{functional} 
    \mathcal{Q}_{g,k}(u) = \frac{\int_{M} Q_{\widetilde g,k} \ud\mu_{\widetilde g}}
    {\operatorname{vol}_{\widetilde g}(M)^{\frac{n-2k}{n}}} 
    = \frac{2}{n-2k} \frac{\int_M u P_{g,k} (u) \ud\mu_g}{\left ( 
    \int_M u^{\frac{2n}{n-2k}} \ud\mu_g \right )^{\frac{n-2k}{n}}} .
    \end{equation} 
    We show below in Lemma \ref{lem:functional_regularity} that $\mathcal{Q}_{g,k}$ is a $\mathcal{C}^2$-functional 
    on $W^{k,2}(M)$. Furthermore, it follows from this proof that $\widetilde{g}$ is a 
    critical point of $\mathcal{Q}_{g,k}$ if and only if $Q_{\widetilde{g}, k}$ is 
    constant, which is in turn equivalent to $u$ solving \eqref{const_q_pde} with     
    $$c = \mathcal{Q}_{g,k} (u) \left (\| u \|_{L^{2_k^*}(M)} 
    \right ) ^{\frac{4k}{n-2k}}.$$
    Observe that $\mathcal{Q}_{g,k}$ is scale-invariant, so we will often restrict our 
    attention to unit-volume metrics in the conformal class $[g]$, or equivalently 
    $$\mathcal{B} = \left \{ u \in W^{k,2}(M) : u > 0 \textrm{ a.e.} \; {\rm and} \; \| u \|_{L^{2_k^*}(M)} = 1\right \} .$$
    We denote this restricted solution set by $\mathcal{CQC}_{g,k}^*= \mathcal{CQC}_{g,k} 
    \cap \mathcal{B}$.

    The variational setting suggests that we seek solutions in $\mathcal{CQC}_{g,k}$ 
    by studying a sequence minimizing the quotient $\mathcal{Q}_{g,k}$, and so we naturally 
    define the $k$-order Yamabe invariant 
    $$\mathcal{Y}_{k,+}(M,[g]) = \inf \left\{ \mathcal{Q}_{g,k}(u) : u \in W^{k,2} (M) \; {\rm and} \; 
    u > 0 \textrm{ a.e.} \right\}$$
    and the minimizing set 
    $$\mathcal{M}_{g,k} = \left\{ u \in W^{k,2}(M) \, :  \, u > 0 \text{ a.e.} \; {\rm and} \; 
    \mathcal{Q}_{g,k} (u) = \mathcal{Y}_{k,+}
    (M,[g]) \right\}. $$
    Observe that $\mathcal{M}_{g,k} \subset \mathcal{CQC}_{g,k}$. Once again, we 
    often restrict to minimizing solutions with unit volume, {\it i.e.} $\| u \|_{L^{2_k^*}(M)} = 1$ and denote this 
    restricted set by $\mathcal{M}_{g,k}^*$.

     The plus sign in the definition of $\mathcal{Y}_{k,+}$ signifies it is the infimum over functions in $W^{k,2}(M)$ that 
     are positive almost everywhere. If $k=1$, one can use the maximum principle to show the minimizer over all functions 
     in $W^{k,2}(M)$ is automatically positive almost everywhere. So we may examine the infimum over all functions 
     in $W^{k,2}(M)$. If $k \geq 2$, we lack the maximum principle. Hence, in general, there is no guarantee a minimizing 
     function is an admissible conformal factor. 
   
    To place these curvature quantities in a more familiar setting, 
    we remind the reader that $Q_{g,1}$ is (up to multiplication by a constant depending only 
    on $n$) the scalar 
    curvature $R_g$ and $P_{g,1}$ is the usual conformal Laplacian 
    $$P_{g,1} = -\Delta_g + \frac{n-2}{4(n-1)} R_g.$$
    In addition, we have
    \begin{equation}
        \label{Q curvature definition}
        Q_{g,2} = -\frac{1}{2(n-1)} \Delta_g R_g - \frac{2}{(n-2)^2} 
    |\Ric_g|^2 + \frac{n^3-4n^2+16n -16}{8(n-1)^2(n-2)^2} R_g^2
    \end{equation}
    is Branson's (fourth-order) $Q$-curvature and     
    $P_{g,2}$ is the Paneitz operator, given by 
    \begin{equation}
        \label{paneitz definition}
        P_{g,2} (u) = (-\Delta_g)^2u + \Div 
    \left ( \frac{4}{n-2} \Ric_g (\nabla u, \cdot) - 
    \frac{(n-2)^2+4}{2(n-1)(n-2)} R_g \nabla u \right ) 
    + \frac{n-4}{2} Q_{g,2},
    \end{equation}
    where $\Ric_g$ is the Ricci tensor. Thus \eqref{const_q_pde} reduces to 
    $$P_{g,1} (u) = -\Delta_g(u) + \frac{4(n-1)}{n-2} R_g u 
    = \frac{n(n-2)}{4} u^{\frac{n+2}{n-2}}$$
    in the second order case, which is one of the most well-studied 
    partial differential equations in geometric analysis, and 
    $$P_{g,2}(u) = \frac{n(n-4)(n^2-16)}{16} u^{\frac{n+4}{n-4}}$$
    in the fourth order case. 

    The existence of solutions in general and minimizing solutions in particular 
    is often a delicate question. In the classical case of the Yamabe problem ({\it i.e.} 
    $k=1$), the search started with Yamabe's work \cite{Yam} and continued with 
    the important contributions of Trudinger \cite{Tru} and Aubin \cite{Aub}. Schoen \cite{Sch1} 
    finally resolved the Yamabe problem, showing that any Riemannian metric on a 
    compact manifold without boundary is conformal to a constant scalar curvature metric. 
    Schoen's solution uses the Green's function of the conformal Laplacian in a 
    fundamental way, highlighting the important connection between the Green's function 
    and scalar curvature. Since then, a sizeable community has sought to understand 
    the solution and minimizing sets in all possible scenarios; we do not attempt to 
    summarize the extensive literature here. We only mention some results characterizing 
    the solution set in various situations. Combining the work of Trudinger and Aubin, one 
    sees that if $\mathcal{Y}_{1,+} (M,[g]) < 0$ then there exists a unique solution, whereas
    in the positive case the existence of many solutions is possible. Schoen conjectured 
    for many years the set of solutions is compact unless $(M,g)$ is conformally equivalent to the round sphere. 
    Eventually, Khuri, Marques and Schoen \cite{Khuri_Marques_Schoen} verified this 
    conjecture under positive mass theorem in the case that $n \leq 24$, while Marques and Brendle \cite{Marques_Brendle} 
    demonstrated noncompactness in higher dimensions. 
    
    The search for solutions and/or minimizers is, as expected, more complicated 
    in the higher order case. Here we highlight only some relatively recent results. 
    Gursky and Malchiodi \cite{Gursky_Malchiodi} showed that if $R_g \geq 0$ and $Q_{g,2} \geq 0$ 
    but not identically zero, then the Green's function of $P_{g,2}$ is positive and the 
    solution set $\mathcal{CQC}_{g,2}$ is nonempty. Later Hang and Yang \cite{HY} weakened 
    the hypotheses of Gursky and Malchiodi, showing that it suffices to assume $\mathcal{Y}_{g,1} 
    (M, [g]) > 0$. Most recently, Mazumdar and V\'etois \cite{mazumdar2022existence} showed the 
    minimizing set $\mathcal{M}_{g,k}$ is non-empty under the following conditions. First 
    they assume $\mathcal{Y}_{k,+}(M,[g])>0$. Next they assume that for every $\xi \in M$ the Green's 
    function $G_{g,k,\xi}$  (which is the unique function such that distributionally $P_{g,k} (G_{g,k,\xi})(x) 
    = \delta_\xi$, for $\delta_\xi$ the Dirac delta function at $\xi$) is positive everywhere on $M$.
    
    Assuming $G_{g,k,\xi}>0$, one can apply the analysis of Gursky 
    and Malchiodi \cite {Gursky_Malchiodi} (see also \cite{HY} and, for the general case $2 \leq k < \frac{n}{2}$, \cite{Michel2010}) to see that 
    \begin{equation} \label{Green's_expansion} 
    G_{g,k,\xi} (x) = b_{n,k} (\operatorname{dist}_g(x,\xi))^{2k-n} + m(\xi) + 
    \mathrm{o}(1).\end{equation} 
    The quantity $m(\xi) = m_g(\xi)$ is usually called the mass of the GJMS operator $P_{g,k}$ 
    at the point $\xi$. The final hypothesis of \cite{mazumdar2022existence} 
    is that if either $2k+1 \leq n \leq 2k+3$ or $(M,g)$ is locally conformally 
    flat, then $m(\xi) > 0$ for some $\xi \in M$. 

    Following the existence result of \cite{mazumdar2022existence}, we define the following 
    space of admissible metrics. Let $k \in \mathbb{N}$ with $n>2k$ and let $\alpha \in (0,1)$. 
    Observe that the space of $\mathcal{C}^{k,\alpha}$-Riemannian metrics on $M$, denoted by 
    $\mathrm{Met}^{k,\alpha}(M)$, is a convex 
    cone in the space of all symmetric, rank-two covariant tensor fields over $M$ whose 
    coefficients are $\mathcal{C}^{k,\alpha}$ functions, and we equip all these spaces of 
    tensor fields with the topology induced by convergence in the $\mathcal{C}^{k,\alpha}$ 
    norm. We say a metric $g \in \rm{Met}^{k,\alpha}(M)$ is \textbf{admissible} if it satisfies all 
    the existence hypotheses of Mazumdar and V\'etois \cite{mazumdar2022existence}, as described in the previous two paragraphs. We 
    denote the space of admissible metrics on $M$ by $\mathfrak{A}_{k,\alpha}(M)$. That is, we let 
    \begin{align*}
        \mathfrak{A}_{k,\alpha}(M) &:= \big\{ g \in \textrm{Met}^{k,\alpha}(M) \, : \, \mathcal Y_{k,+}(M, [g]) > 0, \, G_{g,k,\xi} > 0 \text{ for every } \xi \in M, \\
        & \qquad m_g(\xi) > 0 \text{ for some } \xi \in M   \big\}. 
    \end{align*}

\subsection{Quantitative stability estimates near minimizing metrics}

    We are primarily interested in the stability of the minimizing set and in 
    estimating the difference between $\mathcal{Q}_{g,k}(u)$ and its infimum in terms of the 
    distance between $u$ and the minimizing set $\mathcal{M}_{g,k}$. We 
    define 
    \begin{equation} \label{dist}
    d(u,\mathcal{M}_{g,k}) = \frac{ \inf\{ \|u-v\|_{W^{k,2}(M)} : v \in 
    \mathcal{M}_{g,k}\}} { \| u \|_{W^{k,2}(M)}}. 
    \end{equation} 
    Notice that this distance is well-defined whenever the minimizing set 
    $\mathcal{M}_{g,k}$ is non-empty, which in turn implies the solution set 
    $\mathcal{CQC}_{g,k}$ is non-empty, because $\mathcal{M}_{g,k} \subset 
    \mathcal{CQC}_{g,k}$. 
    We also interpret \eqref{const_q_pde} weakly. 
    In other words, we say that $u \in W^{k,2}(M)$ satisfies \eqref{const_q_pde} in the weak sense if
    \begin{equation*}
        \displaystyle \int_M u P_{g,k}(\phi) \ud\mu_g = 0 \quad {\rm for \ all} \quad  
        \phi \in \mathcal{C}^\infty (M).
    \end{equation*}
    
    Our first theorem in this manuscript is the following general stability 
    estimate. 

    \begin{theorem}\label{quant_stability_thm} 
     Let $n,k\in\mathbb N$ with $n>2k$ and let $(M,g)$ be a smooth, closed, $n$-dimensional Riemannian manifold.
     If $g \in \mathfrak{A}_{k,\alpha}(M)$ for some $\alpha\in (0,1)$,
    then there exists $\gamma=\gamma(g) \geq 0$ such that 
    \begin{equation}\label{stabilityestimate}
    d(u,\mathcal{M}_{g,k})^{2+\gamma}\lesssim \mathcal{Q}_{g,k}(u) - 
    \mathcal{Y}_{k,+}(M,[g]) \quad {\rm for \ all} \quad  u \in W^{k,2}(M).
    \end{equation} 
    Furthermore, there exists a subset $\mathcal{G}\subset \mathfrak{A}_{k,\alpha}(M)$ of 
    the space of admissible Riemannian metrics on $M$ which is open and dense with respect to the 
    $\mathcal{C}^{k,\alpha}$-topology such that if $g \in \mathcal{G}$,
    then \eqref{stabilityestimate} holds with $\gamma(g)=0$. 
    \end{theorem}

    We interpret \eqref{stabilityestimate} as saying there exists a constant $c$ 
    depending only on $g$ such that 
    $$c d(u, \mathcal{M}_{g,k})^{2+\gamma} \leq \mathcal{Q}_{g,k}(u) - 
    \mathcal{Y}_{k,+}(M,[g]). $$
    When $\gamma=0$ we refer to \eqref{stabilityestimate} as 
    {\bf quadratic stability}, and when $\gamma>0$ we call \eqref{stabilityestimate} \textbf{degenerate stability} or 
    {\bf higher-order stability}. The last part of the statement of Theorem \ref{quant_stability_thm} may be phrased as saying that quadratic stability happens \textbf{generically}. 

    \begin{remark}
        The ideas used to prove the genericity part of our main theorem for $2\leqslant k \leqslant \frac{n}{2}$ are in contrast with the ones in the case $k=1$.
        On the one hand, our techniques rely on some results of Case, Lin and Yuan \cite{CaseLinYuan} combined with the transversality method inspired by \cite{MR2560131,MR2982783,MR4314216}.
        On the other hand, the argument for the scalar curvature
        inspired by the ones given by Anderson \cite{MR3336633} relies on some facts that are far from being known for higher-order curvatures such as a classification result by Obata \cite{MR303464}.
    \end{remark}

    Previous work, which we summarize in this paragraph, proved quadratic 
    stability in the case that $(M,g)$ is conformally equivalent to the 
    round sphere $(\Ss^n, \overset{\circ}{g})$. Bianchi and Egnell \cite{BE} proved 
    this in the case $k=1$, then Lu and Wei \cite{LW} proved it in the case $k=2$, Bartsch, Weth and Willem \cite{MR2018667} 
    for integers $1 \leq k < \frac{n}{2}$, and finally Chen, Frank and Weth \cite{CFW} for each $k \in (0,n/2)$, 
    including non-integers. Notice that when the background manifold is conformally equivalent to the round 
    sphere $(\Ss^n, \overset{\circ}{g})$, through the stereographic projection Eq. \eqref{stabilityestimate} is a refined 
    version of the classical Sobolev inequality on the standard Euclidean space $(\mathbb R^n,\delta)$, namely 
    \[\|u\|_{W^{k,2}(\mathbb R^n)}\lesssim \|u\|_{L^{2^*_k}(\mathbb R^n)}.\] In the setting of generic manifolds, 
    Engelstein, Neumayer, and Spolaor \cite{ENS} proved the stability estimate \eqref{stabilityestimate} in the 
    case $k=1$. 
     
    From a geometric point of view, one drawback of our first main result is that the distance
    $d(u, \mathcal{M}_{g,k})$ may depend on the choice of background metric $g \in[g]$. This is because 
    the $W^{k,2}$-norm, with which $d(u, \mathcal{M}_{g.k})$ is defined, is not conformally invariant. 
    However, we can modify our distance function to obtain the following 
    conformally invariant stability estimate. 
    
    We define the following conformally invariant norm for metrics $h=u^{{4}/{(n-2k)}} g$:
    \begin{equation}
        \label{norm 2*}
         \left\|h\right\| :=\left(\int_M|u|^{\frac{2n}{n-2k}} \ud \mu_g\right)^{\frac{n-2k}{2n}}.
    \end{equation}
    The definition appears to depend on the choice of the background metric $g$, but the 
    following computation shows $\| \cdot \|$ depends only on the 
    conformal class $[g]$. If $\widehat{g} = \phi^{{4}/{(n-2k)}} g$ is a conformal metric, then 
    $h = u^{{4}/{(n-2k)}} g = \widehat{u}^{{4}/{(n-2k)}} \widehat{g}$ where 
    $u = \widehat{u} \phi$, and so 
 \[
  \| h \| = \int_M |u|^{\frac{2n}{n-2k}} \ud\mu_g = \int_M |\widehat{u}|^{\frac{2n}{n-2k}}
    \phi^{\frac{2n}{n-2k}} \ud\mu_g 
     =  \int_M |\widehat{u}|^{\frac{2n}{n-2k}} \ud\mu_{\widehat{g}}.  
 \]
      
    Similarly, in the case when $\mathcal{Y}_{k,+}(M,[g]) \geq 0$, for $h = u^{4/(n-2k)} g$ we define
    \begin{equation}
        \label{norm gjms}
        \left\|h\right\|_*=\left(\int_M u P_{g,k} (u) \ud \mu_g\right)^{1 / 2}
    \end{equation}
    for any $g \in \mathrm{Met}^\infty(M)$ with $\operatorname{vol}_g(M)=1$. Again, although $\|\cdot\|_*$ is defined 
    with respect to a fixed conformal representative, it turns out that the definition is independent of this 
    choice. Namely that for any $\hat{g} \in [g]$ and $h = u^{{4}/{(n-2k)}} g = \hat{u}^{{4}/{(n-2k)}} \hat{g} \in [g]$, one has 
        $$
        \|h\|_* = \left(\int_M u P_{g,k} (u) \ud \mu_g\right)^{1 / 2}= \left(\int_M \hat{u} P_{\hat{g},k} (\hat{u}) 
        \ud \mu_{\hat{g}}\right)^{1 / 2}.
       $$ 
    
    \begin{corollary}\label{quant_stability_metrics_thm}
          Let $n,k\in\mathbb N$ with $n>2k$ and let $(M,g)$ be a smooth, closed, $n$-dimensional Riemannian manifold. If $g \in \mathfrak{A}_{k,\alpha}(M)$ for some $\alpha\in(0,1)$ is an 
         admissible Riemannian metric, then 
        there exists $\gamma(g) \geq 0$ such that
        $$
        \left(\frac{\inf \{\|h-\tilde{g}\|: \tilde{g} \in \mathcal{M}_{g,k}\}}
        {\operatorname{vol}_h(M)^{\frac{n-2k}{2n}}}\right)^{2+\gamma} \lesssim \mathcal{Q}_{g,k}(u) 
        -\mathcal{Y}_{k,+}(M,[g]) .
        $$ 
        When $ \mathcal{Q}_{g,k}(u)-
        \mathcal{Y}_{k,+}(M,[g]) \leqslant \delta_0$ for some $0<\delta_0\ll 1$ small enough, there exists $\gamma \geq 0$ 
        depending on $(M,g)$ such that
        $$
         \left(\frac{\inf \left\{\|h-\tilde{g}\|_*: \tilde{g} \in \mathcal{M}_{g,k}\right\}}
         {\operatorname{vol}_h(M)^{1 / 2^*}}\right)^{2+\gamma}\lesssim \mathcal{Q}_{g,k}(u) 
         -\mathcal{Y}_{k,+}(M,[g]) 
         \quad {\rm for \ all} \quad h \in[g].
        $$          
        Moreover, for an open dense subset in the $\mathcal{C}^2$ topology 
        on the space of conformal classes 
        of ${\rm Met}^{\infty}(M)$, the above inequalities hold with $\gamma=0$.
    \end{corollary}

Before going on, let us quickly sketch some heuristic ideas for the proof of Theorem \ref{quant_stability_thm}. 
    In a sense we make precise in our proofs below, quadratic stability is closely related to the nondegeneracy of $u_0$ as a minimizer of $\mathcal Q_{g,k}$. 
    Indeed, if $\displaystyle g_0 = u_0^{{4}/{(n-2k)}} g$ is an element of the minimizing
    set $\mathcal{M}_{g,k}$ and we write a nearby metric in the conformal 
    class as $g_v = (u_0 + v)^{{4}/{(n-2k)}} g$, we can write out a formal 
    Taylor expansion  
    \begin{eqnarray*}
        \mathcal{Q}_{g,k}(u_0 + v) & = & \mathcal{Q}_{g,k}(u_0) + 
        D\mathcal{Q}_{g,k} (u_0)(v) + D^2\mathcal{Q}_{g,k} (u_0) (v,v) + 
        \mathcal{O}(\| v \|^3) \\ 
        & = & \mathcal{Y}_{k,+}(M,[g]) + D^2\mathcal{Q}_{g,k}(u_0) (v,v) 
        + \mathcal{O}(\| v\|^3).
    \end{eqnarray*}
    Here we used the fact that $u_0 \in \mathcal{M}_{g,k}$ and that 
    $g_0$ is a constant $Q$-curvature metric, which, as we mentioned above, 
    means $u_0$ is a critical point of $\mathcal{Q}_{g,k}$. If $u_0$ is a 
    nondegenerate critical point, then the Hessian $D^2\mathcal{Q}_{g,k} (u_0)$ does not vanish on any $v$, implying in turn that the 
    difference $\mathcal{Q}_{g,k}(u_0 + v) - \mathcal{Y}_{k,+}(M,[g])$ is indeed quadratic in $\| v \|$. Furthermore, by a generalized version of Sard--Smale's theorem (see Lemma \ref{lm:henry}), we expect nondegeneracy to happen generically. All of this is carried out in a  rigorous way in Section \ref{sec:generic_stable}. 

    \subsection{Examples for higher-order stability}

    Given the description in the previous subsection, it is of great interest 
    to produce examples of minimizing metrics satisfying a superquadratic 
    stability estimate, i.e., manifolds $(M,g)$ such that Theorem \ref{quant_stability_thm} holds 
    for some $\gamma > 0$, but not for $\gamma = 0$. Motivated by a classical example of 
    Schoen \cite{MR1173050} and a newer one by Carlotto, Chodosh and 
    Rubinstein \cite{MR3352243} (based on work by Caffarelli, Gidas and Spruck \cite{CGS} and by 
    Schoen \cite{Schoen1989}) we produce two different families of 
    constant curvature metrics, each of which satisfies a higher-order stability estimate. The degree of degeneracy of the metrics 
    as minimizers, or more generally as critical points, is made precise by the Adams-Simon positivity condition 
    ($\mathrm{AS}_p$ condition for short), for which, we refer to Definition \ref{def:adams-simon} below. 

    \begin{theorem} \label{AS3_examples}
    Let $m, \ell \in \mathbb{N}$ with $\ell \geq 2$, let $\lambda \in \R$, and let $(M,g,\lambda)$ 
    be an $m$-dimensional Einstein manifold with Einstein constant $\lambda$.  
    If $m\gg1$ is sufficiently large, the function $u\equiv 1$ is a degenerate critical 
    point of $\mathcal{Q}_{h,2}$ for certain values of $\lambda$, where $(X,h) = (M \times \Ss^\ell, g \oplus \overset{\circ}{g} )$. Moreover,
    one can replace $(\Ss^\ell, \overset{\circ}{g})$ by $(\mathbb{CP}^\ell, g_{\rm FS})$, 
    where $g_{\rm FS}$ is the Fubini-Study metric, in this example. 
    \end{theorem}
    
    Please see the statements of Propositions \ref{prop:gamma=1 m=2 sphere}
    and \ref{prop:gamma=1 m=2 CP} below for more detailed statements. 
    The reader will see that our proof of Theorem \ref{AS3_examples} is quite 
    flexible. In particular, one should be able to prove a similar result for 
    $k \geq 3$, but the computations quickly become unwieldy. While our proof is 
    inspired by the example given in  \cite[Section 5.1]{MR3352243}, our techniques
    end up being somewhat different. Carlotto {\it et al} are able to choose 
    the individual manifolds such that the product $h= g \oplus g_{\rm FS}$ 
    is Einstein on $M \times \mathbb{CP}^\ell$, whereas we cannot. This complicates 
    our verification of the $\mathrm{AS}_3$ condition, which in turn explains the requirement that
    the dimension of $M$ is large enough. 
    This is in contrast with the case $k=1$, where no restrictions are imposed on $\dim(M)$. 

    We can also exhibit an explicit example of a manifold for which Theorem \ref{quant_stability_thm} holds with $\gamma = 2$. 
    Our analysis covers every integer order $1 \leq k < \frac{n}{2}$, thus extending the 
    corresponding example for $k=1$ discussed in \cite{MR3352243}, whose degenerate 
    stability was recently analyzed in detail in \cite{F}. We fix $n > 2k$ and consider the manifold  
    \begin{equation}
     \label{M definition}
     M = \mathbb S^1(\tau_0) \times \mathbb S^{n-1}
    \end{equation} 
    with the standard (non-normalized) product metric $h$. Here $\mathbb S^{n-1}$ denotes the unit sphere in $\R^n$ and 
    $\mathbb S^1(\tau_0)$ denotes the unit
    sphere in $\R^2$ of an appropriately chosen radius $\tau_0 > 0$, see \eqref{tau 0 definition general m} below.

    \begin{theorem} \label{AS4_examples} 
    Let $n,k\in\mathbb N$ with $n>2k$.
    There exists a specific $\tau_0 > 0$, defined below as the 
    unique number satisfying \eqref{sin Pm}, 
    such that for $(\Ss^1(\tau_0) \times \Ss^{n-1}, h)$ the 
    constant function $u\equiv 1$ is a degenerate minimizer of the functional $\mathcal Q_{h,k}$ which satisfies the $\mathrm{AS}_4$ condition. 
    Moreover, the function $u \equiv 1$ satisfies 
    \eqref{stabilityestimate} with $\gamma=2$, and does not satisfy \eqref{stabilityestimate} for any $\gamma < 2$.  
    \end{theorem}

    We point out some complications we encounter in the proof of Theorem \ref{AS4_examples}, as compared to 
    its second-order counterparts in \cite{Schoen1989, MR3352243}. First, in the case that $k=1$ one only 
    encounters polynomials of degree at most two, whose roots are relatively easy to find. In our analysis, we 
    must find roots of higher-order polynomials, which we can do only through extremely careful and systematic 
    accounting. More significantly, a fundamental ingredient in \cite{Schoen1989, MR3352243} is the phase-plane 
    analysis of the second-order ODE arising from \eqref{const_q_pde} when $k=1$. This phase-plane 
    analysis allows one to quickly 
    show that all ODE solutions are, up to translations, uniquely characterized by their period and thus occur in a 
    one-parameter family. More refined arguments (see \cite[Appendix B]{MR3352243} and references therein) show that 
    the period length is actually a monotone function of the parameter. Using this, the authors of \cite{Schoen1989, MR3352243} 
    are able to conclude the crucial fact that the only minimizers of $Q_{h,1}$ are the constants for $k=1$. 
    However, for integers $k \geq 2$ the classification of solutions to the corresponding higher-order ODE is only known 
    for $k=2,3$ \cite{Frank2019, Andrade2022} and even in those cases the monotonicity of the period length is an open problem. 
    To overcome this difficulty, inspired by a remark in \cite[p. 1463]{F}, we succeed in adapting an argument due 
    to Beckner \cite[Theorem 4]{MR1230930}: see Lemma \ref{lemma Phi < Lambda gen m} and Step 3 in the proof 
    of Lemma \ref{lemma M gamma2 minimizers gen m}.

    \begin{remark}
    The examples we describe in Theorems \ref{AS3_examples} and \ref{AS4_examples} are not only 
    interesting for their novelty. Additionally, they could also provide important examples illustrating 
    the slow convergence of the geometric flow towards a constant $Q$-curvature metric. We plan to 
    address the convergence of this flow in a future paper. 
    \end{remark}

    We close this introduction with a brief outline of the rest of the paper. We 
    begin with some preliminaries in Section \ref{sec:preliminaries}, proving 
    that the total $Q$-curvature functional is regular in Section \ref{sec:funct_regularity}
    and listing some auxiliary lemmas from elsewhere in Section \ref{sec:auxiliary}.
    We prove Theorem \ref{quant_stability_thm} and Corollary \ref{quant_stability_metrics_thm}
    in Section \ref{sec:generic_stable}. In Section \ref{sec:cubic_stable} we prove 
    Theorem \ref{AS3_examples}, specifically discussing products with spheres in 
    Section \ref{sec:prod_spheres} and products with complex projective space in 
    Section \ref{sec:prod_CP}. Finally we prove Theorem \ref{AS4_examples} in 
    Section \ref{sec:quartic_stable}. We include Juhl's general recursion formulas for the GJMS 
    operators in Appendix \ref{GJMSformulas} for the reader's reference, even though 
    we do not use them in the main text. 

    \numberwithin{equation}{section} 
    \numberwithin{theorem}{section}

    \section{Notation}\label{sec:notation}
	Let us establish some standard terminology and definitions.
	In what follows, we will always be using Einstein's summation convention.
    Furthermore, we omit the subscript $g$, in the section the metric is fixed.
	
	\begin{itemize}
        \item $\mathbb{N}=\{1,2,3,\cdots\}$ and $\mathbb N_0=\mathbb{N}\cup\{0\}$;
        \item $k\in \mathbb N$ and $n>2k$
        \item $(M,g)$ is a smooth closed $n$-dimensional Riemannian manifold;
		\item $\delta=g_{\mathbb R^n}$ denotes the standard Euclidean metric; 
		\item $ \overset{\circ}{g}=g_{\mathbb S^n}$ denotes the standard round metric;
        \item $\omega_n$ denotes the volume of the Euclidean $n$-sphere;
        \item ${\rm Met}^{k,\alpha}(M)$ denotes the space of $\mathcal{C}^{k,\alpha}$-metrics in $M$, when $\alpha=0$, we simply 
        denote ${\rm Met}^{k}(M)$;
        \item ${\rm Met}^{\infty}(M)=\cup_{j\in\mathbb N}{\rm Met}^{j}(M)$ denotes the space of  smooth metrics in $M$;
		\item  $(e_i)_{i=1}^n$ denotes a local coordinate frame;
		\item $\mathfrak{T}^r_s(M)$ denotes the set of $(r,s)$-type tensor over $M$  with $\mathfrak{T}^0_0(M)=\mathcal{C}^{\infty}(M)$;
		\item ${\rm Rm}_g\in\mathfrak{T}^4_0(M)$ (or ${\rm Rm}_g\in\mathfrak{T}^3_1(M)$) denotes the (or covariant) {Riemannian curvature tensor},
		\item $\Ric_g=\operatorname{tr}_g {\rm Rm}_g\in\mathfrak{T}^2_0(M)$, traced over the first and last indices, {\it i.e.} 
        $(\Ric_g)_{ij} = g^{k\ell} \rm{Rm}_{kij\ell}$; 
		\item   $R_g={\rm tr}_g{\rm Ric}_g\in\mathfrak{T}^0_0(M)$ denotes the {scalar curvature} given by $R_g=g^{ij}{\rm Ric}_{ij}$;
		\item  $\Delta_{g}=g^{i j} \nabla_{i} \nabla_{j}$ denotes the Laplace--Beltrami operator;
		\item $\delta_g={\rm div}_g$ denotes the metric divergence;
		\item $\nabla_g$ denotes the Levi--Civita connection;
		\item ${\rm tr}_g:\mathfrak{T}^r_s(M)\rightarrow\mathfrak{T}^{r-2}_s(M)$ denotes a trace operator;
		\item $a_1 \lesssim a_2$ if $a_1 \leqslant C a_2$, $a_1 \gtrsim a_2$ if $a_1 \geqslant C a_2$, and $a_1 \simeq a_2$ 
        if $a_1 \lesssim a_2$ and $a_1 \gtrsim a_2$;
		\item $u=\mathcal{O}(f)$ as $x\rightarrow x_0$ for $x_0\in\mathbb{R}\cup\{\pm\infty\}$, if $\limsup_{x\rightarrow x_0}(u/f)(x)
        <\infty$ is the Big-O notation;
		\item $u=\mathrm{o}(f)$ as $x\rightarrow x_0$ for $x_0\in\mathbb{R}\cup\{\pm\infty\}$, if $\lim_{x\rightarrow x_0}(u/f)(x)=0$ 
        is the little-o notation;
		\item $u\simeq\widetilde{u}$, if $u=\mathcal{O}(\widetilde{u})$ and $\widetilde{u}=\mathcal{O}(u)$ as $x\rightarrow x_0$ 
        for $x_0\in\mathbb{R}\cup\{\pm\infty\}$; 
		\item $\mathcal{C}^{j,\alpha}(M)$, where $j\in\mathbb N$ and $\alpha\in (0,1)$, is the classical H\"{o}lder space over $M$;  we simply 
        denote $\mathcal{C}^{j}(M)$ when $\alpha=0$;
		\item $W^{j,q}_g(M)$ is the Sobolev space over $M$, where $j\in\mathbb N$ and $q\in[1,\infty]$; when $j=0$ we simply denote $L^{q}_g(M)$;
        \item $2^*_k=\frac{2n}{n-2k}$ is the critical exponent of the Sobolev embedding $W^{k,2}_g(M)\hookrightarrow L^{2_k^*}_g(M)$;
        \item $Q_{g,k}$ is the $2k$-th order $Q$-curvature of $g$;
         \item $P_{g,k}$ is the $2k$-th order GJMS operator of $g$;
        \item $\mathcal{Q}_{g,k}$ is the total $2k$-th order $Q$-curvature functional of $g$;
        \item $\mathcal{M}_{g,k}$ is the set of minimizers for $\mathcal{Q}_{g,k}$;
        \item $G_{g,k,\xi}$ is the Green's function of $P_{g,k}$ with pole at $\xi$.
	\end{itemize}
    
    \section{Preliminaries} \label{sec:preliminaries}

    In this section we first establish the regularity of the total $Q$-curvature functional and 
    compute its first two derivatives. Then we list some auxiliary lemmas from other 
    papers which we will need. 

    \subsection{Regularity of the total $Q$-curvature functional}
    \label{sec:funct_regularity}

    We fix a background metric $g$ and recall that the 
    normalized total $Q$-curvature functional on the conformal 
    class $[g]$ is given by 
    $$
     \mathcal{Q}_{g,k}(u) =\frac{2}{n-2k} \frac{\int_M u P_{g,k} (u) \ud\mu_g}{\left ( 
    \int_M u^{\frac{2n}{n-2k}} \ud\mu_g \right )^{\frac{n-2k}{n}}} .
    $$
    Since $\mathcal{Q}_{g,k}(cu)=\mathcal{Q}_{g,k}(u)$ for any $c>0$, it will often be easier 
    to work with functions having $L^{\frac{2n}{n-2k}}$-norm equal to $1$. To that 
    end, we introduce the following Banach manifold
    $$\mathcal{B} = \left\{ u \in W^{k,2}_+ (M) : \int_M u^{\frac{2n}{n-2k}} \ud\mu_g 
    = 1 \right\}$$
    and observe that if $u \in \mathcal{B} \cap \mathcal{C}^\infty (M)$ then 
    $\widetilde g = u^{{4}/{(n-2k)}} g$ is a smooth metric in the conformal class 
    $[g]$ with unit volume. 
    \begin{lemma} 
    \label{lemma tangent space}
     Let $n,k\in\mathbb N$ with $n>2k$ and let $(M,g)$ be a smooth, closed, $n$-dimensional Riemannian manifold.
    The tangent space of $\mathcal{B}$ at $u$ is given by 
    $$T_u \mathcal{B} = \left \{ v \in W^{k,2}(M) : \int_M u^{\frac{n+2k}{n-2k}} v 
    \ud\mu_g = 0\right \}.$$
    Moreover, for each $v \in T_u \mathcal{B}$ the mappings 
    $$v \mapsto \pi_{T_u\mathcal{B}}: T_u\mathcal{B} \rightarrow \mathcal{L}
    (W^{k,2}(M), W^{k,2}(M))$$
    and 
    $$v \mapsto \pi_{T_u\mathcal{B}} : T_u\mathcal{B} \rightarrow \mathcal{L}
    (\mathcal{C}^{2k,\alpha}(M), \mathcal{C}^{2k,\alpha}(M))$$
    are both continuous. 
	\end{lemma} 

    \begin{proof} 
    By density, it suffices to take $v \in \mathcal{C}^\infty(M)$. 
    Indeed, $v \in T_u\mathcal{B}$ precisely when 
    $$0 = \left. \frac{\ud}{\ud\varepsilon} \right |_{\varepsilon = 0} 
    \int_M (u+\varepsilon v)^{\frac{2n}{n-2k}} \ud\mu_g = 
    \frac{2n}{n-2k} \int_M u^{\frac{n+2k}{n-2k}} v \ud\mu_g.$$
    The proof of continuity of $\pi_{\mathcal{B}}$ is the same as the one 
    at the end of the proof of \cite[Lemma 2.1]{ENS}. 
    \end{proof}
    
    \begin{lemma} \label{lem:functional_regularity}
     Let $n,k\in\mathbb N$ with $n>2k$ and let $(M,g)$ be a smooth, closed, $n$-dimensional Riemannian manifold.
    The mapping $u \mapsto \mathcal{Q}_{g,k}(u)$ is $\mathcal{C}^2$. If 
    $u \in  \mathcal{B}$ and $v,w \in T_u\mathcal{B}$ 
    then 
    \begin{eqnarray} \label{proj_derivatives} 
    D\mathcal{Q}_{g,k}(u)(v) & = & \frac{2}{n-2k} \int_M 
    [u P_{g,k}(v) + v P_{g,k} (u)] \ud\mu_g \\ \nonumber 
    D^2 \mathcal{Q}_{g,k}(u) (v,w) & = & \frac{2}{n-2k} \int_M [v P_{g,k} (w) + w P_{g,k}(v)]
    \ud\mu_g - 2 \left ( \frac{n+2k}{n-2k} \right ) \mathcal{Q}_{g,k} (u) \int_M 
    u^{\frac{4k}{n-2k}} vw \ud\mu_g .
    \end{eqnarray} 
    \end{lemma} 

    \begin{proof} 
    We begin with the expansion 
    \begin{eqnarray} \label{vol_integral_expansion} 
    \left ( \int_M (u + \varepsilon v)^{\frac{2n}{n-2k}} \ud\mu_g \right )^{\frac{2k-n}{n}} & = & 
    \left ( \int_M u^{\frac{2n}{n-2k}} \ud\mu_g \right )^{\frac{2k-n}{n}} \\ \nonumber 
    && -2\varepsilon \left( 
    \int_M u^{\frac{2n}{n-2k}} \ud\mu_g \right )^{\frac{2k-2n}{n}} \int_M u^{\frac{n+2k}{n-2k}} v \ud\mu_g 
    \\ \nonumber 
    && + \varepsilon^2\left ( \frac{2n-2k}{n} \right )  
    \left ( \int_M u^{\frac{2n}{n-2k}} \ud\mu_g \right )^{\frac{2k-3n}{n}} 
    \left ( \int_M u^{\frac{n+2k}{n-2k}} v \ud\mu_g \right )^2 \\ \nonumber  
    && -\varepsilon^2 \left( \frac{n+2k}{n-2k} \right ) \left ( \int_M
    u^{\frac{2n}{n-2k}} \ud\mu_g \right )^{\frac{2k-2n}{n}} \int_M u^{\frac{4k}{n-2k}} v^2 \ud\mu_g + 
    \mathcal{O}(\varepsilon^3) .  
    \end{eqnarray} 

    Combining this expansion with 
    $$\int_M (u + \varepsilon v) P_{g,k}(u+ \varepsilon v) \ud\mu_g = \int_M u P_{g,k} (u) \ud\mu_g 
    + \varepsilon \int_M [u P_{g,k} (v) + v P_{g,k}(u)] \ud\mu_g + \varepsilon^2 \int_M
    v P_{g,k}(v) \ud\mu_g, $$
    we obtain 
    \begin{eqnarray*} 
    \mathcal{Q}_{g,k} (u+\varepsilon v)& = & \frac{2}{n-2k} \left ( \int_M (u+\varepsilon v)^{\frac{2n}{n-2k}} \ud\mu_g
    \right )^{\frac{2k-n}{n}} \int_M (u+\varepsilon v) P_{g,k}(u+\varepsilon v) \ud\mu_g \\ 
    & = & \mathcal{Q}_{g,k}(u) + \frac{2\varepsilon}{n-2k} \left ( \int_M u^{\frac{2n}{n-2k}} \ud\mu_g \right )^{\frac{2k-n}{n}}
    \int_M [v P_{g,k} (u) + u P_{g,k}(v)]\ud\mu_g \\ 
    && - \frac{4\varepsilon}{n-2k} 
    \left ( \int_M u^{\frac{2n}{n-2k}} \ud\mu_g \right )^{\frac{2k-2n}{n}}
    \int_M u^{\frac{n+2k}{n-2k}} v \ud\mu_g \int_M u P_{g,k}(u) \ud\mu_g \\ 
    & & - \frac{2\varepsilon^2}{n-2k} \left ( \frac{n+2k}{n-2k}\right ) 
    \left ( \int_M u^{\frac{2n}{n-2k}} \ud\mu_g
    \right )^{\frac{2k-2n}{n}} \int_M u^{\frac{4k}{n-2k}} v^2 \ud\mu_g \int_M u P_{g,k} (u) \ud\mu_g \\ 
    && + \frac{2\varepsilon^2(2n-2k)}{n-2k} \left ( \int_M u^{\frac{2n}{n-2k}} \ud\mu_g \right )^{\frac{2k-3n}{n}} \left (
    \int_M u^{\frac{n+2k}{n-2k}} v \ud\mu_g \right )^2 \int_M u P_{g,k} (u) \ud\mu_g \\ 
    && + \frac{2\varepsilon^2}{n-2k} \left ( \int_M u^{\frac{2k-n}{n}} \ud\mu_g \right )^{\frac{2k-n}{n}} 
    \int_M v P_{g,k} (v) \ud\mu_g \\ 
    && - \frac{4 \varepsilon^2}{n-2k} \left ( \int_M u^{\frac{2n}{n-2k}} 
    \ud\mu_g\right )^{\frac{2k-2n}{n}}
    \int_M u^{\frac{n+2k}{n-2k}} v \ud\mu_g \int_M [v P_{g,k} (u) + u P_{g,k} (v)]\ud\mu_g + \mathcal{O}(\varepsilon^3) \\ 
    & = & \mathcal{Q}_{g,k} (u) + \varepsilon D\mathcal{Q}_{g,k}(u)(v) + \frac{1}{2} 
    \varepsilon^2 D^2\mathcal{Q}_{g,k}(u)(v,v) + \mathcal{O}(\varepsilon^3), 
    \end{eqnarray*} 
    which implies $\mathcal{Q}_{g,k}$ is a $\mathcal{C}^2$ functional. 

    We can read off from this last expansion that 
    \begin{eqnarray} \label{full_first_derivative}
    D\mathcal{Q}_{g,k}(u) (v) & = & \frac{2}{n-2k} \left ( \int_M u^{\frac{2n}{n-2k}} 
    \ud\mu_g \right )^{\frac{2k-n}{n}} \left ( \int_M \left[v P_{g,k} (u) + u P_{g,k}(v) \right]\ud\mu_g \right ) \\ \nonumber 
    && - \frac{4}{n-2k}
    \left (\int_M u^{\frac{2n}{n-2k}} \ud\mu_g \right )^{\frac{2k-2n}{n}} \int_M u^{\frac{n+2k}{n-2k}} v 
    \ud\mu_g \int_M u P_{g,k} (u) \ud\mu_g \end{eqnarray} 
    and 
    \begin{eqnarray} \label{full_second_derivative} 
    &&\frac{n-2k}{2} D^2 \mathcal{Q}_{g,k} (u)(v,w)\\
    \nonumber
    & = & \int_M [v P_{g,k} (w) + w P_{g,k}(v)] \ud\mu_g \left ( 
    \int_M u^{\frac{2n}{n-2k}} \ud\mu_g \right )^{\frac{2k-n}{n}} \\ \nonumber 
    && - 2 \left ( \int_M u^{\frac{2n}{n-2k}} 
    \ud\mu_g \right )^{\frac{2k-2n}{n}} \\ \nonumber 
    && \quad \times \left ( \int_M u^{\frac{n+2k}{n-2k}} v\ud \mu_g \int_M [w P_{g,k}(u) + u P_{g,k}(w)] \ud\mu_g 
    + \int_M u^{\frac{n+2k}{n-2k}} w \ud\mu_g \int_M [vP_{g,k} (u) + u P_{g,k}(v)] \ud\mu_g \right ) \\ \nonumber 
    && - 2\left ( \frac{n+2k}{n-2k} \right ) \left ( \int_M u^{\frac{2n}{n-2k}} \ud\mu_g \right )^{\frac{2k-2n}{n}}
    \int_M u^{\frac{4k}{n-2k}} vw \ud\mu_g \int_M u P_{g,k} (u) \ud\mu_g  \\ \nonumber 
    && + \left ( \frac{2n-2k}{n} \right ) \left ( \int_M u^{\frac{2n}{n-2k}} \ud\mu_g \right )^{\frac{2k-3n}{n}} 
    \int_M u P_{g,k}(u) \ud\mu_g  \\ \nonumber 
    && \quad \times \left ( \left ( \int_M u^{\frac{n+2k}{n-2k}} (v+w) \ud\mu_g \right)^2 - 
    \left ( \int_M u^{\frac{n+2k}{n-2k}} v \ud\mu_g \right )^2 - \left ( \int_M u^{\frac{n+2k}{n-2k}} w 
    \ud\mu_g \right )^2 \right ) . 
    \end{eqnarray} 
    In this last expression, we have used the polarization identity 
    $$D^2 \mathcal{Q}_{g,k}(u) (v,w) = \frac{1}{2} \left ( D^2 \mathcal{Q}_{g,k} (u) 
    (v+w, v+w) - D^2 \mathcal{Q}_{g,k} (u)(v,v) - D^2 \mathcal{Q}_{g,k} (u)(w,w) \right ).$$
    
    Restricting to the case $u \in \mathcal{B}$ and $v,w \in T_u \mathcal{B}$, 
    we impose the constraints 
    \begin{equation} \label{unit_vol_constraint}
    \int_M u^{\frac{2n}{n-2k}} \ud\mu_g = 1 \quad {\rm and} \quad \int_M u^{\frac{n+2k}{n-2k}} v \ud\mu_g = 0 
    = \int_M u^{\frac{n+2k}{n-2k}} w \ud\mu_g,
    \end{equation} 
    we see that for $u \in \mathcal{B}$ and $v,w \in T_u\mathcal{B}$ the 
    expressions \eqref{full_first_derivative} and \eqref{full_second_derivative}
    reduce to 
    $$ D\mathcal{Q}_{g,k}(u)(v) = \frac{2}{n-2k} \int_M \left[v P_{g,k} (u) + u P_{g,k} (v)\right]
    \ud\mu_g 
    $$
    and 
    \begin{align*}
        D^2 \mathcal{Q}_{g,k}(u)(v,v) &= \frac{2}{n-2k} \int_M \left[v P_{g,k} (w) + w P_{g,k}(v)\right]
    \ud\mu_g \\
    & \quad - \frac{4}{n-2k} \left ( \frac{n+2k}{n-2k} \right ) \int_M u^{\frac{4k}{n-2k}} 
    vw \ud\mu_g  \int_M u P_{g,k} (u) \ud\mu_g.
    \end{align*}
    The first of these formulas is exactly the first derivative as listed in \eqref{proj_derivatives}. 
    We obtain the listed formula for the second derivative after using the identity 
    $$\mathcal{Q}_{g,k} (u) = \frac{2}{n-2k} \int_M u P_{g,k}(u) \ud\mu_g.$$
    \end{proof} 

    We obtain modulus of continuity estimates from the structure of the linearized 
    operator. For a given $u \in \mathcal{B}$ we consider the linearization of
    \eqref{const_q_pde} about $u$, which is the operator
    $$L_u = P_{g,k} - \left ( \frac{n+2k}{n-2k}\right ) \mathcal{Q}_{g,k}(u) 
    u^{\frac{4k}{n-2k}}.$$

    \begin{lemma} 
    Let $n,k\in\mathbb N$ with $n>2k$ and let $(M,g)$ be a smooth, closed, $n$-dimensional Riemannian manifold.
    The mappings 
    $$u \mapsto \frac{D^2(\mathcal{Q}_{g,k})(u)(v,w)}{\| v \|_{W^{k,2}(M)}
    \| w \|_{W^{k,2}(M)}}$$
    and 
    $$u \mapsto \frac{D^2(\mathcal{Q}_{g,k})(u)(v,\cdot)}
    {\| v\|_{\mathcal{C}^{2k,\alpha}(M)}}$$
    are both continuous with moduli of continuity that are uniformly 
    bounded with respect to $v$ and $w$. 
    \end{lemma}

    \begin{proof} 
    Let $u_0, u_1 \in \mathcal{B}$. Using the fundamental theorem of 
    calculus we see 
    \begin{eqnarray*} 
    L_{u_1}(v) - L_{u_0}(v) & = & - \frac{n+2k}{n-2k} \int_0^1
    D(\mathcal{Q}_g^2)(tu_1+(1-t)) (u_1 - u_0) ((1-t)u_0 + tu_1)^{\frac{4k}{n-2k}} v \ud t \\ 
    && - \frac{4k(n+2k)}{(n-2k)^2} \int_0^1 \mathcal{Q}_{g,k}((1-t)u_0+tu_1)
    ((1-t)u_0+tu_1)^{\frac{6k-n}{n-2k}} (u_1-u_0) v \ud t,
    \end{eqnarray*} 
    which we can in turn integrate to obtain the following estimate 
    $$\| (L_{u_1} - L_{u_0})(v)\|_{W^{k,2}(M)} \lesssim \| u_1 - u_0 \|_{W^{k,2}(M)} \| v \|_{W^{k,2}(M)},$$
    uniformly on $u$ and $v$. It follows from 
    the second formula in \eqref{proj_derivatives} that 
    $$u \mapsto \frac{D^2(\mathcal{Q}_{g,k})(u) (v,w)}{\| v \|_{W^{k,2}(M)} 
    \| w \|_{W^{k,2}(M)}}$$
    is a continuous map whose modulus of continuity is uniform over $v,w \in W^{k,2}(M)$. 
    One can similarly show 
    $$u \mapsto \frac{D^2(\mathcal{Q}_{g,k})(u)(v,\cdot)}{\|v\|_{\mathcal{C}^{2k,\alpha}(M)}} 
    \in \mathcal{C}^{2k,\alpha}(M)$$
    is also a continuous map. \end{proof}

    \subsection{Some auxiliary results} \label{sec:auxiliary} 

    In this section, we present some auxiliary results that will be used in the proofs of our main results.
    
    First, we establish the variational setting we will use for the 
    remainder of our analysis. This is a version of the Lyapunov-Schmidt 
    reduction, which appears in many applications. Heuristically, it divides a 
    variation of our functional into a component that changes the functional by 
    an amount we can estimate and an orthogonal component that lies in a finite-dimensional 
    space. 

    Let $(M,g)$ be a closed Riemannian manifold with $n>2k$ and $k\in\mathbb N$ such that $g \in {\rm Met}^3(M)$.
    Let $u \in \mathcal{M}_{g,k}^* = \mathcal{B} \cap \mathcal{M}_{g,k}$ correspond to a minimizing metric 
    in the conformal class $[g]$ with unit volume and let 
    $$K = \ker (D^2\mathcal{Q}_{g,k}(u) (\cdot, \cdot)) = \{ v \in W^{k,2}(M) \cap 
    T_u\mathcal{B} : D^2\mathcal{Q}_{g,k}(u)(v,v) = 0 \}. $$
    This second variation operator is elliptic, so $K$ is finite-dimensional (see, 
    for instance, \cite[Section 10.4]{Nic}). We let $\ell = \dim(K)$ and denote by $K^\perp$ the 
    orthogonal complement of $K$ inside $W^{k,2}(M)$ with respect to the $L^2$ 
    inner product. 

    One can find proof of the following lemma \cite[Appendix~A]{ENS}. 

     \begin{lemmaletter}\label{lm:lyapunov-schmidt}
        Let $n,k\in\mathbb N$ with $n>2k$ and let $(M,g)$ be a smooth, closed, $n$-dimensional Riemannian manifold.  Assume 
        that $u\in \mathcal{CQC}_{g,k}^*$. Then there 
        is an open neighborhood $U \subset K$ of 0 in $K$ and a map $F: U \rightarrow K^{\perp}$ with 
        $F(0)=0$ and $\nabla F(0)=0$ satisfying the following properties:
        \begin{itemize}
            \item[{\rm (i)}] Firstly 
            $$
          \mathcal L := \left \{  u+\varphi+F(\varphi) \, : \varphi \in U \right\}  \subset \mathcal{B}.
            $$
            \item [{\rm (ii)}] If we define $q: U \rightarrow \mathbb{R}$ by 
            $q(\varphi)=\mathcal{Q}_{g,k}(u+\varphi+F(\varphi))$ then we have
            $$
            \begin{aligned}
            \nabla_{\mathcal{B}} \mathcal Q_{k,g}(v+\varphi+F(\varphi)) & =\pi_K \nabla_{\mathcal{B}} \mathcal Q_{g,k}(v+\varphi+F(\varphi))=\nabla q(\varphi) .
            \end{aligned}
            $$
            Furthermore, $\varphi \mapsto q(\varphi)$ is real analytic.
            \item[{\rm (iii)}] There exists $\delta>0$ depending on $u$ such that for 
            any $\widetilde u  \in \mathcal{B}$ with $\| u- \widetilde u\|_{W^{k,2}(M)} \leq \delta$ 
            we have $\pi_K(u-\widetilde u) \in U$. Furthermore, 
            if $\widetilde u \in \mathcal{CQC}_{g,k}^* = \mathcal{CQC}_{g,k} \cap \mathcal{B}$ with 
            $\| u-\widetilde u\|_{W^{k,2}(M)} \leq \delta$ then
            $$
            \widetilde u=u+\pi_K(\widetilde u-u)+F\left(\pi_K(\widetilde u-u)\right) .
            $$
            \item[{\rm (iv)}] For all $\varphi \in U$ and $\eta \in K$, we have
            $$
            \|\nabla F(\varphi)[\eta]\|_{\mathcal C^{2k, \alpha}(M)} \lesssim 
            \|\eta\|_{\mathcal C^{0, \alpha}(M)} .
            $$
        \end{itemize}
    \end{lemmaletter}

    \begin{proof}
        See \cite[Appendix~A]{ENS}
    \end{proof}

    Second, we need the following compactness result for minimizing sequences.

    \begin{lemmaletter}
    \label{lemma compactness minseq}
        Let $n,k\in\mathbb N$ with $n>2k$ and let $(M,g)$ be a smooth, closed, $n$-dimensional Riemannian manifold.
        If $g\in\mathfrak{A}_k$ is admissible and $\{u_m\}_{m\in\mathbb N}\subset \mathcal{B}$ is a sequence such 
        that $\lim_{m\rightarrow\infty}\mathcal Q_{g,k}(u_m) =\mathcal{Y}_{k,+}(M,[g])$, 
        then there exists $v \in \mathcal{M}_{g,k} \cap \mathcal{B}$ such 
        that $ \lim_{m\rightarrow\infty}\|u_m-v\|_{W^{k,2}(M)}=0$ up to a subsequence.
    \end{lemmaletter}

    \begin{proof}
For $k = 1$, this is a classical result due to Aubin \cite{Aub}. For $k \geq 1$, the lemma follows 
from \cite[Theorem 3]{Mazumdar2016} and its proof is in the spirit of Lions' concentration-compactness 
\cite[Theorem 4.1]{Lions-limitII}. For the analogous statement in a dual formulation when $k =2$, see 
also \cite[Proposition 2.6]{HY}.
    \end{proof}

    We remark that the statement of \cite[Theorem 3]{Mazumdar2016} allows for singular limits, but the 
    additional hypotheses we've listed here preclude the development of singularities. 

   Third, we introduce the so-called finite-dimensional "distance" Łojasiewicz inequality.

    \begin{lemmaletter}
    \label{lemma lojasiewicz}
        Let $q: \mathbb{R}^{\ell} \rightarrow \mathbb{R}$ with $\ell\in\mathbb N$ be a 
        real analytic function and assume that $\nabla q(\varphi_0)=0$. 
        There exist $\tilde{\delta}>0$ and $\gamma \geq 0$, depending on $q$ and 
        $\varphi_0$ such that 
        $$
        |q(\varphi) - q(\varphi_0)| \gtrsim  \inf \{|\varphi-\bar{\varphi}|: 
        \bar{\varphi} \in B(\varphi_0, \tilde{\delta}) \; {\rm and} 
        \;
         q(\bar{\varphi})= q (\varphi_0)\}^{2+\gamma} \quad {\rm for \ all} 
        \quad \varphi \in B(\varphi_0, \tilde{\delta}).
        $$
    \end{lemmaletter}

     \begin{proof}
        See \cite[Théorème 2]{Loja1965}.
    \end{proof}   

    Finally, we state a generalized version by Henry \cite{MR2160744} of Smale's version of Sard's theorem \cite{MR185604}, which we use to 
    prove the genericity part of our main results. 

    \begin{lemmaletter}\label{lm:henry}
        Let $X,Y$ and $Z$ be Banach spaces, let $U \subset X$ and $V \subset Y$ be 
        open subsets and let $\mathcal{F}: V \times U \rightarrow Z$ be a map of class $\mathcal{C}^1$ 
        with $z_0 \in {\rm Im}(\mathcal{F})$. Suppose that 
        \begin{itemize}
        \item[{\rm (i)}] For each $y \in V$ the map $x \mapsto \mathcal{F}(y,x)=:\mathcal{F}_y(x)$ is Fredholm of index $\ell < 1$, {\it i.e.} we have $\partial_x \mathcal{F}_y:X \rightarrow Z$ is Fredholm of 
        index $\ell < 1$ for 
        any $x \in U$;
        \item[{\rm (ii)}]  $z_0$ is a regular value of $\mathcal{F}$; 
        \item[{\rm (iii)}]  Let $\iota: \mathcal{F}^{-1} (z_0) \rightarrow Y \times X$ be the canonical embedding and let 
        $\pi_1 : Y \times X \rightarrow Y$ be projection onto the first factor. Then, $\pi_1 \circ \iota$ 
        is $\sigma$-proper, {\it i.e.}, $ \mathcal{F}^{-1}(z_{0})
       =\bigcup_{j=1}^{\infty} C_{j}$, where $C_{j}$ is a closed subset of $ \mathcal{F}^{-1}(z_{0})$ and 
       $\left . \pi_{1} \circ \iota\right |_{C_{j}}$ is proper for all $j\in\mathbb{N}$.
        \end{itemize}
       Then, the set $\{ y \in V : z_0 \mbox{ is a regular value of }\mathcal{F}(y, \cdot) \}$ is 
       an open and dense subset of $V$.
    \end{lemmaletter}

     \begin{proof}
        See \cite[Theorem~5.4]{MR2160744}
    \end{proof}

    \section{Generic stability estimates (proofs of Theorem~\ref{quant_stability_thm} and Corollary~\ref{quant_stability_metrics_thm})} \label{sec:generic_stable} 
  
    We prove Theorem~\ref{quant_stability_thm} in stages. 
Like in \cite{BE, ENS} the main point is to prove a local version of the quantitative stability 
(Proposition \ref{proposition local bound} below), which can then be  extended to the
    global case. Finally we prove the genericity statement, the arguments for which go beyond those in \cite{ENS}. 

    \subsection{Local stability estimate}
    In this section, we prove that if $u$ lies in a small neighborhood of the minimizing 
    set $\mathcal{M}_{g,k}^*$ then it satisfies \eqref{stabilityestimate}. As a first step, 
    we introduce the notion of integrability and prove that the function $q$ defined in 
    Lemma \ref{lm:lyapunov-schmidt} is constant in the integrable case, which we will see  
    implies stability. 
    
    \begin{definition}
        Let $k\in\mathbb N$ and let $(M,g)$ be a smooth, closed, $n$-dimensional Riemannian manifold with $n>2k$.
        A function $v \in \mathcal{CQC}_{g,k}^*$ is said to be integrable 
        if for all $\varphi \in K$ 
        there exists a one-parameter family of functions $\{v_t\}_{t \in(-\delta, \delta)}$, 
        with $v_0=v$, $\left.{\partial_t}\right|_{t=0} v_t=\varphi$, and 
        $v_t \in \mathcal{CQC}_{g,k}^*$ for all $0<t\ll1$ sufficiently small.
    \end{definition}

    With this definition in mind, we have the following auxiliary result.

    \begin{lemma}
    \label{lemma v integrable q constant}
         Let $n,k\in\mathbb N$ with $n>2k$ and let $(M,g)$ be a smooth, closed, $n$-dimensional Riemannian manifold.
        If $v \in \mathcal{M}_{g,k}^*$, then $v$ is integrable if and only if $q$ is constant in a 
        neighborhood of $0 \in K$. In particular, if $v \in \mathcal{M}_{g,k}^*$ is an integrable minimizer, 
        then there is $\delta > 0$ such that 
        \begin{equation}
            \label{M cap B = L lemma}
             \mathcal{M}_{g,k}^* \cap \mathcal{B}(v, \delta)=\mathcal{L} \cap \mathcal{B}(v, \delta),
        \end{equation}
        where $\mathcal{L}$ is as described in Lemma \ref{lm:lyapunov-schmidt} {\rm (i)}. 
    \end{lemma}

    \begin{proof}
    In this proof, we identify $K$ with $\R^\ell$. We may assume that $\ell \geq 1$, for otherwise there is nothing to show.
    
First, suppose that $q \equiv q(0)$ in a neighborhood $V \subset \R^{\ell}$ of $0$, and let $\varphi \in K$. 
Let $s \in \R$ be small enough so that $s\varphi \in U$ and  consider 
\[ v_s = v + s \varphi + F(s \varphi) \]
where $U$ and $F$ are as in Lemma \ref{lm:lyapunov-schmidt}. Then 
\[ \left. \frac{\mathrm d}{\mathrm d s}\right |_{s = 0} v_s = \varphi \]
because $\nabla F(0) = 0$ by Lemma \ref{lm:lyapunov-schmidt}. Still by that lemma, we have $v_s \in \mathcal B$ and 
\[ \nabla_{\mathcal B} \mathcal Q_{g,k}(v_s) = \nabla q(s \varphi) = 0 \]
for $s$ small enough, because $\nabla q \equiv 0$ on $V$ by assumption. Hence $v_s \in \mathcal {CQC}_{g,k}^*$. 
Since $\varphi \in K$ was arbitrary, $v$ is integrable. 

Conversely, assume now that $v$ is integrable. By contradiction, suppose that $q$ is non-constant on every 
neighborhood of $0$ in $\R^{\ell}$. Since $q$ is analytic, we then have, for $\varphi$ in some neighborhood $V$ 
of $0$, 
     \[ q(\varphi) = q(0) +  q_{k_0}(\varphi) + q_R(\varphi). \]
     Here $q_{k_0}$ is a non-trivial homogeneous polynomial of some degree $k_0 \in \mathbb N$ and the remainder 
     term $q_R$ is a sum of homogeneous polynomials of degree greater than $k_0$. 

     We fix a $\varphi \in U$ such that $\nabla q_{k_0}(\varphi) \neq 0$. Since $v$ is integrable, there 
     exists $(v_s)_{s \in (-\delta, \delta)} \subset \mathcal{CQC}_{g,k}^*$ such that $v_0 = v$ 
     and $\left. \frac{\mathrm d}{\mathrm d s}\right |_{s = 0} v_s = \varphi$ for 
     all $s \in (-\delta, \delta)$. By 
     Lemma \ref{lm:lyapunov-schmidt}.(ii), after possibly choosing $\delta$ and $U$ to be smaller, there 
     are $\varphi_s \in U$ such that
     \begin{equation}
         \label{v s Q integrable proof}
         v_s = v + \varphi_s + F(\varphi_s) \quad \text{ for every }  \quad s \in (-\delta, \delta), 
     \end{equation} 
     where $F$ is the map from Lemma \ref{lm:lyapunov-schmidt}. By that lemma, we have 
     \begin{eqnarray}
         \label{Q integrable contradiction}
             0 & = & \nabla_{\mathcal B} \mathcal Q_{g,k}(v_s) = \nabla q(\varphi_s) = 
             \nabla q_{k_0}(\varphi_s) + \nabla q_R(\varphi_s) 
             \\ \nonumber 
             & = &  |\varphi_s|^{k_0 -1} \left( \nabla q_{k_0} \left( \frac{\varphi_s}{|\varphi_s|}  \right) 
             + |\varphi_s|^{-k_0+1} \nabla q_{R} (\varphi_s) \right).
     \end{eqnarray}
          Here we used that the vector $\nabla q_{k_0}(\varphi)$ consists of homogeneous polynomials of 
          degree $k_0-1$. Moreover, since $q_R$ is a sum of homogeneous polynomials of degree strictly greater 
          than $k_0$, we have $|\varphi_s|^{-k_0+1} \nabla q_{R} (\varphi_s)$ as $s \to 0$. On the other hand, 
          $\left. \frac{\mathrm d}{\mathrm d s}\right |_{s = 0} v_s = \varphi$ and 
          \eqref{v s Q integrable proof} combined with $\nabla F(0) = 0$ imply 
          $\frac{\varphi_s}{|\varphi_s|} \to \frac{\varphi}{|\varphi|}$ as $s \to 0$.  Thus the sum inside the 
          parentheses on the right side of \eqref{Q integrable contradiction} is non-zero for $s$ small enough. 
          This contradiction finishes the proof of the converse implication. 

          It remains to prove that \eqref{M cap B = L lemma} holds provided $v$ is integrable. The 
          inclusion $\subset$ is given by Lemma \ref{lm:lyapunov-schmidt}.(ii), because 
          $\mathcal M_{g,k}^* \subset \mathcal{CQC}_{g,k}^*$. For the 
          reverse inclusion, we have already proved that $q$ is locally constant on $K$ near $0$ if $v \in 
          \mathcal{M}_{g,k}^*$ is integrable. By the definition of $q$, this is the same thing as saying that 
          $\mathcal{Q}_{g,k}$ is constant on $\mathcal L \cap \mathcal B(v, \delta)$ for some $\delta > 0$. Hence, 
          for $u \in \mathcal L \cap \mathcal B(v, \delta)$ it follows $\mathcal{Q}_{g,k}(u) = 
          \mathcal{Q}_{g,k}(v)$ and hence $u \in \mathcal{M}_{g,k}^*$. Thus \eqref{M cap B = L lemma} is proved. 
    \end{proof}

    We recall the notion of the Adams-Simon positivity condition as in 
    \cite{MR963501} (see also \cite{MR3352243}).

\begin{definition}\label{def:adams-simon}
    Let $k\in\mathbb N$ and let $(M,g)$ be a smooth, closed, $n$-dimensional Riemannian manifold with $n>2k$.
    Suppose that $u_0 \in \mathcal{CQC}_{g,k}^*$ is nonintegrable and that $q: U \rightarrow$ $\mathbb{R}$, 
    where $U \subset \operatorname{ker} \nabla_{\mathcal{B}}^2 \mathcal{Q}_{g,k}(u_0) \cong \mathbb{R}^{\ell}$, 
    is the function defined in Lemma~\ref{lm:lyapunov-schmidt}.
    Since $q$ is analytic, we can expand it in a power series
    $$
    q(v)=q(0)+\sum_{j \geq p} q_j(v),
    $$
    where each $q_j$ is a  homogeneous polynomial of degree $j$ and $p$ is chosen so that $q_p \not \equiv 0$. 
    We  call $p$ the order of integrability of $u_0$. 
    We say that $u_0$ satisfies the Adams-Simon positivity condition of order $p$, denoted by $\mathrm{AS}_p$, if $p$ is 
    the order of integrability of $u_0$ and $\left.q_p\right|_{\mathbb{S}^{\ell-1}}$ attains a positive maximum at 
    some $v \in \mathbb{S}^{\ell-1}$.
\end{definition}

    \begin{remark}
        When the order of integrability $p$ is odd, then $q_p \not \equiv 0$ is an odd function, and hence the 
        Adams--Simon positivity 
        condition is automatically satisfied in this case. Moreover, by arguing as in  \cite[Appendix A]{MR3352243} one 
        finds that the order of integrability 
        always satisfies $p \geqslant 3$, and that $q_3(v)$ can be explicitly expressed  
    \begin{equation}\label{AS3}
            q_3(v)=C_{n,k} \mathcal Q_{g,k}(u_0) \int_M v^3 \ud \mu_{g} 
        \end{equation}
        for some dimensional constant $C_{n,k}>0$. 
    \end{remark}

    We establish the local version of Theorem~\ref{quant_stability_thm}. 
    We need a localized measure of how far $u$ is from being a minimizer close to some given minimizer $v$.
    Given $\delta>0$ and $v \in \mathcal{M}_{g,k}^*$, we let
    \begin{equation}
        \label{dist delta}
           d_{v,\delta}\left(u, \mathcal{M}_{g,k}^*\right)=\frac{\inf \left\{\|u-\tilde{v}\|_{W^{k,2}(M)} : 
           \tilde{v} \in \mathcal{M}_{g,k}^* \cap \mathcal{B}(v, \delta)\right\}}{\|u\|_{W^{k,2}(M)}}.
    \end{equation}

    \begin{proposition}
    \label{proposition local bound}
       Let $n,k\in\mathbb N$ with $n>2k$ and let $(M,g)$ be a smooth, closed, $n$-dimensional Riemannian manifold.
        For any  $v \in \mathcal{M}_{g,k}^*$, there exist constants $c > 0$, $\gamma \geq 0$ and $\delta > 0$ 
        depending on $v$ such that
        $$
        \mathcal Q_{g,k}(u)-\mathcal{Y}_{k,+}(M,[g]) \geq c d_{v,\delta}\left(u, \mathcal{M}_{g,k}^*\right)^{2+\gamma} 
        \quad {\rm for \ all} \quad u \in \mathcal{B}(v, \delta) .
        $$
    \end{proposition}

    \begin{proof}
Let $v \in \mathcal M_{g,k}^*$. We will use the notation of Lemma \ref{lm:lyapunov-schmidt} without further comment 
throughout this proof. Also, we will abbreviate $\mathcal Y := \mathcal Y_{k,+}(M, [g])$ to make notation lighter. To start with, 
let $u \in \mathcal B(v, \delta)$. 

We divide the proof into some steps as follows:

\noindent \textit{Step 1. Decomposing $u$.}

We now decompose $u$ according to the Lyapunov-Schmidt reduction from Lemma \ref{lm:lyapunov-schmidt} by letting 
\[ u_{\mathcal L} := v + \pi_K (u-v) + F(\pi_K(u-v))\]
and setting $u^\perp := u - u_{\mathcal L}$. Note that, since $F$ maps into $K^\perp$, we have
\[ u^\perp = (u-v) - \pi_K (u-v) - F(\pi_K(u-v)) \in K^\perp.  \]

It will be convenient to write 
\begin{equation}
    \label{Q decomp proof}
    \mathcal Q_{g,k}(u) - \mathcal Y = (\mathcal Q_{g,k}(u) - \mathcal Q_{g,k}(u_{\mathcal L})) 
    + (\mathcal Q_{g,k}(u_{\mathcal L}) - \mathcal Y) 
\end{equation}
and bound the two brackets on the right side separately. 

\noindent \textit{Step 2. The non-degenerate and the integrable case. }

We first bound the first term of \eqref{Q decomp proof}. Here is where we crucially use the decomposition 
$u = u_{\mathcal L} + u^\perp$. Using Taylor's theorem with the mean-value formula for the 
remainder term, we can express
\begin{eqnarray}
    \label{taylor Q}
    \mathcal Q_{g,k}(u) - \mathcal Q_{g,k}(u_{\mathcal L}) & = & D_{\mathcal B} 
    \mathcal Q_{g,k}(u_{\mathcal L})(u^\perp) + 
    \frac{1}{2} D^2_{\mathcal B} \mathcal Q_{g,k}(\zeta)(u^\perp, u^\perp) \\ \nonumber 
    & = & 
    \frac{1}{2} D^2_{\mathcal B} \mathcal Q_{g,k}(v)(u^\perp, u^\perp)  + \mathrm{o}(1) \|u^\perp\|^2_{W^{k,2}(M)},  
\end{eqnarray} 
where $\zeta$ lies on a geodesic in $\mathcal B$ between $u$ and $u_{\mathcal L}$. For the second equality we 
used that $u^\perp \in K^\perp$ implies $D_{\mathcal B} \mathcal Q_{g,k}(u_{\mathcal L})
(u^\perp) = 0$ by Lemma \ref{lm:lyapunov-schmidt}, 
as well as the fact that $D^2_{\mathcal B} \mathcal Q_{g,k}(\cdot)$ is continuous. 

To estimate the right-hand-side of \eqref{taylor Q} we use the fact that $v$ minimizes 
$\mathcal{Q}_{g,k}$, and so $D^2 \mathcal{Q}_{g,k}(v) \geq 0$. Since $u^\perp \in K^\perp$, we 
have the lower bound 
\begin{equation}
    \label{spec-prelim}
    D^2_{\mathcal B} \mathcal Q_{g,k}(v)[u^\perp, u^\perp] \geq \lambda_1 \|u^\perp\|_{L^2(M)} 
    \geq c \int_M v^{\frac{4k}{n-2k} }(u^\perp)^2 \ud \mu_g,
\end{equation} 
where $\lambda_1 > 0$ is the smallest positive eigenvalue of $D^2_{\mathcal B} 
\mathcal Q_{g,k}(v)$. Since $v >0$ on 
the compact manifold $M$, we may choose $c = \lambda_1 (\max_M v)^{-{4k}/{(n-2k)}}$ in the second 
inequality. Recalling from Lemma \ref{lem:functional_regularity} that 
\[D^2 \mathcal Q_{g,k}(v)(w,w) = \frac{4}{n-2k} \left( \int_M w P_{g,k} (w)\ud \mu_g + \frac{n+2k}{n-2k} \mathcal Y 
\int_M v^\frac{4k}{n-2k} w^2\ud\mu_g  \right)
\]
 some elementary manipulations give us 
\[ 
D^2_{\mathcal B} \mathcal Q_{g,k}(v)[u^\perp, u^\perp] \geq  c_1 \int_M u^\perp P_{g,k}
(u^\perp)\ud\mu_g  \geq c_2 \|u\|_{W^{k,2}(M)}^2 ,
 \]
where $c_1 = {c}\left(\frac{n+2k}{n-2k}\mathcal Y - c\right)^{-1}>0$, $c>0$ is given in \eqref{spec-prelim}, 
and $c_2 > 0$ exists since $P_{g,k}$ is coercive. 

Inserting this last expression into \eqref{taylor Q}, we obtain
\begin{equation}
    \label{Q spectral bound}
    \mathcal Q_{g,k}(u) - \mathcal Q_{g,k}(u_{\mathcal L}) \geq \frac{1}{4} c_2 \|u^\perp\|_{W^{k,2}(M)}^2
\end{equation} 
whenever $0<\|u^\perp\|_{W^{k,2}(M)}\ll1$ is sufficiently small. 

On the other hand, the second term in \eqref{Q decomp proof} trivially satisfies 
\begin{equation}
    \label{Q - Y geq 0}
    \mathcal Q_{g,k}(u_{\mathcal L}) - \mathcal Y \geq 0
\end{equation}
by the definition of $\mathcal Y$. 

If $K =0$ (i.e., if $v$ is non-degenerate), then it follows directly from the definitions 
that $u_{\mathcal L} = v$, and hence
\begin{equation}
    \label{Q final bound K=0}
    \|u^\perp\|_{W^{k,2}(M)}^2 =  \|u-v\|_{W^{k,2}(M)}^2 \geq \inf \{  \|u - \Tilde{v}\|_{W^{k,2}(M)}^2  \, 
    : \, \Tilde{v} \in \mathcal M_{g,k}^* \cap \mathcal B(v, \delta)\} \geq c d_{v, \delta}
    (u, \mathcal M_{g,k}^*)^2,
\end{equation}  
where in the last inequality, we used 
\[\|u\|_{W^{k,2}(M)}^2 \geq \frac{1}{2} \|v\|_{W^{k,2}(M)}^2 
\gtrsim \int_M v P_{g,k} (v)\ud\mu_g = \mathcal Y > 0.\]
Thus the assertion in case $K = 0$ follows by putting together \eqref{Q decomp proof}, 
\eqref{Q spectral bound}, \eqref{Q - Y geq 0} 
and \eqref{Q final bound K=0}. 

More generally, if $v$ is integrable, then by Lemma \ref{lemma v integrable q constant} we 
have $u_{\mathcal L} \in \mathcal M_{g,k}^* 
\cap \mathcal B(v, \delta)$ if $u$ is close enough to $v$. Thus, similarly to the above,
\begin{equation}
    \label{Q final bound v integrable}
    \|u^\perp\|_{W^{k,2}(M)}^2 =  \|u-u_{\mathcal L}\|_{W^{k,2}(M)}^2 \geq \inf 
    \{  \|u - \Tilde{v}\|_{W^{k,2}(M)}^2  \, : \, \Tilde{v} 
    \in \mathcal M_{g,k}^* \cap \mathcal B(v, \delta)\} \geq c d_{v, \delta}(u, \mathcal M_{g,k}^*)^2. 
\end{equation} 
Thus the assertion in case $v$ is integrable follows by putting 
together \eqref{Q decomp proof}, \eqref{Q spectral bound}, 
\eqref{Q - Y geq 0} and \eqref{Q final bound v integrable}.

We emphasize that in both of these cases ($v$ non-degenerate or $v$ integrable) the 
inequality \eqref{stabilityestimate}
restricted to $\mathcal{B}(v,\delta)$ holds with $\gamma = 0$. 

\noindent\textit{Step 3. The non-integrable case. }

It remains to discuss the third and hardest case, namely the case when $v$ is non-integrable. 
The complication is that
the estimates \eqref{Q final bound K=0} resp. \eqref{Q final bound v integrable} are not applicable because $u_{\mathcal L}$ is not necessarily in $\mathcal M_{g,k}^\ast$. Consequently 
we require a better lower bound on $\mathcal{Q}_{g,k}(u_{\mathcal{L}}) - \mathcal{Y}$. To obtain it, we 
invoke the Łojasiewicz 
inequality from Lemma \ref{lemma lojasiewicz}, which applies because $q$ is analytic by 
Lemma \ref{lm:lyapunov-schmidt}. 
We conclude that there exist $\gamma > 0$ such that  
\begin{equation}
    \label{Q estimate nonintegrable}
     \mathcal Q_{g,k}(u_{\mathcal L}) - \mathcal Y = q(\varphi) - q(0) \gtrsim \inf \{|\varphi-\bar{\varphi}|: 
        \bar{\varphi} \in K\cap  B(0, \delta) \; {\rm and} \;
        q(\bar{\varphi})=0\}^{2+\gamma}, 
\end{equation}
where $\varphi = \pi_K(u-v)$. (Note that we could drop the absolute value of $|q(\varphi) - q(0)|$ because $0$ is a 
local minimum of $q$.) To turn this into the desired bound, we observe that for 
any $\bar\varphi \in K \cap B(0, \delta)$ with 
$q(\bar \varphi) = q(0)$, the function $\bar v = v + \bar \varphi + F(\bar \varphi)$ satisfies 
$\mathcal Q_{g,k}(\bar v) =  q(\bar{\varphi}) = 0$ by 
Lemma \ref{lm:lyapunov-schmidt}. Hence, 
$\bar{v} \in \mathcal M_{g,k}^*$. 
  Moreover, for such $\bar{v}$, we have (still writing $\varphi = \pi_K(u-v)$) 
\begin{align*}
    \| u_{\mathcal L} - \bar{v}\|_{W^{k,2}(M)} &= \| \varphi + F(\varphi) - \bar{\varphi} 
    - F(\bar{\varphi}) \|_{W^{k,2}(M)} \\
    &\leq \| \varphi - \bar{\varphi} \|_{W^{k,2}(M)} + \|  F(\varphi) - F(\bar{\varphi}) \|_{W^{k,2}(M)} \\
    &\lesssim \| \varphi - \bar{\varphi} \|_{W^{k,2}(M)} +  \| \varphi - \bar{\varphi} \|_{\mathcal{C}^{0,\alpha}(M)} \\
    &\lesssim  \| \varphi - \bar{\varphi} \|_{W^{k,2}(M)}.     
\end{align*}
Here, the $\mathcal{C}^{0,\alpha}$ estimate follows from the corresponding one in 
Lemma \ref{lm:lyapunov-schmidt}. The last inequality simply 
uses the equivalence of any two norms on the finite-dimensional space $K$. 
From this chain of inequalities we get 
\[  \inf \{|\varphi-\bar{\varphi}|: 
        \bar{\varphi} \in K\cap  B(0, \delta), 
        \nabla q(\bar{\varphi})=0\}^{2+\gamma} \gtrsim \inf \left\{\|u_{\mathcal L}-\tilde{v}\|_{W^{k,2}(M)} : 
        \tilde{v} \in \mathcal{M}^*_{g,k} \cap \mathcal{B}(v, \delta)\right\}^{2 + \gamma} \]

Together with \eqref{Q estimate nonintegrable} and \eqref{Q spectral bound}, since 
$u^\perp = u - u_{\mathcal L}$ satisfies $\|u^\perp\|_{W^{k,2}} \leq 1$, we thus get
\begin{align*}
    \mathcal Q_{g,k}(u) - \mathcal Y &\gtrsim \frac{1}{4}\|u - u_{\mathcal L}\|_{W^{k,2}(M)}^2 
    + \inf \left\{\|u_{\mathcal L}-\tilde{v}\|_{W^{k,2}(M)} : \tilde{v} \in \mathcal{M}^*_{g,k} 
    \cap \mathcal{B}(v, \delta)\right\}^{2 + \gamma} \\
    &  \gtrsim \inf \left\{ \lambda_1 \|u - u_{\mathcal L}\|_{W^{k,2}(M)}^{2 + \gamma}+ 
    \|u_{\mathcal L}-\tilde{v}\|^{2 + \gamma}_{W^{k,2}(M)} : \tilde{v} \in \mathcal{M}^*_{g,k} 
    \cap \mathcal{B}(v, \delta)\right\} \\
    & \gtrsim  \inf \left\{  \|u - \tilde{v}\|_{W^{k,2}(M)}^{2+\gamma} : \tilde{v} 
    \in \mathcal{M}^*_{g,k} \cap \mathcal{B}(v, \delta)\right\}
\end{align*}
and the proof is complete. 
 \end{proof}

    If $v$ is integrable or non-degenerate ({\it i.e.} the kernel $K = \ker(D^2_{\mathcal{B}} 
    \mathcal{Q}_{g,k} (v)(\cdot, \cdot))$
    is trivial) then we may take $\gamma=0$.
    
    \subsection{Global stability estimate}

In the following, we again abbreviate $\mathcal Y = \mathcal Y_{k,+}(M, [g])$.
    
    \begin{proof}[Proof of Theorem~\ref{quant_stability_thm}]
Since both sides of inequality \eqref{stabilityestimate} are zero-homogeneous, it suffices 
to prove \eqref{stabilityestimate} for 
every $u \in \mathcal B$. By contradiction, suppose that for every $\eta \geq 0$ there exists 
a sequence $\{u_m\}_{m\in\mathbb N} \subset \mathcal B$ such 
that 
\begin{equation}
    \label{BE contradiction ass}
    \frac{\mathcal Q_{g,k}(u_m) - \mathcal Y}{d(u_m, \mathcal{M}_{g,k})^{2 +\eta}} \to 0 
    \qquad \text{ as } \quad m \to \infty. 
\end{equation}
By diagonal extraction of a subsequence, we may actually assume that \eqref{BE contradiction ass} 
holds for every $\eta \geq 0$ with 
the \emph{same} sequence $(u_m)$.

Since $d(u_m, \mathcal{M}_{g,k}) \leq 1$, \eqref{BE contradiction ass} implies that 
$\mathcal Q_{g,k}(u_m) - \mathcal Y \to 0$ as $m \to \infty$. 
By Lemma \ref{lemma compactness minseq}, there exists $v \in \mathcal{M}_{g,k}^*$ such that, 
up to extracting a further subsequence, one has $u_m \to v$ strongly in $W^{k,2}(M)$. 

Now, let $\delta = \delta(v) > 0$ and $\gamma = \gamma(v) \geq 0$ be as in 
Proposition \ref{proposition local bound}. Since 
$\mathcal{M}_{g,k}^* \cap \mathcal B(v, \delta) \subset \mathcal{M}_{g,k}$, the 
definitions \eqref{dist} and \eqref{dist delta} 
of $d$ and $d_{v, \delta}$ directly imply $d_{v, \delta}(u_m, \mathcal{M}_{g,k}^*) 
\geq d(u_m, \mathcal{M}_{g,k})$. Hence,  
by \eqref{BE contradiction ass} applied with $\eta = \gamma$, it follows
\[ \frac{\mathcal Q_{g,k}(u_m) - \mathcal Y}{d_{v, \delta}(u_m, 
\mathcal{M}_{g,k}^*)^{2 + \gamma}} \leq \frac{\mathcal Q_{g,k}(u_m) - 
\mathcal Y}{d(u_m, \mathcal{M}_{g,k})^{2 + \gamma}} \to 0 \quad \text{ as } \quad m \to \infty. \]
But since $u_m \in \mathcal B(v, \delta)$ for every $m\gg1$ large enough, this 
contradicts Proposition \ref{proposition local bound}. 
Hence \eqref{BE contradiction ass} is false, and the theorem is proven. 
    \end{proof}

    \begin{proof}[Proof of Corollary~\ref{quant_stability_metrics_thm}]
   By Theorem \ref{quant_stability_thm} and the definition \eqref{norm 2*} of $\|\cdot\|$ we have 
   \begin{eqnarray*}
      \mathcal Q_{g,k}(u) - \mathcal Y & \geq&  c \left( \frac{\inf_{v \in \mathcal{M}_{g,k}} 
      \|u - v\|_{W^{k,2}(M)}}{\|u\|_{W^{k,2}(M)}} 
      \right)^{2 + \gamma}\\
      &\geq&\Tilde{c}\left( \frac{\inf_{v \in \mathcal M_{g,k}} \|u - v\|_{L^{\frac{2n}{n-2k}}(M)}}
      {\|u\|_{L^{2_k^*}(M)}} \right)^{2 + \gamma} \\
   &=& \left(\frac{\inf \{\|h-\tilde{g}\|: \tilde{g} \in \mathcal{M}_{g,k}\}}
        {\operatorname{vol}_h(M)^{\frac{n-2k}{2n}}}\right)^{2+\gamma}.    
   \end{eqnarray*} 
   For the second inequality, note that if $d(u, \mathcal{M}_{g,k}) \leq \delta_0$ (for some 
   appropriately small $\delta_0 > 0$), 
   then $\|u\|_{W^{k,2}(M)}\simeq\|u\|_{L^{{2_k^*}(M)}}$, 
   so the inequality follows from Sobolev's 
   inequality. If  $d(u, \mathcal{M}_{g,k}) > \delta_0$ on the other hand, it suffices to take 
   $\Tilde{c}$ sufficiently small because 
   \[\inf_{v \in \mathcal{M}_{g,k}} \|u - v\|_{L^{2_k^*}(M)} \|u\|_{L^{2_k^*}(M)}^{-1} \leq 1.\]  

   Similarly, by Theorem \ref{quant_stability_thm} and the definition of $\|\cdot \|_*$ in 
   \eqref{norm gjms} we have 
 \begin{eqnarray*}
       \mathcal{Q}_{g,k}(u)  - \mathcal Y &\geq&  c \left( \frac{\inf_{v \in \mathcal{M}_{g,k}} 
       \|u - v\|_{W^{k,2}(M)}}{\|u\|_{W^{k,2}(M)}} 
       \right)^{2 + \gamma}\\
       &\gtrsim&\left( \frac{\inf_{v \in \mathcal{M}_{g,k}} (\int_M (u-v) P_{g,k} (u-v) \, 
       \ud \mu_g)^{1/2}}{\|u\|_{L^{2_k^*}(M)}} \right)^{2 + \gamma} \\
   &=&  \left(\frac{\inf \{\|h-\tilde{g}\|_*: \tilde{g} \in \mathcal{M}_{g,k}\}}
        {\operatorname{vol}_h(M)^{\frac{n-2k}{2n}}}\right)^{2+\gamma}.    
   \end{eqnarray*} 
   Here we have used $\int_M (u-v) P_{g,k} (u-v)\ud\mu_g \lesssim\|u-v\|_{W^{k,2}(M)}^2$, which holds 
   because $P_{g,k}$ is a 
   $k$-th order elliptic differential operator. Moreover, since by assumption $\mathcal Q_{g,k}(u) 
   - \mathcal Y \leq \delta_0$, 
   Theorem \ref{quant_stability_thm} yields 
   \[d(u, \mathcal{M}_{g,k}) \lesssim\delta_0^\frac{1}{2 + \gamma}.\]
   Thus, for $0<\delta_0\ll1$, one has
   $\|u\|_{W^{k,2}(M)}\simeq \|u\|_{L^{2_k^*}(M)},$
   which proves that the 
   inequality above holds.
    \end{proof}

     \subsection{Generic nondegeneracy}
    Now, let us prove the genericity part of our main result.
    Following ideas in \cite{MR2560131,MR2982783,MR4314216,MR2401291}, our strategy is to verify 
    hypotheses (i), (ii), and (iii) of the abstract transversality result in Lemma~\ref{lm:henry}.
    To apply Lemma~\ref{lm:henry} we must examine the second variation of the functional 
    $\mathcal{Q}_{g,k}$ among \emph{all} admissible metrics, not just those conformal to a 
    fixed metric $g_0$. To this end, we recall the machinery developed by Case, Lin and Yuan
    in \cite{CaseLinYuan}. In their language, the $k$-th-order $Q$-curvature is an example of 
    a conformal variational Riemannian invariant of weight $-2k$. 
    
    \begin{lemma}\label{lm:genericity}
        Let $n \in \mathbb{N}$ with $n>2k$. Let $M$ be a smooth, closed, $n$-dimensional 
        manifold. 
        The set 
        \begin{equation*}
            \mathcal{G}:=\{g\in \mathfrak{A}_{k,\alpha}(M) : \gamma(g)=0\}\subset \mathfrak{A}_{k,\alpha}(M)
        \end{equation*}
        is open and dense with respect to the $\mathcal{C}^{k,\alpha}-$topology for some $\alpha\in (0,1)$.
    \end{lemma}

    \begin{proof}
        In Step 2 of the proof of Proposition \ref{proposition local bound}, we have seen 
        that $\gamma(g) = 0$ if 
        $\ker D^2 \mathcal Q_{g,k}(v) = \{0\}$ for every $v \in \mathcal M^*_{g,k}$.  Thus it suffices to show that the (potentially) smaller set 
        \[ \{ g \in \mathfrak A_{k, \alpha}(M) \, : \ker D^2 \mathcal Q_{g,k}(v) = \{0\} \text{ for every } v \in \mathcal M^*_{g,k}  \}  \subset \mathcal G 
        \]
        is open and dense. 

        We recall that $G_{g,k,\xi}$ is the Green's function of the GJMS operator 
        $P_{g,k}$ with its pole at $\xi \in M$, and that $G_{g,k,\xi} > 0$ if $g$ is 
        an admissible metric. Now we fix $g_0\in \mathfrak{A}_{k,\alpha}(M)$ and define 
        $$\mathcal{F}: \mathfrak{A}_{k,\alpha} (M) \times W^{k,2}_{g_0}(M) \rightarrow W^{k,2}_{g_0}(M),$$
        by
         $$\mathcal{F} (g,u) (\xi) = u(\xi) - \int_M G_{g,k,\xi} (y) f_{n,k} (u(y)) \ud\mu_g(y),$$
        where $\displaystyle f_{n,k}(u) = c_{n,k} u^{\frac{n+2k}{n-2k}}$. Observe that 
        the zero set of $\mathcal{F}$ consists of the pairs $(g,u)$ such that 
        $\displaystyle \widetilde g = u^{{4}/({n-2k})} g$ has constant $k$-th-order $Q$-curvature, 
        {\it i.e.} $u$ solves the PDE 
        \begin{flalign}\label{zeropde}\tag{$\mathcal{P}_{g,k}$}
            P_{g,k}(u)-f_{n,k}(u)=0 \quad {\rm on} \quad M .
        \end{flalign}
        
        These solutions are also critical points of the functional 
        \begin{align*}
            \mathcal{E}_{g,k}(u)&= \frac12 \int_{M}uP_{g,k}(u)\ud\mu_{g}-F_{n,k}(u),
        \end{align*}
        where $F_{n,k}(u)=\int_{0}^{u}f_{n,k}(s)\ud s$.
        Furthermore, the set of metrics such that $\ker D^2 \mathcal{Q}_{g,k} = \{ 0 \}$ 
        are those admissible metrics such that every conformal constant $Q$-curvature metric $\displaystyle \widetilde g = u^{{4}/({n-2k})} g$ 
        is nondegenerate in the sense that if $\mathcal{F}(g,u) = 0$, then any $v\in W^{k,2}_{g_0}(M)$
        such that 
          \begin{flalign}\label{zeropdelinearized}\tag{$\mathcal{L}^\prime_{g,k,u}$}
            L_{g,k,u} (v) = P_{g,k}(v)-f_{n,k}^{\prime}(u)v=0 \quad {\rm on} \quad M .
        \end{flalign}
        must be the zero function itself. 

        We have now reduced our problem to showing that the set 
        $$\widetilde {\mathcal{G}} = \{ g \in \mathfrak{A}_{k,\alpha} (M) : 
        \mathcal{F}(g,u) = 0 \; {\rm and} \; {\rm ker}(L_{g,k,u})= \{0\}\} $$
        is open and dense. It is helpful now to discuss the regularity of $\mathcal{F}$ and 
        describe some of its partial derivatives. We've already shown that for each 
        $g\in \mathfrak{A}_{k,\alpha}(M)$ the mapping $u \mapsto \mathcal{F}(g,u)$ 
        is $\mathcal{C}^2$, and we computed the second derivative explicitly in 
        Lemma \ref{lem:functional_regularity}. Case, Lin and Yuan \cite{CaseLinYuan} proved 
        that for each $u> 0$ the mapping $g \mapsto \mathcal{F}(g,u)$ is also 
        $\mathcal{C}^2$. 
        In their language, the map 
        $$g \mapsto Q_{g,k} = \frac{2}{n-2k} P_{g,k} (1)$$ 
        is a conformally variational Riemannian invariant of weight $-2k$ with a conformal 
        primitive $\mathcal{E}_{g,k}(1)$. Furthermore, they denote the linearization of 
        $g \mapsto Q_{g,k}$ by 
        \begin{equation}\label{metriclinearization}
            \Gamma_g (h) = \left . \frac{\ud}{\ud t} \right |_{t=0} \frac{2}{n-2k} P_{g+th, k} (1).
        \end{equation}
        This operator has a formal adjoint $\Gamma_g^*$, determined by 
        $$\int_M \langle \Gamma_g^*(f), h \rangle \ud\mu_g = \int_M f\Gamma_g(h) \ud\mu_g; $$
        it is this formal adjoint that is the (Frechet) derivative of $\mathcal{F}$ with 
        respect to $g$, namely
        $$\Gamma^*_g = \partial_g\mathcal{F}.$$
    
        In \cite[Section 8.1]{CaseLinYuan} they show this 
        functional is variationally stable, which in the notation established above means 
        \begin{equation} \label{var_stable} 
        \ker (D^2 (g \mapsto \mathcal{E}_{g,k} (1)) =: T_g\mathfrak{K} = \left\{ Y \in 
        \mathcal{C}^\infty(M) : Y= \frac{n}{2} \operatorname{div}(X) 
        \; \textrm{ for some} \; X \in \mathcal{K}(g) \right\}. 
        \end{equation} 
        Here $\mathcal{L}_X g$ is the Lie derivative of the metric tensor $g$ in the 
        direction of $X$ and $\mathcal{K}(g)$  is the space of conformal Killing fields on $(M,g)$, 
        {\it i.e.} the space of vector fields whose flows act as one-parameter families of 
        global conformal diffeomorphisms. Recall that $X\in \mathcal{K}(g)$ if and only 
        if $\mathcal{L}_X g= \psi g$ for some function $\psi$. Evaluating the trace of 
        both sides of this equation shows that multiple 
        $\psi$ must be $\frac{2}{n} \operatorname{div}(X)$. 
        In summary, the only variations of $\mathcal{F}$ with respect to the 
        metric preserving 
        the functional $\mathcal{E}_{g,k}(1)$ at a critical point are those arising from global 
        conformal diffeomorphisms. Given that solution space of the PDE \eqref{zeropde} 
        is conformally invariant this is the smallest one asks for the kernel to be. In 
        any case, this kernel is always finite-dimensional and by \cite[Theorem 7.4]{BCS} it is 
        generically trivial. 

        We complete our proof by verifying the hypotheses of Lemma~\ref{lm:henry}.
        It follows from the analysis in \cite[Section 2.1]{MR3073449} that the  
        transversality map
        $\mathcal{F}:\mathfrak{A}_{k,\alpha}(M) \times W^{k,2}_{g_0}(M) \rightarrow 
        W^{k,2}_{g_0}(M)$ is of class $\mathcal{C}^1$. We must work slightly more to show 
        $\mathcal{F}(g,\cdot)$ is Fredholm and has index zero for each $g \in \mathfrak{A}_{k,\alpha}(M)$. 
        Let $H^k_g(M)$ be the completion of $\mathcal{C}^\infty(M)$ with respect to the 
        inner product 
        \begin{equation}\label{innerproductnondeg}
            (( u,v ))_g = \int_M u P_{g,k} (v) \ud\mu_g.
        \end{equation}
        Standard computations show that the Hilbert subspaces $H^k_g(M)$ and $W^{k,2}_g(M)$ of 
        $\mathcal{C}^{\infty}(M)$ are equivalent, 
        and so the canonical inclusion $W^{k,2}_g(M) \hookrightarrow H^k_g(M)$ is an isomorphism of Banach spaces. 
        The same holds for the inclusion $W^{k,2}_{g^{\prime}}(M) \hookrightarrow H^k_g(M)$ for any $g^{\prime} \in 
        \mathfrak{A}_{k,\alpha}(M)$ (see \cite[Proposition 2.2]{MR1688256}).     
        By the Kondrakov theorem, the canonical inclusion $i_{g}: H^k_g(M) \rightarrow L^{2_k^*}_g(M)$ is compact, 
        and we let $i_g^*$ be its adjoint with respect to the canonical isomorphism 
        $(L^p_g(M))^\prime \simeq L_g^{p^{\prime}}(M)$ with $p^{\prime}=\frac{p}{p-1}$. 
        In fact, we can specify $i_g^*$ by the 
        relation 
        $$((i_g^* u,v))_g = \int_M u v \ud\mu_g \quad \text{for each } \quad v \in L^{2_k^*}_g(M).$$
        
        Now define the Nemytskii operator
        $$\mathcal{N}_{n,k} : W^{k,2}_g(M) \rightarrow L^{\frac{2n}{n+2k}}_g(M),$$
        by
        $$
        \mathcal{N}_{n,k} (u) = f_{n,k}(u) = c_{n,k} u^{\frac{n+2k}{n-2k}},$$
        which is a compact operator because $g\in\mathfrak{A}_{k,\alpha}(M)$ is an admissible metric. 
        The inclusion $W^{k,2}_g(M) \hookrightarrow L^{\frac{2n}{n+2k}}_g(M)$ given 
        by the Sobolev embedding shows $\mathcal{N}_{n,k}$ is $\mathcal{C}^1$ and its Frechet derivative is 
        $$\ud\mathcal{N}_{n,k}[u](v) = \frac{n+2k}{n-2k} c_{n,k} u^{\frac{4k}{n-2k}} v.$$
        In addition, direct computation now yields 
        $$\partial_u \mathcal{F} (v) = v - \left(i_g^* \circ  \ud\mathcal{N}_{n,k}[u]\right)(v). $$
        Since the composition $i_g^* \circ \ud\mathcal{N}_{n,k}[u]$ is compact, we have just demonstrated that 
        $\mathcal{F}(g,\cdot)$ is Fredholm and has index zero. We now conclude that the following 
        four conditions are equivalent: 
        \begin{itemize}
        \item[{\rm (i)}] $0$ is a regular value of $\mathcal F(g,\cdot)$;
        \item[{\rm (ii)}]$\partial_u \mathcal F(g, u) = \text{id} - (i_g^* \circ f'_{n,k})(u)$ is injective for 
        every solution $u$ to \eqref{zeropde};
        \item[{\rm (iii)}] $v\equiv 0$ is the only solution to \eqref{zeropdelinearized} for every solution $u$ to \eqref{zeropde};
        \item[{\rm (iv)}] every solution $u\in W^{k,2}_{g_0}(M)$ to \eqref{zeropde} is non-degenerate.
        \end{itemize}

       We now verify that $0$ is a regular value of $\mathcal F$, which is one of the assumptions 
       of Lemma \ref{lm:henry}. To do so, let $(g,u) \in \mathfrak{A}_{k,\alpha} (M) \times 
       W^{k,2}_{g_0}(M)$ such 
        that $\mathcal{F}(g,u) = 0$ and let $v \in W^{k,2}_{g_0}(M)$. We need to find 
        a symmetric $2$-tensor $h \in \Sigma^2(TM)$ and a function $w\in W^{k,2}_{g_0}(M)$ 
        such that 
        \begin{equation}
        \label{regular value}
          \left. \partial_g \mathcal F\right |_{(g,u)} (h) + 
        \left. \partial_u \mathcal F \right |_{(g,u)} (w) = v. 
        \end{equation}
        
        Since $\partial_u \mathcal F(g,u)$ is a self-adjoint compact perturbation of the identity, 
        we have the orthogonal decomposition
        \[ W_{g_0}^{k,2}(M) = \ker \left. \partial_u \mathcal F \right |_{(g,u)} 
        \oplus \operatorname{Im} \left. \partial_u \mathcal F \right |_{(g,u)},   \]
        and that $\dim \ker \left. \partial_u \mathcal F \right |_{(g,u)} < \infty$. 

       Let $\Pi: W^{k,2}_{g_0}(M) \to \ker \left. \partial_u \mathcal F \right |_{(g,u)}$ be the projection onto 
       $\ker \left. \partial_u \mathcal F \right |_{(g,u)}$. We first claim that $\Pi \left. \partial_g 
\mathcal F \right |_{(g,u)}$ is surjective onto $\ker \left. \partial_u \mathcal F \right |_{(g,u)}$. Indeed, 
if this is not the case, there is $0 \not 
\equiv \psi \in \ker \left. \partial_u \mathcal F \right |_{(g,u)}$ such that 
\begin{equation}\label{regular1new}
            0=((\partial_g\mathcal{F}(\cdot,u)[h]),\psi))_g=
            ((\Gamma_g^*(h), \psi))_g \quad {\rm for \ all} 
            \quad h \in {\rm Sym}^{k,\alpha}(M).
        \end{equation}
where $((,))_g$ is the inner product given by \eqref{innerproductnondeg}.
        A direct computation shows that \eqref{regular1new} can be reformulated as 
        \[\int_M \left[\Gamma^*_g (h)(\nabla u, \nabla \psi)-\frac{1}{2}\left(\operatorname{tr}_g 
        h\right)f_{n,k}(u)
        \psi \right]\ud\mu_g=0,\]
        where $\Gamma^*_g (h)$ is given by \eqref{metriclinearization}.

        Following \cite[Lemma~12]{MR2982783} and using normal coordinates centered at arbitrary $x \in M$ and specific perturbations of $g$, one can prove that this last displayed equation implies 
        \begin{equation}\label{regular2new}
            \langle\nabla u, \nabla \psi\rangle_g= 
            0 \quad \mu_g \text {-a.e.} 
            \quad \text{for each }h \in {\rm Sym}^{k,\alpha}(M).
        \end{equation}
        By taking $h=\varphi g$ for arbitrary $\varphi \in \mathcal{C}^{\infty}(M)$ we find that \eqref{regular1new} 
        and \eqref{regular2new} imply $f_{n,k}(u) \psi=0$ almost everywhere on $M$ with respect to $\mu_g$. 
       However, since $f_{n,k}(u) = c_{n,k} u^\frac{n+2k}{n-2k} > 0$ on $M$, this implies that $\psi \equiv 0$ on $M$, 
       which is a contradiction. Thus we have shown that $\Pi \left. \partial_g 
\mathcal F \right |_{(g,u)}$ is surjective onto $\ker \left. \partial_u \mathcal F \right |_{(g,u)}$. 
Consequently, there is $h \in \Sigma^2(TM)$ such that $\left. \partial_g 
\mathcal F \right |_{(g,u)}(h) = v_1 + z$, for some $z \in \operatorname{Im} \left. \partial_u \mathcal F \right |_{(g,u)}$.
Moreover, since $v_2 - z  \in \operatorname{Im} \left. \partial_u \mathcal F \right |_{(g,u)}$, there is $w \in 
W^{k,2}_{g_0}(M)$ such that $\left. \partial_u \mathcal F \right |_{(g,u)}(w) = v_2 - z$. Thus, $h$ and $w$ 
satisfy \eqref{regular value}. We conclude that $0$ is a regular value of $\mathcal F$.

        Moreover, the same argument as in \cite[ Lemma 11]{MR2982783} and \cite[Lemma 4.2]{MR2560131} shows 
        that the map 
        $\pi_1\circ\iota: \mathcal{F}^{-1}(0)\rightarrow W^{k,2}_{g_0}(M)$ is $\sigma$-proper. 
        Here $\iota: \mathcal{F}^{-1} (0) \rightarrow 
        W^{k,2}_{g_0}(M) \times \mathfrak{A}_{k,\alpha}(M)$ is the canonical embedding and 
        $\pi_1 : W^{k,2}_{g_0}(M) \times \mathfrak{A}_{k,\alpha}(M)\rightarrow W^{k,2}_{g_0}(M)$ is the 
        projection onto the first factor. 

        Finally, putting all this information together, we see the hypotheses of Lemma~\ref{lm:henry} all hold, and so 
         \begin{equation*}
			\left\{g\in \mathfrak{A}_{k,\alpha}(M) : \text{$0$ is a regular value of } \mathcal F(g,\cdot) \right\} = 
           \mathcal{G}
		\end{equation*}
        is an open, dense set subset $\mathfrak{A}_{k,\alpha}(M)$ with respect to the $\mathcal{C}^{k,\alpha}$-topology. 
    \end{proof}

\section{Cubic degenerate stability for $k=2$ (proof of Theorem \ref{AS3_examples})} \label{sec:cubic_stable}

In this section, we construct examples of manifolds satisfying the Adams--Simon integrability condition in Definition~\ref{def:adams-simon}.
Our approach follows closely the ones in \cite[Section 5]{MR3352243}.

\subsection{Products with spheres} \label{sec:prod_spheres}

In this section we prove that certain products of Einstein manifolds with 
a standard round sphere are degenerate critical points of $\mathcal{Q}_{h,2}$ 
satisfying the $\mathrm{AS}_3$ condition.  
\begin{proposition}
\label{prop:gamma=1 m=2 sphere}
    Let $\ell,m \in \mathbb{N}$ with $\ell \geq 2$, let $\lambda \in \R$, 
    let $(M,g)$ be an $m$-dimensional $\lambda$-Einstein manifold,
    and let $(X,h) = (M \times \Ss^\ell, g \oplus \overset{\circ}{g})$. If $m\gg1$ is 
    sufficiently large, then there exist two real numbers $\lambda_- < \lambda_+$ 
    such that the function $u\equiv 1$ is a degenerate critical point of the 
    functional $\mathcal{Q}_{h,2}$ satisfying the $\mathrm{AS}_3$ condition. If $\ell 
    \in \{ 2,3,4,5\}$ then $\lambda_- < 0 < \lambda_+$, but if $\ell \geq 6$ then 
    $\lambda_- < \lambda_+ < 0$.
\end{proposition}

\begin{proof}  Following \eqref{AS3} and the 
analysis in \cite[Section 5]{MR3352243}, our goal now 
is to produce a nontrivial test function $v\in W^{2,2}_h(X)$ such that 
\begin{equation}\label{kernel_condition1}
v \in \Lambda_0 = \ker \left ( P_{h,2} - \frac{m+\ell+4}{2} Q_{h,2} \right ).
\end{equation}

Recall that the round metric on the sphere is an Einstein metric with 
Einstein constant of the round sphere $\ell - 1$. 
Direct computation shows that since $\Ric_h = \lambda g \oplus (\ell-1) \overset{\circ}{g}$, one has
$$R_h = m\lambda + \ell(\ell-1),\quad {\rm and} \quad |\Ric_h|^2 = m \lambda^2 + \ell (\ell-1)^2, $$
which we substitute into \eqref{Q curvature definition} to see 
\begin{eqnarray} \label{sphere_prod_q_curv} 
Q_{h,2} & = & -\frac{2}{(m+\ell-2)^2} |\Ric_h|^2 + \frac{(m+\ell)^3 - 4(m+\ell)^2 
+ 16 (m+\ell) - 16}{8(m+\ell-1)^2(m+\ell-2)^2} R_h^2 \\ \nonumber 
& = & -\frac{2(m\lambda^2 + \ell(\ell-1)^2)}{(m+\ell-2)^2} + \frac{((m+\ell)^3 
- 4(m+\ell)^2 + 16(m+\ell) - 16)(m\lambda + \ell(\ell-1))^2}{8(m+\ell-1)^2(m+\ell-2)^2} 
\\ \nonumber 
&^ = & \left ( \frac{m}{8} - \frac{1}{2} \right ) \lambda^2 + \frac{\ell(\ell-1)}{4} \lambda 
+ \mathcal{O}\left ( \frac{1}{m} \right ).
\end{eqnarray} 
Using these expressions for $\Ric_h$, $R_h$ and $Q_{h,2}$ in \eqref{paneitz definition} we 
find 
\begin{eqnarray} \label{shifted_paneitz1} 
\left ( P_{h,2} - \frac{n+4}{2} Q_{h,2} \right ) & = & \Delta_h^2 + 
\frac{4}{m+\ell-2} (\lambda \Delta_g + (\ell-1)\Delta_{\overset{\circ}{g}} ) \\ \nonumber 
&& - \frac{((m+\ell-2)^2+4)(m\lambda + \ell(\ell-1))}{2(m+\ell-1)(m+\ell-2)} (\Delta_g + 
\Delta_{\overset{\circ}{g}} )  \\ \nonumber 
&& - \left ( \left (\frac{m}{2} - 2\right )\lambda^2 + \ell(\ell-1) \lambda 
+ \mathcal{O}\left ( \frac{1}{m} \right ) \right ) \\ \nonumber 
&= & \Delta_g^2 + \Delta_g \circ \Delta_{\overset{\circ}{g}} 
+ \Delta_{\overset{\circ}{g}} \circ \Delta_g + \Delta_{\overset{\circ}{g}}^2 \\ \nonumber 
&& + \left ( -\frac{m\lambda}{2} + 2\lambda - \frac{\ell(\ell-1)}{2} + \mathcal{O}
\left ( \frac{1}{m} \right ) \right ) (\Delta_g + \Delta_{\overset{\circ}{g}}) \\ \nonumber 
&&  - \left ( \left (\frac{m}{2} - 2\right )\lambda^2 + \ell(\ell-1) \lambda 
+ \mathcal{O}\left ( \frac{1}{m} \right ) \right ) .
\end{eqnarray} 

For our test function we choose $v = 1 \otimes \widetilde v$, where 
$$\widetilde{v} = x_1 x_2 + x_2 x_3 + x_3 x_1 $$
is a homogeneous, harmonic polynomial of degree two restricted to the sphere. This is 
an eigenfunction of $-\Delta_{\overset{\circ}{g}}$ with eigenvalue $2(\ell+1)$. As in \cite[p.5]{Koenig2023BLMS}, observe that 
$$\int_X v^3 \ud\mu_h = \operatorname{vol}_g(M) \int_{\Ss^n} \widetilde{v}^3 
\ud\mu_{\overset{\circ}{g}} \neq 0.$$ 
Substituting this choice of $v$ into \eqref{shifted_paneitz1} we obtain 
$$\left ( P_{h,2}- \frac{n+4}{2} Q_{h,2} \right ) v = (a(m,\ell)\lambda^2 + b(m,\ell)\lambda 
+ c(m,\ell)) v,$$
where 
\begin{eqnarray}\label{quadratic_coeff1}
a(m,\ell) & = & -\frac{m}{2} + 2 + \mathcal{O}\left ( \frac{1}{m} \right ) \\ \nonumber 
b(m,\ell) & = & -m(\ell+1) -\ell^2+5\ell +4 + \mathcal{O}\left ( \frac{1}{m} \right ) \\ \nonumber 
c(m,\ell) & = & -\ell(\ell^2-1) + 4(\ell+1)^2  + \mathcal{O}\left ( \frac{1}{m} 
\right ) \\ \nonumber 
& = & -\ell^3 + 4\ell^2 + 9 \ell + 4 + \mathcal{O} \left ( \frac{1}{m} \right ) .
\end{eqnarray} 
Thus we complete our proof by showing that, provided $m\gg1$ is sufficiently large, 
the quadratic polynomial $\mathfrak{p}(\lambda)=a\lambda^2 + b\lambda + c$ has two real roots. The discriminant is 
$${\rm disc}_\lambda(\mathfrak{p})=b^2 - 4ac = m^2 (\ell+1)^2 + \mathcal{O} (m) ,$$
which is indeed positive for $m$ sufficiently large, proving the existence of the 
two real roots. To see that one root is positive and the other is negative  when 
$\ell \in \{ 2,3,4,5 \}$, observe that $a(m,\ell) < 0$ for $m$ sufficiently large, whereas 
$c(m,\ell) >0$ for $m$ sufficiently large and $\ell \in \{ 2,3,4,5\}$. This in turn implies the 
quadratic is positive when $\lambda =0$ and negative when $|\lambda|\gg1$ is sufficiently 
large, giving us one positive root and one negative root by the intermediate value theorem. 
\end{proof} 
 
\subsection{Products with complex projective space} \label{sec:prod_CP}

We repeat the analysis of the previous section using complex 
projective space instead of a sphere. 
\begin{proposition}
\label{prop:gamma=1 m=2 CP}
    Let $\ell,m \in \mathbb{N}$ with $\ell \geq 2$, let $\lambda \in \R$, 
    let $(M,g)$ be an $m$-dimensional $\lambda$-Einstein manifold,
    and let $(\widetilde X,\widetilde h) = (M \times \mathbb{CP}^\ell, 
    g \oplus g_{\rm FS})$, where $g_{\rm FS}$ is the Fubini-Study metric. If $m\gg1$ is 
    sufficiently large, then there exist two real numbers $\lambda_- < \lambda_+$ 
    such that the function $u\equiv 1$ is a degenerate critical point of the 
    functional $\mathcal{Q}_{h,2}$ satisfying the $\mathrm{AS}_3$ condition. If $\ell 
    \in \{ 2,3\}$ then $\lambda_- < 0 < \lambda_+$, but if $\ell \geq 4$ then 
    $\lambda_- < \lambda_+ < 0$.
\end{proposition}
We recall that the Fubini-Study metric is the unique metric on $\mathbb{CP}^\ell$ 
making the natural quotient $\Ss^{2\ell+1} \rightarrow \mathbb{CP}^\ell$ a Riemannian 
submersion. 

\begin{proof} In this setting, we use \eqref{AS3} to see that our goal is to 
produce a nontrivial test function $v$ with 
\begin{equation} \label{kernel_condition2}
v \in \ker \left ( P_{\widetilde{h},2} - \frac{m+2\ell + 4}{2} Q_{\widetilde{h},2} \right ). 
\end{equation}
The Fubini-Study metric on $\mathbb{CP}^\ell$ is an Einstein metric with 
Einstein constant $\lambda= 2(\ell + 1)$. Once again, since $\Ric_{\widetilde h} = \lambda g \oplus 2(\ell+1)g_{\rm FS}$, 
direct computations implies 
$$R_{\widetilde h} = m\lambda + 4\ell(\ell+1) \quad {\rm and} \quad |\Ric_{\widetilde h} |^2 = 
m\lambda^2 + 8\ell(\ell+1)^2,$$
and so \eqref{Q curvature definition} yields 
\begin{eqnarray*} 
Q_{\widetilde h, 2} & = & -\frac{2}{(m+2\ell-2)^2} |\Ric_{\widetilde h} |^2 + 
\frac{(m+2\ell)^3 - 4(m+2\ell)^2 + 16 (m+2\ell) - 16)}{8(m+2\ell-1)^2 (m+2\ell-2)^2}
R_{\widetilde h}^2 \\ 
&  = & \frac{-2m\lambda^2 - 16 \ell(\ell+1)^2}{(m+2\ell-2)^2} + 
\frac{((m+2\ell)^3-4(m+2\ell)^2+16(m+2\ell) - 16)(m\lambda +4\ell(\ell+1))^2}
{8(m+2\ell-1)^2(m+2\ell-2)^2} \\ 
& = & \left ( \frac{m}{8} - \frac{1}{2} \right ) \lambda^2 + \ell(\ell+1)\lambda 
+ \mathcal{O} \left ( \frac{1}{m} \right ). 
\end{eqnarray*} 
Substituting this into \eqref{paneitz definition} we then obtain 
\begin{eqnarray} \label{shifted_paneitz2} 
P_{\widetilde h, 2} - \frac{n+4}{2} Q_{\widetilde h, 2} & = & \Delta_{\widetilde h}^2 
+ \frac{4}{m+2\ell-2} (\lambda \Delta_g + 2(\ell+1) \Delta_{g_{\rm FS}} ) \\ \nonumber 
&& -  \frac{((m+2\ell-2)^2 + 4)(m\lambda + 4\ell(\ell+1))}
{2(m+2\ell-1)(m+2\ell -2)} (\Delta_g + \Delta_{\rm g}) - 4Q_{\widetilde h, 2} \\ \nonumber 
& = & \Delta_g^2 + \Delta_g \circ \Delta_{g_{\rm FS}} + \Delta_{g_{\rm FS}} \circ 
\Delta_g + \Delta_{g_{\rm FS}}^2 \\ \nonumber 
&& + \left ( -\frac{m\lambda}{2} + 2\lambda - 2\ell(\ell+1) 
+ \mathcal{O}\left ( \frac{1}{m} \right ) \right ) (\Delta_g + \Delta_{g_{\rm FS}}) 
\\ \nonumber 
&& + \left ( -\frac{m}{2} + 2 \right ) \lambda^2 - 4\ell(\ell+1) \lambda + 
\mathcal{O} \left ( \frac{1}{m} \right ) .
\end{eqnarray} 

Once again we choose a test function of the form $v = 1 \otimes \widetilde{v}$, where 
$\widetilde{v}$ is an eigenfunction of $\Delta_{g_{\rm FS}}$. In this case we choose 
$$\widetilde{v} (z, \overline{z}) = z_1 \overline{z_2} + z_2 \overline{z_1} + z_2 \overline{z_3}
+ z_3 \overline{z_2} + z_3 \overline{z_1} + z_1 \overline{z_3}, $$
which satisfies the equation $-\Delta_{g_{\rm FS}} \widetilde{v} = (8\ell + 16) \widetilde{v}$. 
Furthermore, by \cite[page 25]{Kro20} we have 
$$\int_{\widetilde{X}} v^3 \ud\mu_{\widetilde{h}} = \operatorname{vol}_g (M) 
\int_{\mathbb{CP}^\ell} \widetilde{v}^3 \ud\mu_{g_{\rm FS}} \neq 0,$$
so $v$ is indeed nontrivial. Substituting this choice of $v$ into \eqref{shifted_paneitz2} we obtain 
$$\left ( P_{\widetilde h,2} - \frac{n+4}{2} Q_{\widetilde h, 2} \right ) v = 
(\widetilde a (m,\ell) \lambda^2 + \widetilde b (m,\ell) \lambda + \widetilde c(m,\ell) ) v=:\mathfrak{p}(\lambda),$$
where
\begin{eqnarray} \label{quadratic_coeff2}
\widetilde a (m,\ell) & = & -\frac{m}{2} + 2 + \mathcal{O} \left ( \frac{1}{m} \right ) 
\\ \nonumber 
\widetilde b (m,\ell) & = & -4m(\ell+2) -4(\ell^2-3\ell-8) + \mathcal{O} 
\left ( \frac{1}{m} \right ) \\ \nonumber 
\widetilde c (m,\ell) & = & -16 \ell^3 +16 \ell^2 + 224 \ell + 256 + \mathcal{O} \left ( 
\frac{1}{m} \right ). 
\end{eqnarray} 
We see directly that the discriminant is 
$${\rm disc}(\mathfrak{p})=\widetilde b^2 - 4 \widetilde a \widetilde c = 16 m^2 + \mathcal{O} (m),$$
which is once again positive so long as $m$ is sufficiently large. This once 
more proves that the quadratic polynomial $\mathfrak{p}(\lambda)=\widetilde a \lambda^2 + \widetilde b \lambda 
+ \widetilde c$ has two real roots. We also see directly from \eqref{quadratic_coeff2}
that $\widetilde a<0$ provided $m$ is sufficiently large, and that $\widetilde c > 0$ 
for $\ell=2,3$. By the same argument as in the previous case, we see that when $\ell=2,3$
we obtain one positive and one negative root, while when $\ell \geq 4$ we 
have two negative roots. 
\end{proof} 

\section{Quartic degenerate stability for arbitrary integer $1 \leq k < \frac{n}{2}$ (proof of Theorem \ref{AS4_examples})}
\label{sec:quartic_stable}

In this section, we give the proof of Theorem \ref{AS4_examples}. Before we start with the main argument, several preparations are needed.  

\subsection{Preliminaries: Explicit formulas for $P_{g,k}$ and $\tau_0$}

It is crucial for our purpose to provide and expression for the operator $P_{g,k}$ efficiently and explicitly for arbitrary order $k \geq 1$. This is done 
in the following lemma when $M_\tau:=\mathbb S^1(\tau) \times \mathbb S^{n-1}$ with $\tau > 0$ is furnished with the product metric $g\in {\rm Met}^k(M)$ defined as
\begin{equation}\label{productmetric}
   g=g_{\mathbb S^1(\tau)}\oplus g_{\mathbb S^{n-1}}.
\end{equation} 
Notice that the radius $\tau > 0$ of the circle $\mathbb S^1(\tau)$ is allowed to be arbitrary and the 
resulting formulas \eqref{Lmj} and \eqref{Pgm emden fowler trafo} do not depend on $\tau$. 

\begin{lemma}
    \label{lemma Pgm}
    Let $n,k\in\mathbb N$ 
         with $n>2k$ and let $u\in \mathcal{C}^k(M_\tau)$ be of the form $u(t, \omega) = f(t) Y_j(\omega)$.
         Then, one has 
\begin{equation}
\label{Pg in channel j}
    P_{g,k} u(t, \omega) = (\mathcal L_{k, j} f)(t) Y_j(\omega)  \quad {\rm on } \quad M_\tau,
\end{equation}
where
\begin{equation}
    \label{Lmj}
    \mathcal L_{k, j}  = \prod_{\ell = 1}^k \left(-\partial_t^2 + \left(j + \frac{n}{2} + k - 2 \ell\right)^2 \right).
\end{equation}
Here $(t, \omega) \in \mathbb S^1(\tau) \times \mathbb S^{n-1}$ and $Y_j$ is a spherical harmonic of $\mathbb S^{n-1}$, {\it i.e.} it is an 
eigenfunction of the Laplace-Beltrami operator $(-\Delta)_{\mathbb S^{n-1}}$ with eigenvalue $\mu_j = j(j + n - 2)$. 
Moreover, for any $u \in \mathcal{C}^{2k}(M_\tau)$, let $\Tilde{u} \in \mathcal{C}^{2k}(\R \times \mathbb S^{n-1})$ be its $2 \pi \tau$-periodic extension and 
let $\hat{u} \in \mathcal{C}^{2k}(\R^n \setminus \{0\})$ be defined by 
\begin{equation}
    \label{emden-fowler}
    \hat{u}(x) = |x|^{-\frac{n-2k}{2}} \Tilde{u} \left(\ln |x|, \omega \right) \quad {\rm with} \quad \omega=\frac{x}{|x|}.
\end{equation} 
Then, it holds
\begin{equation}
\label{Pgm emden fowler trafo}
     P_{g,k} u(t, \omega) = e^{\frac{n+2k}{2} t} ((-\Delta)^k \hat{u}) (e^t \omega) \quad {\rm for} \quad (t,\omega) \in 
     M_\tau.
\end{equation}
Finally, the Green's function $G_\tau: M_\tau \times M_\tau \to \R$ is given by 
\begin{equation}
    \label{Green's Pgm}
    G_\tau(t, \omega, s, \eta) = c_{n,k} \sum_{m \in \mathbb Z} | \cosh(t - s - 2 \pi m \tau) - \langle \omega, \eta \rangle |^{-\frac{n-2k}{2}},
\end{equation}
where $c_{n,k} > 0$ is a normalizing dimensional constant and $\langle \cdot, \cdot \rangle$ denotes the scalar product in $\R^n$. 
\end{lemma}

Since $L^2(M_\tau)$ is spanned by functions of the form $u(t, \omega) = f(t) Y_j(\omega)$, \eqref{Pg in channel j} completely 
describes the action of $P_{g,k}$ on $L^2(M_\tau)$. 
Indeed, fix an orthonormal basis $(Y_{j,l})_{l\in\{1,...,N_j\}}$ of the $N_j$-dimensional space of spherical harmonics 
of degree $j$ on $\mathbb S^{n-1}$, namely \[N_j=\frac{(2j+n-2)(j+n-3)!}{(n-2)!j!}.\]
Then, the functions $\mathfrak{a}_{m,j,l}(t) = \cos(\frac{mt}{\tau}) Y_{j,l}(\omega)$ and 
$\mathfrak{b}_{m,j,l} \sin(\frac{mt}{\tau}) Y_{j,l}(\omega)$ form an orthogonal basis 
\begin{equation}
    \label{onb spher harm}
    (\mathfrak{a}_{m,j,l}, \mathfrak{b}_{m,j,l})_{(m,j,l) \in \mathbb N_0\times \mathbb N_0\times\in\{1,...,N_j\}} 
\end{equation}
of $L^2(M_\tau)$. Through Lemma \ref{lemma Pgm}, we can express the action of $P_{g,k}$ on this basis.

\begin{corollary}
\label{corollary Pgk eigenvalues}
Let $n,k\in\mathbb N$ 
         with $n>2k$. For every $\tau > 0$ and $m, j \in \mathbb N_0$, one has
\[ P_{g,k} (\mathfrak{a}_{m,j}) = \alpha_{m,j}(\tau) \mathfrak{a}_{m,j} \quad {\rm and }  \quad P_{g,k} (\mathfrak{b}_{m,j}) = \alpha_{m,j}(\tau) \mathfrak{b}_{m,j} \quad {\rm on} \quad M_\tau\]
for the eigenvalues 
\begin{equation}
    \label{alpha k j definition}
    \alpha_{m,j}(\tau) =  \prod_{\ell = 1}^k \left(m^2 \tau^{-2} + (j + \frac{n}{2} + k - 2 \ell)^2 \right). 
\end{equation} 
\end{corollary}

\begin{proof}
    This is an immediate consequence of \eqref{Pg in channel j} and \eqref{Lmj} together with the fact that 
    $-\partial_t^2 \cos(\frac{mt}{\tau}) = m^2 \tau^{-2} \cos(\frac{mt}{\tau})$ and $-\partial_t^2 \sin(\frac{mt}{\tau}) 
    = m^2 \tau^{-2} \sin(\frac{mt}{\tau})$. 
\end{proof}

\begin{proof}[Proof of Lemma \ref{lemma Pgm}]
We let $u(t,\omega) = f(t) Y_j(\omega)$ and write $\mu_j = j(j + n -2)$ for the $j$-th eigenvalue of the Laplace-Beltrami 
operator $(-\Delta)_{\mathbb S^{n-1}}$. 

Furthermore, by \cite[Theorem 1.1]{Case2023}, we have
\[ P_{g,k} u(t,\omega) = (\mathcal L_{k, j} f)(t) Y_j(\omega),
\]
where the operator $\mathcal L_{k, j}$ is given by 
\begin{equation}
    \label{product CM}
    \mathcal L_{k, j}  = 
\begin{cases}
    \prod_{s = 0}^\frac{k-2}{2} D_{k-1 -2s,j} & \text{if $k$ is even,} \\
    (-\partial_t^2 + \mu_j + \frac{(n-2)^2}{4}) \prod_{s=0}^\frac{m-3}{2} D_{k-1-2s,j} & \text{if $k$ is odd.}
\end{cases}
\end{equation}
Here, for any $L \in \mathbb N_0$, the operator $D_{L,j}$ is given by
\[ D_{L,j} := (\partial_t^2 - \mu_j)^2 - \frac{(n-2)^2}{2} (\partial_t^2 - \mu_j) - 2L^2 (\partial_t^2 + \mu_j) - 
\left(\frac{(n-2)^2}{4} - L^2 \right)^2.  \]
Recalling $\mu_j = j(j+n-2)$, we can check by direct computation that 
\begin{equation}
    \label{DMj product}
    D_{L,j} = \left(- \partial_t^2 + (j + \frac{n-2}{2} - L)^2 \right) \left( - \partial_t^2 + (j + \frac{n-2}{2} + L)^2  \right).
\end{equation}
Let $k\in\mathbb N$ be an even number. Then, as the number $L= k-1-2s$ in \eqref{product CM} runs through $\{k-1,k-3,\dots,3,1\}$, the number $\frac{n-2}{2}-L$  
in \eqref{DMj product} runs through $\{\frac{n}{2}-k,\frac{n}{2} - k +2,\dots,\frac{n}{2}-4,\frac{n}{2}-2\}$, and the number 
$\frac{n-2}{2}+L$ in \eqref{DMj product} runs through $\{\frac{n}{2}+k-2, \frac{n}{2} + k -4,\dots,\frac{n}{2}+2, \frac{n}{2}\}$. Thus, 
for $k\in\mathbb N_0$ even, \eqref{Lmj} follows from \eqref{product CM}. An analogous reasoning, together with the fact that $\mu_j + \frac{(n-2)^2}{4} 
= (j + \frac{n}{2}-1)^2$, gives the conclusion if $k\in\mathbb N_0$ is odd. 

The claimed transformation formula \eqref{Pgm emden fowler trafo} is now a direct consequence of \eqref{Pg in channel j} and \eqref{Lmj} 
combined with \cite[Lemma 2]{Egorov1991}. Indeed, the cited result shows that the transformation of $(-\Delta)^k$ through \eqref{emden-fowler} coincides exactly with the action of $P_{g,k}$ given by   \eqref{Pg in channel j} and \eqref{Lmj}. 

Using \eqref{Pgm emden fowler trafo}, we can also derive the expression \eqref{Green's Pgm} of the Green's function $G_\tau$. To do so, 
suppose that $u \in \mathcal{C}^{2k}(M_\tau)$ and $f \in \mathcal{C}(M_\tau)$ are such that $P_{g,k} u = f$ on $M_\tau$. By \eqref{Pgm emden fowler trafo}, 
the functions 
\[ \hat{u}(x) := |x|^{-\frac{n-2k}{2}} \Tilde{u}(\ln|x|, \omega) \quad {\rm and} \quad \hat{f}(x) := |x|^{-\frac{n+2k}{2}} \Tilde{f}(\ln|x|, \omega) \]
(where $\Tilde{u}$ and $\Tilde{f}$ again denote the $\tau$-periodic extensions of $u$ and $f$ to $\R \times \mathbb S^{n-1}$) then 
satisfy $(-\Delta)^k \hat{u} = \hat{f}$ in $\R^n \setminus \{0\}$. Since the Green's function on $\R^n$ of $(-\Delta)^k$ is 
$\Tilde{c}_{n,k} |x-y|^{- n+2k}$ (for a certain constant $\Tilde{c}_{n,k} >0$), this implies 
\begin{align*}
    \hat{u}(x) = \Tilde{c}_{n,k} \int_{\R^n} |x-y|^{-n+2k} \hat{f}(y)\ud y. 
\end{align*}
Using this, for any $(t, \omega) \in M_\tau$, we compute 
\begin{align*}
    u(t, \omega) = \Tilde{u}(t, \omega) 
    &= e^{\frac{n-2k}{2}t} \hat{u}(e^t \omega)\\
    &=  \Tilde{c}_{n,k} e^{\frac{n-2k}{2}t}  
    \int_{\R^n} |e^t \omega -y|^{-n+2k} \hat{f}(y) \ud y \\
    &= \Tilde{c}_{n,k} e^{\frac{n-2k}{2}t} 
 \int_0^\infty \ud r \int_{\mathbb S^{n-1}}  \ud \eta \,  r^{n-1} |e^{t} \omega - r \eta|^{-n + 2k} r^{-\frac{n+2k}{2}} \Tilde{f}(\ln r, \eta) \\
    &= \Tilde{c}_{n,k} \int_{\R} \ud s \int_{\mathbb S^{n-1}}  \ud \eta \, e^{(t-s)\frac{n-2k}{2}} |e^{t-s} \omega - \eta|^{-n + 2k}  
    \Tilde{f}(s, \eta) \\
    &= \Tilde{c}_{n,k} 2^{-\frac{n-2k}{2}} \int_{\R} \ud s \int_{\mathbb S^{n-1}}  \ud \eta \, |\cosh (t-s) - 
    \langle \omega, \eta \rangle |^{-\frac{n-2k}{2}}  \Tilde{f}(s, \eta) \\
    &= \int_{0}^{2 \pi \tau} \ud s \int_{\mathbb S^{n-1}}  \ud \eta \, \left( c_{n,k} \sum_{m \in \mathbb Z} |\cosh (t-s- 2 \pi m \tau) 
    - \langle \omega, \eta \rangle |^{-\frac{n-2k}{2}} \right) f(s, \eta),
\end{align*}
where in the last step we have set $c_{n,k} := \Tilde{c}_{n,k} 2^{-\frac{n-2k}{2}}$ and used the $\tau$-periodicity of $\Tilde{f}$. 
Since $u$ was arbitrary, it follows that 
\[ (t,\omega,s , \eta) \mapsto  c_{n,k} \sum_{m \in \mathbb Z} |\cosh (t-s- 2 \pi m \tau) - \langle \omega, \eta \rangle |^{-\frac{n-2k}{2}}\]
is the Green's function of $P_{g,k}$ on $M_\tau$. 
\end{proof}

We now turn to the proper choice of the radius $\tau_0$ in \eqref{M definition}. For this purpose, by \eqref{Lmj}, let us write 
\begin{equation}
\label{polynomial Pm definition}
    \mathsf P_{k} (X) := X^{2k} + p_{k,{k-1}}X^{2k-2} + ... + p_{k,1} X^2 + p_{k,0}  := \prod_{\ell = 1}^k \left(X^2 + 
    (\frac{n}{2} + k - 2 \ell)^2 \right) 
\end{equation} 
for the polynomial such that $\mathsf P_{k} (-\partial_t^2) = \mathcal L_{k,0}$. 
Notice that the coefficients $p_{k,m}$ are positive for all $m = 0,...,k-1$ because $\frac{n}{2} + k - 2 \ell > 0$ for all 
$\ell \in\{ 1,\dots,k\}$ (as a consequence of $n > 2k$). For completeness, we define $p_{k,k} :=1$. 

We choose $\tau_0> 0$ such that $\varphi(t,\omega) = \sin({t}/{\tau_0} )$ is in 
the kernel of the linearization of the equation 
\[ P_{g,k} u = p_{k,0} u^\frac{n+2k}{n-2k} \quad {\rm on} \quad M_{\tau_0} \]
about the constant solution $u = 1$. In other words, we require $\varphi(t,\omega) = \sin({t}/{\tau_0} )$  to solve the linear equation
\begin{equation}
    \label{lineq gamma=2 general m}
     P_{g,k} \varphi = p_{k,0} \frac{n+2k}{n-2k} \varphi \quad {\rm on} \quad M_{\tau_0}.
\end{equation}
This will lead to the desired degeneracy, as we check in Lemma \ref{lemma kernel gen m}. Let us check that this requirement determines $\tau_0 > 0$ uniquely. Indeed, using that $\mathcal L_{k,0} = 
 \mathsf P_k(-\partial_t^2)$, we have, for every $\tau > 0$, 
\begin{equation}
    \label{sin Pm}
    P_{g,k} \sin\left (\frac{t}{\tau}\right ) =  \mathsf P_k(\tau^{-2})  \sin\left (\frac{t}{\tau}\right ) 
    \quad {\rm on} \quad  M_{\tau}.
\end{equation} 
Since all the $p_{k,m}$ are positive, $\mathsf P_k(\tau^{-2})$ is strictly decreasing as a function of $\tau \in (0, \infty)$, 
with $\lim \limits_{\tau \to 0} \mathsf P_k(\tau^{-2}) = \infty$ and $\lim \limits_{\tau \to \infty} \mathsf P_k(\tau^{-2}) = p_{k,0}$. 
Since $\frac{n+2k}{n-2k} > 1$, we can pick $\tau_0 > 0$ as the unique number satisfying 
\begin{equation}
    \label{tau 0 definition general m}
\mathsf P_k(\tau_0^{-2}) = p_{k,0} \frac{n+2k}{n-2k}.  
\end{equation} 
Thus, by \eqref{sin Pm}, $\varphi(t, \omega) = \sin({t}/{\tau_0})$ solves \eqref{lineq gamma=2 general m} as desired.

\begin{remark}
    \label{remark tau0 is smallest}
   By a similar argument one can check more generally that for every $m \in \mathbb N$, there is precisely one $\tau_0^{(m)} > 0$ such 
   that $\alpha_{m,0}(\tau_0^{(m)}) = \frac{n+2k}{n-2k} p_{k,0}$.  Consequently, the functions \[\varphi(t, \omega) = 
   \sin\left (\frac{mt}{\tau_0^{(m)}}\right ) \quad {\rm or} \quad \varphi(t, \omega) = 
   \cos\left (\frac{mt}{\tau_0^{(m)}}\right )\quad \] 
   solve \eqref{lineq gamma=2 general m} on $\mathbb S^1(\tau_0^{(m)}) \times \mathbb S^{n-1}$. 
   On the other hand, if $j \geq 1$, then for every $m \in \mathbb N_0$ and $\tau > 0$ we have
   \[ \alpha_{m,j}(\tau) \geq \alpha_{0,1}(\tau) = \alpha_{0,1}(\tau_0) > \alpha_{1,0}(\tau_0) = \frac{n+2k}{n-2k} p_{k,0}, \]
   by the inequality \eqref{alpha 01 > alpha 10} below.  

   Since the $\alpha_{m,j}(\tau)$ are decreasing functions of $\tau$, the value $\tau=\tau_0$  from \eqref{tau 0 definition general m} 
   is actually the \emph{smallest} value of $\tau$ such that the linearized equation \eqref{lineq gamma=2 general m} has a non-zero solution. 
\end{remark}

\subsection{Bounding $\mathsf{P}_{k}$ and $\tau_0$}

We will find the following bounds regarding $\tau_0$ useful in later computations. 
\begin{lemma}
\label{lemma tau 0 bounds}
    Let $n,k\in\mathbb N$ 
         with $n>2k$ and let $\tau_0 = \tau_0(n,k)$ be defined by \eqref{tau 0 definition general m}.  Then, one has
    \[ \frac{1}{\sqrt{n + 2k - 4}} \leq \tau_0(n,k) \leq \frac{1}{\sqrt{n-2k}}. \]
\end{lemma}

Notice that for $k = 1$ the upper and lower bounds coincide and we recover 
$\tau_0(n,1)= \frac{1}{\sqrt{n-2}}$. 

\begin{proof}
    
    Since $\mathsf P_k$ is strictly increasing on $[0, \infty)$, in view of \eqref{tau 0 definition general m}, the inequality $\tau_0 \leq 
    \frac{1}{\sqrt{n-2k}}$ follows if we can show that $\mathsf P_k(n-2k) \leq \frac{n+2k}{n-2k} 
    p_{k,0}$. 

    By writing 
    \[ \frac{n+2k}{n-2k} p_{k,0} = \prod_{\ell=1}^k \left(\frac{n}{2} + k - 2 \ell + 2\right)\left(\frac{n}{2} 
    + k - 2 \ell\right) \]
    and 
    \[ \mathsf P_k(n -2k) = \prod_{\ell=1}^k \left( n - 2k + \left(\frac{n}{2} + k - 2 \ell\right)^2 \right) \]
    the desired inequality $\mathsf P_k(n-2k) \leq \frac{n+2k}{n-2k} p_{k,0}$ follows from the fact that 
    \begin{align*}
        \left(\frac{n}{2} + k - 2 \ell + 2\right)\left(\frac{n}{2} + k - 2 \ell\right) - \left( n-2k + \left(\frac{n}{2} 
        + k - 2 \ell\right)^2 \right)&= 2 \left(\frac{n}{2} + k - 2 \ell\right) - (n-2k)\\
        &= 4 k - 4 \ell \geq 0
    \end{align*}
for every $\ell \in\{ 1,\dots,k\}$. 

Analogously, from the fact that 
   \begin{align*}
        \left(\frac{n}{2} + k - 2 \ell + 2\right)\left(\frac{n}{2} + k - 2 \ell\right) - \left( n+2k - 4 + \left(\frac{n}{2} + k - 2 \ell\right)^2 \right)
        &= 2 \left(\frac{n}{2} + k - 2 \ell\right) - (n+2k-4)\\
        &= -4 \ell + 4 \leq 0
    \end{align*}
for every $\ell \in\{ 1,\dots,k\}$, we deduce $\mathsf P_k(n+2k -4) \geq \frac{n+2k}{n-2k} p_{k,0}$, and 
hence $\tau_0(n,k) \geq \frac{1}{\sqrt{n+2k-4}}$.  
\end{proof}

\begin{lemma}
    \label{lemma strict binding ineq}
    Let $n,k \in \mathbb N$ with $n>2k$ and let $M = \mathbb S^1(\tau_0) \times \mathbb S^{n-1}$ 
    with $\tau_0$ defined by \eqref{tau 0 definition general m} furnished with the product metric $g\in{\rm Met}^{k}(M)$ given by \eqref{productmetric}. Then, one has 
\begin{equation}
    \label{strictbinding statement}
    \mathcal Q_{g,k} (1) < \mathcal S_{n,k},
\end{equation}
where $\mathcal S_{n,k} := \omega_n^{{2k}/{n} }{\Gamma(\frac{n+2k}{2})}
{\Gamma(\frac{n-2k}{2})}^{-1}$ is the best constant for the $k$-th order Sobolev inequality on $\mathbb S^n$ with $\omega_n = {2 \pi^{\frac{n+1}{2}}}{\Gamma(\frac{n+1}{2})^{-1}}$ its volume measure. 
\end{lemma}

\begin{proof}
    We have $\mathcal Q_{g,k}(1) = {\rm vol}_g(M)^\frac{2}{n} p_{k,0}$, with 
    \[ p_{k,0} = \prod_{\ell=1}^k \left( \frac{n}{2} + k - 2 \ell \right)^2\]
    the $0$-th order coefficient of the polynomial $\mathsf P_k(X) = \prod_{\ell=1}^k 
    (X+ (\frac{n}{2} + k - 2 \ell)^2)$ defined in \eqref{polynomial Pm definition}. Using that ${\rm vol}_g(M) = 2 \pi \tau_0 \omega_{n-1}$ and since $\omega_n ={2 \pi^{\frac{n+1}{2}}}{\Gamma(\frac{n+1}{2})}^{-1}$, 
    inequality \eqref{strictbinding statement} is equivalent to 
    \begin{equation}
        \label{strictbinding proof}
       p_{k,0} \frac{\Gamma(\frac{n}{2}-k)}{\Gamma(\frac{n}{2} + k)} < \left( 
\frac{1}{2 \sqrt{\pi} \tau_0} \frac{\Gamma(\frac{n}{2})}{\Gamma(\frac{n+1}{2})} \right)^\frac{2}{n} . 
    \end{equation}

By Lemma \ref{lemma tau 0 bounds}, we know that $\tau_0 = \tau_0(n,k) \leq \frac{1}{\sqrt{n-2k}}$.

Consequently, in view of \eqref{strictbinding proof}, it suffices to prove 
   \begin{equation}
        \label{strictbinding proof 2}
      \Psi_k(n) :=  \left( p_{k,0} \frac{\Gamma(\frac{n}{2}-k)}{\Gamma(\frac{n}{2} + k)} 
      \right)^\frac{n}{2} < 
\frac{\sqrt{n-2k}}{2 \sqrt{\pi}} \frac{\Gamma(\frac{n}{2})}{\Gamma(\frac{n+1}{2})} =: \Phi_k(n). 
    \end{equation}
    
By Stirling's formula, $\Gamma(z) = \sqrt{\frac{2\pi}{z}} \big( \frac{z}{e} \big)^z (1 + \mathrm{o}(1))$ as 
$z \to \infty$. It follows
\[ \frac{\Gamma(\frac{n}{2})}{\Gamma(\frac{n+1}{2})} = \sqrt{\frac{2}{n}} (1 + \mathrm{o}(1)), \quad 
\text{ and so  } \quad \Phi_k(n) \to \frac{1}{\sqrt{2 \pi}} \quad 
\text{as} \quad n \to \infty.\]
Moreover, we have 
\begin{align*}
     \Psi_k(n) &=  \left( p_{k,0} \frac{\Gamma(\frac{n}{2}-k)}{\Gamma(\frac{n}{2} + k)} 
     \right)^\frac{n}{2} = \prod_{\ell=1}^k \left( \frac{\frac{n}{2} + k - 2 \ell}{\frac{n}{2}+ k 
     - 2 \ell + 1} \right)^\frac{n}{2} \\
     & \prod_{\ell=1}^k \left(1 - \frac{1}{\frac{n}{2} + k - 2 \ell + 1}\right)^{\frac{n}{2} + k 
     - 2 \ell + 1} \left(1 - \frac{1}{\frac{n}{2} + k - 2 \ell + 1}\right)^{- k + 2 \ell - 1} \to e^{-k} 
\end{align*} 
as $n \to \infty$. Since $e^{-k} < \frac{1}{\sqrt{2 \pi}}$ for every $k \geq 1$, we can conclude  the 
proof of \eqref{strictbinding proof 2}, by showing that (treating $n$ as a real variable)
\begin{equation}
    \label{sbder}
    (\log \Psi_k)'(n) \geq (\log \Phi_k)'(n) \quad \text{ for all } \quad n > 2k. 
\end{equation} 
On the one hand, we have 
\begin{equation}
    \label{sbder1}
    (\log \Phi_k)'(n) = \frac{1}{2} \left( \frac{1}{n-2k} + \psi\left(\frac{n}{2}\right) - \psi\left(\frac{n+1}{2}\right) 
    \right) \leq \frac{1}{2} \left( \frac{1}{n-2k} - \frac{1}{n} \right) = \frac{k}{n(n-2k)}, 
\end{equation} 
where $\psi(z) = {\Gamma'(z)}/{\Gamma(z)}$ is the Digamma function. The claimed inequality follows 
from the concavity of $\psi$ together with the functional equation $\psi(z + 1) = \psi(z) + 
\frac{1}{z}$ (used with $z = n/2$). 

On the other hand, 
\begin{align}
       \label{log Psi k}
    \log \Psi_k(n) &= \sum_{\ell=1}^k 
 \left(\frac{n}{2} + k - 2 \ell +1\right)  \log \left(1 - \frac{1}{\frac{n}{2} + k - 2 \ell + 1}\right)  \\
 & \quad + \sum_{\ell=1}^k  
 (- k + 2 \ell - 1) \log \left(1 - \frac{1}{\frac{n}{2} + k - 2 \ell + 1}\right). \nonumber
\end{align}

To treat the first summand of \eqref{log Psi k}, we estimate, for every $m > 1$, 
\[ \log\left(1 - \frac{1}{m}\right) = -\sum_{j=1}^\infty \frac{1}{j m^j} \geq  - \frac{1}{m} - \frac{1}{2} 
\sum_{j=2}^\infty \frac{1}{m^j} = - \frac{1}{m} - \frac{1}{2} \left(- 1 - \frac{1}{m} + 
\frac{1}{1 - \frac{1}{m}} \right) = \frac{1}{2} - \frac{1}{2m} - \frac{m}{2(m-1)},\]
using the series expansion of $\log(1 + z)$ and the geometric series. As a consequence, 
\[
\frac{\ud}{\ud m} \left(m \log\left(1 - \frac{1}{m}\right)\right)  = \log\left(1 - \frac{1}{m}\right)+ \frac{1}{m-1} 
\geq \frac{1}{2} - \frac{1}{2m} - \frac{m}{2(m-1)} + \frac{1}{m-1} = \frac{1}{2m(m-1)}
 \]
Applying this with $m = m_{k,\ell} =  \frac{n}{2} + k - 2 \ell + 1$ (note that $m > 1$ because 
$\ell \leq k < \frac{n}{2}$) gives 
\begin{align}
\label{sbder3}
\frac{\ud}{\ud n} \left( \sum_{\ell=1}^k 
 (\frac{n}{2} + k - 2 \ell +1)  \log \left(1 - \frac{1}{\frac{n}{2} + k - 2 \ell + 1}\right) \right) &\geq 
 \frac{1}{4} \sum_{\ell=1}^k \frac{1}{m_{k,\ell}(m_{k,\ell}-1)} \\ \nonumber
 &= \sum_{\ell=1}^k \frac{1}{(n+2k-4\ell)(n+2k - 4 \ell + 2)}.    
\end{align}

Next, direct computation gives that the derivative of the second summand in \eqref{log Psi k} is 
\begin{equation}
    \label{log Psi 2nd}
    \frac{\ud}{\ud n} \left( \sum_{\ell=1}^k  (- k + 2 \ell - 1) \log \left(1 - \frac{1}{\frac{n}{2} + k - 
    2 \ell + 1}\right) \right) = - \sum_{\ell = 1}^k \frac{2k-4 \ell + 2}{(n + 2k - 4\ell + 2) (n + 2k - 4 \ell)}. 
\end{equation}

By combining \eqref{sbder1}, \eqref{sbder3} and \eqref{log Psi 2nd} we obtain 
\begin{align*}
    (\log \Psi_k - \log \Phi_k)'(n) &\geq  \sum_{\ell=1}^k \frac{-2k + 4 \ell - 1}{(n+2k-4\ell)
    (n+2k - 4 \ell + 2)} - \frac{k}{n(n-2k)} \\
    & =  \sum_{\ell=1}^{k-1} \frac{-2k + 4 \ell - 1}{(n+2k-4\ell)(n+2k - 4 \ell + 2)} + 
    \left( \frac{2k-1}{(n-2k)(n-2k + 2)} - \frac{k}{n(n-2k)} \right) \\
    &\geq \sum_{\ell=1}^{k-1} \frac{-2k + 4 \ell - 1}{(n+2k-4\ell)(n+2k - 4 \ell + 2)} + 
    \frac{k-1}{(n-2k)(n-2k+2)}. 
\end{align*} 
In the $k$ summands on the right side, the denominators form a decreasing sequence of positive 
numbers. Moreover, the numerators form an increasing sequence of numbers which sum to zero:
\[ \sum_{\ell=1}^{k-1} (-2k + 4 \ell -1) + (k-1) = \left( -2 k(k-1)  + 4 \frac{(k-1)k}{2} - 
(k-1) \right) + (k-1) =  0. \]
From these facts it is elementary to conclude that the right side is non-negative. 
Thus, \eqref{sbder} follows,  and the proof is complete. 
\end{proof}

\begin{lemma}
\label{lemma Phi < Lambda gen m}
    Let $n,k\in\mathbb N$ and $n>2k$ and let $\tau_0$ be defined by \eqref{tau 0 definition general m}. For any $\ell \geq 0$, one has 
    \[     \frac{\Gamma\left(\frac{n-2k}{2n}\right) \Gamma\left(\frac{n+2k}{2n}+\ell\right)}
    {\Gamma\left(\frac{n+2k}{2n}\right) \Gamma\left(\frac{n-2k}{2n}+\ell\right)}  \leq p_{k,0}^{-1} 
    \mathsf P_k(\tau_0^{-2} \ell^2).  \]
    Furthermore, equality holds if and only if $\ell = 0$ or $\ell = 1$. 
\end{lemma}

\begin{proof}
Let us denote 
\[ \Phi_{n,k}(\ell) := \frac{\Gamma\left(\frac{n-2k}{2n}\right) 
\Gamma\left(\frac{n+2k}{2n}+\ell\right)}{\Gamma\left(\frac{n+2k}{2n}\right) 
\Gamma\left(\frac{n-2k}{2n}+\ell\right)} \quad {\rm and} \quad \Lambda_{n,k}(\ell) :=   p_{k,0}^{-1} 
\mathsf P_k(\tau_0^{-2} \ell^2). \]
  The equality $\Phi_{n,k}(0) = \Lambda_{n,k}(0)$ is immediate. Using the 
  definition \eqref{tau 0 definition general m} we moreover have $\mathsf P_k(\tau_0^{-2} \ell^2) 
  = \frac{n+2k}{n-2k} p_{k,0}$, so that 
  \[ \Lambda_{n,k}(1) = p_{k,0}^{-1} \mathsf P_k(\tau_0^{-2}) =  \frac{n+2k}{n-2k} = \Phi_{n,k}(1). \]
 To prove the strict inequality $\Phi_{n,k}(\ell) < \Lambda_{n,k}(\ell)$ for every $n > 2k$ and 
 $\ell \geq 2$, we argue similarly to Beckner \cite[proof of Theorem 4]{MR1230930}. Since 
 $\Phi_{n,k}(1) = \Lambda_{n,k}(1)$, it is enough to show
  \begin{equation}
      \label{ineq derivative}
      \frac{\ud}{\ud \ell} \ln \Phi_{n,k}(\ell) < \frac{\ud}{\ud \ell} \ln \Lambda_{n,k}(\ell) \quad 
      \text{ for every } \quad \ell > 1
  \end{equation}
 (where we treat $\ell$ as a real variable). To prove \eqref{ineq derivative}, we write $p = 
 \frac{2n}{n-2k} = \frac{2}{1-2s}$ with $s = \frac{m}{n} \in (0, \frac{1}{2})$ thanks to $n >2m$. 
 The left side of \eqref{ineq derivative} then is 
  \begin{align*}
     \frac{\ud}{\ud \ell} \ln \Phi_{n,k}(\ell) &=  \psi\left(\frac{1}{p'} + \ell\right) - \psi\left(\frac{1}{p} 
     + \ell\right) = \psi\left(\frac{1}{2} + s + \ell\right) - \psi\left(\frac{1}{2} - s + \ell\right) \\
     &= \sum_{m = 0}^\infty \frac{1}{\frac{1}{2} - s + \ell + m} - \frac{1}{\frac{1}{2} + s + \ell 
     + m} = 2s \sum_{m =0}^\infty \frac{1}{(\frac{1}{2} + \ell + m)^2 - s^2} \\
     &< 2s \sum_{m =0}^\infty \frac{1}{(\frac{1}{2} + \ell + m)^2 - (\frac{1}{2})^2}.  
  \end{align*}
  Here $\psi = \frac{\Gamma'}{\Gamma}$ denotes the Digamma function. The claimed series expansion 
  follows, e.g., from integrating the expansion \cite[eq. (6.4.10)]{Abramowitz1972}
  \[ \psi'(z) = \sum_{m = 0}^\infty \frac{1}{(z + m)^2}. \]
  But now it is easy to directly evaluate 
  \[ \sum_{m =0}^\infty \frac{1}{(\frac{1}{2} + \ell + m)^2 - (\frac{1}{2})^2} = \sum_{m =0}^\infty 
  \frac{1}{\ell + m} - \frac{1}{1 + \ell + m } = \frac{1}{\ell}\]
  as a telescopic sum. Recalling $s = \frac{m}{n}$, in conclusion we have shown 
  \begin{equation}
      \label{ineq ln Phi}
      \frac{\ud}{\ud \ell} \ln \Phi_{n,k}(\ell) < \frac{2m}{n \ell} . 
  \end{equation} 
  In view of \eqref{ineq ln Phi} it remains to show that 
  \[\frac{2m}{n \ell} \leq  \frac{\ud}{\ud \ell} \ln \Lambda_{n,k}(\ell) = \frac{\frac{\ud}{\ud \ell} 
  \Lambda_{n,k}(\ell)}{\Lambda_{n,k}(\ell)} \]
  Since $\Lambda_{n,k}(\ell) = p_{k,0}^{-1} \mathsf P_k(\tau_0^{-2} \ell^2) = p_{k,0}^{-1} 
  \sum_{k=0}^k p_{k,m} \tau_0^{-2m} \ell^{2m}$, some elementary manipulations show that this is 
  equivalent to 
  \begin{equation}
      \label{ineq Lambda gen m}
        m p_{k,0} \leq \sum_{m=1}^k p_{k,m} (nk - m) \tau_0^{-2m}  \ell^{2m}. 
  \end{equation} 
  Since $\ell > 1$ and $p_{k,m} > 0$ for all $k \in\{ 1,\dots,m\}$, the right side of \eqref{ineq Lambda gen m} 
  is estimated by
  \[ \sum_{m=1}^k  p_{k,m} (nk - m)\tau_0^{-2m}  \ell^{2m} \geq (n-m) \sum_{m=1}^k p_{k,m} \tau_0^{-2m}  
  = (n-m) (\mathsf P_k(\tau_0^{-2}) - p_{k,0}) = (n-m) \frac{4m}{n-2k} p_{k,0}.  \]
  Since $(n-m)\frac{4m}{n-2k} > m$, the proof of \eqref{ineq Lambda gen m} is complete. As explained 
  above, this concludes the proof of the lemma. 
\end{proof}

\subsection{The minimizers of $\mathcal Y_{k,+}(M, [g])$ and their degeneracy}

From now on, the value $\tau = \tau_0$ from \eqref{tau 0 definition general m} and the manifold $M = M_{\tau_0} = 
\mathbb S^1(\tau_0) \times \mathbb S^{n-1}$ (furnished with the product metric $g\in {\rm Met}^k(M)$ defined as \eqref{productmetric}) are fixed. We will write $\alpha_{m,j}:= \alpha_{m,j}(\tau_0)$ for the eigenvalues 
of $P_{g,k}$ from \eqref{alpha k j definition}.    

\begin{lemma}
    \label{lemma M gamma2 minimizers gen m}
   Let $n,k \in \mathbb N$ with $2k<n$ and let $M = \mathbb S^1(\tau_0) \times \mathbb S^{n-1}$ 
    with $\tau_0$ defined by \eqref{tau 0 definition general m} furnished with the product metric $g\in{\rm Met}^{k}(M)$ given by \eqref{productmetric}.
    Then $\mathcal Y_{k,+}(M, g)$ is uniquely minimized by the constant functions. 
\end{lemma}

\begin{proof}
    The proof will be divided into three steps. 
    
    \noindent \textit{Step 1: Existence of minimizers.  }

Once we prove coercivity of $P_{g,k}$, the existence of a minimizer $u> 0$ follows directly from 
\cite[Theorem 3]{Mazumdar2016} combined with Lemma \ref{lemma strict binding ineq} and 
Lemma \ref{lemma Pgm} (which guarantees the needed positivity of the Green's function). It only remains to show that there exists $c>0$ such that 
\[\int_M u P_{g,k}(u)\ud \mu_g \geq c \|u\|_{W^{k,2}(M)}^2\]
Write \[u = \sum_{(m,j,l)\in \mathcal{I}} u_{m,j,l} \mathfrak{a}_{m,j,l} + \tilde u_{m,j,l} \mathfrak{b}_{m,j,l}\]
for certain 
coefficients $ u_{m,j,l}, \tilde  u_{m,j,l} \in \R$, where $\mathfrak{a}_{m,j,l}$ and $\mathfrak{b}_{m,j,l}$ are 
the basis functions from \eqref{onb spher harm} and $\mathcal{I}:=\mathbb N_0\times\mathbb N_0\times \{1,\dots,N_j\}$. Then 
it can be deduced from the orthogonality properties of the $\mathfrak{a}_{m,j,l}$, $\mathfrak{b}_{m,j,l}$ that 
\[ \|u\|_{W^{k,2}(M)}^2 \lesssim \sum_{(m,j,l)\in \mathcal{I}}  (u_{m,j,l}^2 + \tilde u_{m,j,l}^2)(1 + m^{2k} + j^{2k}). \]
On the other hand, the expression \eqref{alpha k j definition} of the eigenvalues $\alpha_{m,j}$ 
of $P_{g,k}$ implies the estimate
\[ \alpha_{m,j} \gtrsim 1 + m^{2k} + j^{2k},\]
Hence, the estimate below holds
\[ \int_M u P_{g,k}(u)\ud\mu_g= \sum_{(m,j,l)\in \mathcal{I}}  (u_{m,j,l}^2 + \tilde u_{m,j,l}^2) \alpha_{m,j} 
\gtrsim \sum_{(m,j,l)\in \mathcal{I}}  (u_{m,j,l}^2 
+ \tilde u_{m,j,l}^2)(1 + m^{2k} + j^{2k}), \]
which proves coercivity. 

\noindent\textit{Step 2: Minimizers are radial. }

 Let $0< u \in \mathcal{C}^{2k}(M)$ be a minimizer of $\mathcal Y_{k,+}(M, [g])$ on $M$, which exists by Step 1. Then, up to multiplying it by a suitable scalar factor,$u$ satisfies the Euler-Lagrange equation 
\begin{equation}
    \label{EL eq minimizer}
    P_{g,k} u = c_{n,k}u^\frac{n+2k}{n-2k} \quad {\rm on} \quad M,
\end{equation} 
where $c_{n,k}>0$ is a normalizing dimensional constant.
Let $\Tilde{u} \in \mathcal{C}^{2k}(\R \times \mathbb S^{n-1})$ be the $2 \pi 
\tau_0$-periodic extension of $u$ to $\R \times \mathbb S^{n-1}$ and $v(x) = |x|^{-\frac{n-2k}{2}} \Tilde{u}(\ln |x|, \omega)$ be its logarithmic transform on $\R^n \setminus \{0\}$. By \eqref{EL eq minimizer} and \eqref{Pgm emden fowler trafo}, $v$ satisfies 
\[ (-\Delta)^k v = c_{n,k}v^\frac{n+2k}{n-2k} \quad {\rm on} \quad \R^n \setminus \{0\}. \]
Since $v > 0$, by applying the moving planes method as in \cite{Lin, WeiXu1999} it follows that $v(x)$ only depends on $|x|$. Equivalently, $\Tilde{v}(t, \omega)$, and hence $v(t, \omega)$, only depends on $t$. 

  \noindent \textit{Step 3: Constants are the unique minimizers. }

    While Steps 1 and 2 still are valid for arbitrary $\tau > 0$, in this step we will make crucial use of the expression \eqref{tau 0 definition general m} for $\tau_0$.

By Steps 1 and 2, we only need to prove that constant functions uniquely minimize $\mathcal Q_{g,k}$ among functions $u$ only depending on $t \in\mathbb S^1(\tau_0)$. By Lemma \ref{lemma Pgm}, the inequality we need to show is thus 
    \begin{equation}
        \label{1d ineq S1tau gen m}
        \int_0^{2\pi \tau_0}  u \mathsf P_k (-\partial_t^2) u \ud t  \geq p_{k,0} (2 \pi \tau_0)^{\frac{2k}{n}} \left( \int_0^{2 \pi \tau_0} |u|^\frac{2n}{n-2k} \ud t\right)^\frac{n-2k}{n}
    \end{equation} 
(where $p_{k,0}$ is defined through \eqref{polynomial Pm definition})    with equality if and only if $u$ is constant. 
    
    For the following argument, it will be convenient to rescale $\mathbb S^1(\tau_0)$ back to the unit sphere $\mathbb S^1 \simeq (0, 2 \pi)$. For $v(t)= u(t \tau_0)$ the inequality we need to show then reads as 
    \begin{equation}
        \label{1d ineq S1 gen m}
        \int_{\mathbb S^1}  v \mathsf P_k (- \tau_0^{-2} \partial_t^2) v \ud \sigma \geq p_{k,0} (2 \pi)^{\frac{2k}{n}} \left( \int_{\mathbb S^1} |v|^\frac{2n}{n-2k} \ud\sigma\right)^\frac{n-2k}{n}.
    \end{equation}
    where $\ud\sigma$ denotes the standard measure on the unit circle $\mathbb S^1$.
    We will now prove that \eqref{1d ineq S1 gen m} holds, with equality if and only if $v$ is constant, through an argument inspired by Beckner \cite[proof of Theorem 4]{MR1230930}. For this purpose, we decompose a given $v \in W^{k,2}(\mathbb S^1)$ into its Fourier eigenmodes
		\begin{equation*}
	v(t)=\sum_{\ell=0}^{\infty}Y_\ell(t),
		\end{equation*}
where, for every $\ell \geq 0$, $Y_\ell(t)=c_\ell\cos(\ell t) + d_\ell \sin(\ell t)$ for some $c_\ell, d_\ell \in\mathbb R$. 

    Then the left side of \eqref{1d ineq S1 gen m} reads as 
    \begin{align*}
         \int_{\mathbb S^1}  u \mathsf P_k (- \tau_0^{-2} \partial_t^2) u \ud t  = \sum_{\ell = 0}^\infty \mathsf P_k(\tau_0^{-2} \ell^2) \int_{\mathbb S^1} |Y_\ell|^2\ud \sigma. 
    \end{align*}
   By the 'dual-spectral' version of the Hardy--Littlewood--Sobolev inequality \cite[eq. (19)]{MR1230930} we can estimate the right side of \eqref{1d ineq S1 gen m} as
    \begin{equation}
        \label{hls beckner gen m}
        (2 \pi)^{1 - \frac{n-2k}{n}} \left( \int_{\mathbb S^1} |v|^{\frac{2n}{n-2k}} \ud\sigma\right)^\frac{n-2k}{n} \leq \sum_{\ell = 0}^\infty \frac{\Gamma\left(\frac{n-2k}{2n}\right) \Gamma\left(\frac{n+2k}{2n}+\ell\right)}{\Gamma\left(\frac{n+2k}{2n}\right) \Gamma\left(\frac{n-2k}{2n}+\ell\right)} \int_{\mathbb{S}^1}\left|Y_\ell\right|^2\ud\sigma, 
    \end{equation}

In Lemma \ref{lemma Phi < Lambda gen m} we have proved that 
\begin{equation}
    \label{ineq beckner n, ell gen m}
    \frac{\Gamma\left(\frac{n-2k}{2n}\right) \Gamma\left(\frac{n+2k}{2n}+\ell\right)}{\Gamma\left(\frac{n+2k}{2n}\right) \Gamma\left(\frac{n-2k}{2n}+\ell\right)}  \leq p_{k,0}^{-1} \mathsf P_k(\tau_0^{-2} \ell^2). 
\end{equation} 
for every $n > 2k$, $\ell \geq 0$, and so \eqref{ineq beckner n, ell gen m} follows. 

Moreover, still by Lemma \ref{lemma Phi < Lambda gen m}, equality in \eqref{ineq beckner n, ell gen m} occurs precisely for $\ell = 0,1$. 

On the other hand, by the classification of HLS optimizers \cite{Lieb1983}, equality in \eqref{hls beckner gen m} holds if and only if $v$ is a  conformal factor, i.e., 
    \begin{equation}
        \label{conformal factor gen m}
        v(s) = c (1 + a \cos s + b \sin s)^\frac{n-2k}{2n} 
    \end{equation} 
  for some $c \in \R \setminus \{0\}$, $a, b \in \R$ with $a^2 + b^2 < 1$. 

Hence, if equality holds in \eqref{1d ineq S1 gen m} for some $v$, then by the equality conditions for \eqref{hls beckner gen m} and \eqref{ineq beckner n, ell gen m} we must have 
\[ v(s) = C( 1 + A \cos s + B \sin s) = c (1 + a \cos s + b \sin s)^\frac{n-2k}{2n}. \]
It is easy to see that this implies that $v$ must be constant. This ends the proof.
\end{proof}

In what follows, it will be convenient to 
write $\|u\|_{{2_k^*}}=\|u\|_{L^{2_k^*}(M)}$ with $2_k^*=\frac{2n}{n-2k}$, as well as $\|u\|_{k,2}=\|u\|_{W^{k,2}(M)}$ and 
$\mathcal Y := \mathcal Y_{k,+}(M, [g])$. Here, $M$ is given by \eqref{M definition} furnished with the standard product metric, denoted by $g$. 
Moreover, we abbreviate 
\begin{equation}
    \label{E definition}
    \mathcal E(u) := \frac{2}{n-2k} \int_M u P_{g,k} (u)\ud\mu_g,
\end{equation}
so that
\[\mathcal Q_{g,k}(u) = \frac{\mathcal E(u)}{\|u\|^{2}_{{2_k^*}}}. \]

 Lemma \ref{lemma M gamma2 minimizers gen m} yields that 
\begin{equation}
    \label{F definition}
    \mathcal F(u) := \mathcal E(u)  - \mathcal Y\|u\|^{2}_{{2_k^*}}
\end{equation} 
is uniquely minimized by constant functions (with minimal value $0$). 

We now check that, by our special choice of $\tau_0$, the kernel of $D^2 \mathcal F(1)$ contains a non-constant function.

\begin{lemma}
    \label{lemma kernel gen m}
Let $n,k \in \mathbb N$ with $n>2k$ and let $M = \mathbb S^1(\tau_0) \times \mathbb S^{n-1}$ 
    with $\tau_0$ defined by \eqref{tau 0 definition general m} furnished with the product metric $g\in{\rm Met}^{k}(M)$ given by \eqref{productmetric}. 
    The kernel of $D^2 \mathcal F(1)$ is spanned by $1$, $\cos({t}/{\tau_0})$ and $\sin({t}/{\tau_0})$. 
\end{lemma}

\begin{proof}
Clearly, $1$ is in the kernel of $D^2 \mathcal F(1)$ because $\mathcal F(c) = 0$ for all $c \in \R$. 

On the other hand, for any $\rho\in W^{k,2}(M)$ with $\int_M \rho \ud\mu_g= 0$, we have 
\[ \mathcal E(1 + \eps \rho) = \mathcal E(1) + \mathcal \eps^2 \mathcal E(\rho)  \]
and 
\begin{align*}
    \|1 + \eps \rho\|_{{2_k^*}}^2 &= \left[ \int_M  \left(1 + \eps^2 \frac{2_k^*(2_k^*-1)}{2} \rho^2 + \mathrm{o}(\eps^2) \right)\ud\mu_g\right]^{2/2_k^*} \\
&= \text{vol}_g(M)^{2/2_k^*} + \eps^2 (2_k^*-1) \text{vol}_g(M)^{{2}/{2_k^*} -1} \int_M \rho^2\ud\mu_g + \mathrm{o}(\eps^2).
\end{align*}
Thus, we find
\begin{align}
    \label{D2F(1)}
    \nonumber
    D^2 \mathcal F(1)[\rho, \rho] &= \mathcal E(\rho) - (2_k^*-1) \text{vol}_g(M)^{{2}/{2_k^*} -1}  \mathcal Y \int_M \rho^2 \ud\mu_g\\
    &= \mathcal E(\rho) - \frac{n+2k}{n-2k} \cdot \frac{2}{n-2k} p_{k,0} \int_M \rho^2\ud\mu_g, 
\end{align} 
For the last equality, we used that by Lemma \ref{lemma M gamma2 minimizers gen m}, it holds $\mathcal Y = \mathcal Q_{g,k}(1) = \frac{2}{n-2k} \text{vol}_g(M)^{1 - 2/2_k^*} p_{k,0}$. 

Thus, recalling \eqref{E definition}, $\varphi \in W^{k,2}(M)$ with $\int_M \varphi\ud\mu_g = 0$ is in the kernel of $D^2 \mathcal F(1)$ if and only if 
\begin{equation}
    \label{lin eq proof gen m}
     P_{g,k} (\varphi) = \frac{n+2k}{n-2k} p_{k,0} \varphi.   
\end{equation}
We have already checked in \eqref{sin Pm} that the definition \eqref{tau 0 definition general m} of $\tau_0$ ensures that $\varphi(t) = \sin({t}/{\tau_0})$ solves \eqref{lin eq proof gen m}. By exactly the same argument, $\varphi(t) = \cos({t}/{\tau_0})$ solves \eqref{lin eq proof gen m}. 

It remains to justify that there can be no functions in the kernel of $D^2 \mathcal F(1)$ which are linearly independent of $1$, $\sin({t}/{\tau_0})$ and $\cos({t}/{\tau_0})$.
In view of \eqref{lin eq proof gen m}, it therefore remains to check that $\alpha_{m,j} \neq \alpha_{1,0}$ for all $(m,j) \neq (1,0)$. Since, by \eqref{alpha k j definition}, the eigenvalues $\alpha_{m,j}$ are strictly increasing in $m$ and $j$, this follows if we can show 
\begin{equation}
    \label{alpha 01 > alpha 10}
    \alpha_{0,1} > \alpha_{1,0}. 
\end{equation}
But since 
\[ \alpha_{1,0} = \frac{n+2k}{n-2k} p_{k,0} = 2^{-2k} \frac{n+2k}{n-2k} \prod_{\ell = 1}^k (n + 2k - 4 \ell)^2 = (n-2k)(n-2k+4)^2 \cdots (n+2k -4)^2 (n+2k)
\] 
and 
\[
\alpha_{0,1} = 2^{-2k} \prod_{\ell = 1}^k (2 + n + 2k - 4 \ell)^2, 
\]
we obtain
\begin{align*}
    \frac{\alpha_{0,1}}{\alpha_{1,0}} &= \frac{(n-2k +2)^2}{(n-2k)(n-2k + 4)} \cdot  \frac{(n-2k +6)^2}{(n-2k + 4)(n-2k+8)} \cdot \dots  \cdot \frac{(n+2k-2)^2}{(n+2k-4)(n+2k)}  \\
    &= \prod_{\ell=1}^k \frac{(n+2k+2 - 2\ell)^2}{(n+2k-4\ell)(n+2k-4 \ell + 4)} > 1,
\end{align*}
because of the inequality $\frac{N^2}{(N+2)(N-2)} = \frac{N^2}{N^2-4} > 1$, applied with $N := n+2k+2 - 2\ell$ for every $\ell \in\{ 1,\dots,k\}$. Thus \eqref{alpha 01 > alpha 10} is shown, and the proof is complete. 
\end{proof}

\subsection{The secondary non-degeneracy condition}
\label{subsec secondary nondeg}

    We can now start to give the core argument for the proof of Theorem \ref{AS4_examples}. 
    
    As in \cite{F}, our goal is to verify a 'secondary nondegeneracy condition' using an iterative refinement of Bianchi and Egnell's 
    classical strategy. This strategy consists in decomposing a candidate sequence into a main part and a remainder part orthogonal 
    to it. The orthogonality then implies an improved spectral estimate which can be used to conclude in the classical (i.e., non-degenerate) setting. Because of the degeneracy given 
    by Lemma \ref{lemma kernel gen m}, in our setting, we need to further decompose the remainder into a main part which turns out to 
    be in $\ker D^2 \mathcal F(1)$, and a secondary remainder. At this point only, one has precise enough information to conclude by 
    spectral estimates. 
    
   The first and more standard step of this strategy is contained in the following lemma. 

    \begin{lemma}
\label{lemma frank 1}
Let $n,k \in \mathbb N$ with $n>2k$ and let $M = \mathbb S^1(\tau_0) \times \mathbb S^{n-1}$ 
    with $\tau_0$ defined by \eqref{tau 0 definition general m} furnished with the product metric $g\in{\rm Met}^{k}(M)$ given by \eqref{productmetric}. Assume that $\{u_m\}_{m\in\mathbb N} \subset W^{k,2}(M)$ is 
a sequence such that $\mathcal Q_{g,k}(u_m) \to \mathcal Y$ and $\|u_m\|_{2_k^*} = \mathrm{vol}_g(M)^{1/2_k^*}$. Then, there 
are sequences $\{\lambda_m\}_{m\in\mathbb N} \subset \R$ and $\{\rho_m\}_{m\in\mathbb N} \subset W^{k,2}(M)$ such that $\lambda_m \to \pm 1$, $\int_M \rho_m = 0$,  
$\mathcal E(\rho_m) \to 0$, and up to extracting a subsequence, one has
\[ u_m = \lambda_m (1 + \rho_m). \]
    \end{lemma}

    \begin{proof}
    By Lemma \ref{lemma M gamma2 minimizers gen m}, the only minimizers of $\mathcal Y$ are the constants. By 
    Lemma \ref{lemma compactness minseq}, up to the extraction of a subsequence every minimizing sequence converges to a minimizer 
    of $\mathcal Y$, strongly in $W^{k,2}(M)$. Notice that $M$ satisfies the assumptions of Lemma \ref{lemma compactness minseq}, 
    as we have checked in the proof of the first step of Lemma \ref{lemma M gamma2 minimizers gen m}.

 In view of the normalization of $u_m$, we thus must have $u_m \to \pm 1$ in $W^{k,2}(M)$. Setting $\lambda_m := 
 \overline{u}_m := \text{vol}_g(M)^{-1} \int_M u_m\ud\mu_g$ and $\rho_m := \frac{u_m}{\overline{u}_m} - 1$, the assertion of the lemma follows. 
    \end{proof}

Now we derive a more precise expansion. For sequences $\{u_m\}_{m\in\mathbb N}$ such that $\mathcal E(u_m) - \mathcal Y \|u_m\|_{2_k^*}^2$ goes to 
zero superquadratically in $\mathcal E(u_m - \overline{u}_m)$, we show that the subleading term is necessarily proportional 
to the function 
\begin{equation}
    \label{g cos definition}
    \phi(t, \omega) := \cos\left(\frac{t}{\tau_0}\right) \in \ker D^2 \mathcal F(1),
\end{equation}
up to a rotation in the $t$-coordinate.

    \begin{lemma}
        \label{lemma frank 2}
Let $n,k \in \mathbb N$ with $n>2k$ and let $M = \mathbb S^1(\tau_0) \times \mathbb S^{n-1}$ 
    with $\tau_0$ defined by \eqref{tau 0 definition general m} furnished with the product metric $g\in{\rm Met}^{k}(M)$ given by \eqref{productmetric}.
        Let $\{u_m\}_{m\in\mathbb N}, \{\rho_m\}_{m\in\mathbb N} \subset W^{k,2}(M)$ be sequences such that $u_m = 1 + \rho_m$, where $\int_M \rho_m\ud\mu_g = 0$ 
        and $\mathcal E(\rho_m) \to 0$ as $ m \to \infty$. Suppose that 
        \[ \frac{\mathcal E(u_m) - \mathcal Y \|u_m\|_{2_k^*}^2}{\mathcal E(\rho_m)} \to 0. \]
        Then, up to extracting a subsequence and a rotation in the $t$-coordinate, one has
        \begin{equation}
            \label{frank 2 expansion}
            u_m = 1+ \rho_m = 1 + \xi_m( \phi + R_m), 
        \end{equation} 
        where $\xi_m \to 0$, $\phi\in \ker D^2 \mathcal F(1)$ is defined by \eqref{g cos definition}, and $\{R_m\}_{m\in\mathbb N} \subset W^{k,2}(M)$ satisfies
        $\int_M R_m\ud\mu_g = \int_M R_m \sin({t}/{\tau_0})\ud\mu_g =  \int_M R_m \cos({t}/{\tau_0})\ud\mu_g = 0$, and $\mathcal E(R_m) \to 0$ as $m\to\infty$.  
    \end{lemma}

    \begin{proof}
    By $\int_M \rho_m\ud\mu_g = 0$, we clearly have
    \[\mathcal E(u_m) = \mathcal E(1) + \mathcal E(\rho_m). \]
    Moreover, for any $a \in \R$ and $\nu>2$ the pointwise expansion below holds
    \[|1 + a|^{\nu} = 1 + \nu a + \frac{\nu(\nu-1)}{2} a^2 + \mathcal O(|a|^{\min\{3, \nu\}}),\]
    which yields
        \[ \int_M |u_m|^{2_k^*}\ud\mu_g = \int_M \ud\mu_g + \frac{2_k^*(2_k^*-1)}{2} \int_M \rho_m^2\ud\mu_g + \mathrm{o}(\mathcal E(\rho_m)),  \]
    and hence 
    \[ \|u_m\|_{2_k^*}^2 = \left( \int_M \ud\mu_g\right)^{2/2_k^*} + (2_k^*-1) \left( \int_M \ud\mu_g\right)^{{2}/{2_k^*}-1} \int_M \rho_m^2\ud\mu_g + \mathrm{o} (\mathcal E(\rho_m)).  \]
    Since \[\mathcal Y = \frac{2}{n-2k} {\mathcal E(1)}{\left(\int_M \ud\mu_g\right)^{-2/2_k^*}} = \left(\int_M \ud\mu_g\right)^{1 - {2}/{2_k^*}} \alpha_{0,0}\]
    and $\alpha_{1,0} = (2_k^*-1) \alpha_{0,0}$, it follows
    \begin{align}
    \label{quotient zero}
        \mathrm{o}(1) = \frac{\mathcal E(u_m) - \mathcal Y \|u_m\|_{2_k^*}^2}{\mathcal E(\rho_m)} = \frac{\mathcal E(\rho_m) - \frac{2}{n-2k} \alpha_{1,0} \int_M \rho_m^2 \ud\mu_g}{\mathcal E(\rho_m)}. 
    \end{align}
The quadratic form $\rho \mapsto \mathcal E(\rho) - \frac{2}{n-2k} \alpha_{1,0} \int_M \rho^2\ud\mu_g$ vanishes on the subspace spanned by $\cos({t}/{\tau_0})$ and $\sin({t}/{\tau_0})$. Since $\alpha_{m,j} > \alpha_{1,0}$ for $m \geq 2$ or $j \geq 1$, it is positive definite and equivalent to $\mathcal E$ on the orthogonal complement of $1$, $\cos({t}/{\tau_0})$ and $\sin({t}/{\tau_0})$. Together with $\int_M \rho_m\ud\mu_g = 0$, it is easy to deduce \eqref{frank 2 expansion} from \eqref{quotient zero} by arguing as in \cite[proof of Lemmas 5 and 9]{F}.
    \end{proof}

With the refined expansion from Lemma \ref{lemma frank 2} at hand, we can now expand $\mathcal E(u_m)$ up to fourth order and derive the following 'second-order' stability inequality: 

    \begin{lemma}
        \label{lemma frank 3}
Let $n,k \in \mathbb N$ with $n>2k$ and let $M = \mathbb S^1(\tau_0) \times \mathbb S^{n-1}$ 
    with $\tau_0$ defined by \eqref{tau 0 definition general m} furnished with the product metric $g\in{\rm Met}^{k}(M)$ given by \eqref{productmetric}.
        Let $\{u_m\}_{m\in\mathbb N}, \{\rho_m\}_{m\in\mathbb N} \subset W^{k,2}(M)$ be sequences such that $u_m = 1 + \rho_m$, where $\int_M \rho_m\ud\mu_g = 0$ and $\mathcal E(\rho_m) \to 0$ as $ m \to \infty$. Then, one has 
        \[ \mathcal E(u_m) - \mathcal Y \|u_m\|_{2_k^*}^2 \geq (c + \mathrm{o}(1)) \mathcal E(\rho_m)^2, \]
        where 
        \[ c:= \frac{n-2k}{2}\cdot \frac{(2_k^*-2)}{8 \,  \mathrm{vol}_g(M)} \alpha_{1,0}^{-1} \left(  (2_k^*+1) -  \frac{\alpha_{1,0}} {\alpha_{2,0} - \alpha_{1,0}} (2_k^*-2)  \right)  > 0. \]
    \end{lemma}

    \begin{proof}
By Lemma \ref{lemma frank 2}, we can write $u_m = 1+ \rho_m =  1 + \xi_m ( \phi + R_m)$ with $\xi_m \to 0$ and $\mathcal E(R_m) \to 0$ as $m\to\infty$. 
    
Again we follow the strategy of \cite{F}, but we apply a technical simplification of the argument from \cite{Koenig2023} (see also \cite{FP}) to deal with the set where $R_m$ has pointwise large values. 

Let $M_m \subset M$ be the subset of points $x \in M$ such that $|\rho_m(x)| < \frac{1}{2}$. In particular, on $M_m$, the function $1 +\rho_m$ takes values in $(1/2, 3/2)$, where $t \mapsto |t|^{2_k^*}$ is smooth. In particular, by Taylor's theorem, we can expand to fourth order to get
\begin{align*}
    \int_{M_m} |1 + \rho_m|^{2_k^*}\ud\mu_g &= \int_{M_m} \ud\mu_g + 2_k^* \int_{M_k} \rho_m\ud\mu_g + \frac{2_k^*(2_k^*-1)}{2} \int_{M_m} \rho_m^2\ud\mu_g \\
    &+\frac{2_k^*(2_k^*-1)(2_k^*-2)}{6} \int_{M_m} \rho_m^3\ud\mu_g + \frac{2_k^*(2_k^*-1)(2_k^*-2)(2_k^*-3)}{24} \int_{M_m} \rho_m^4\ud\mu_g + \mathrm{o}(\xi_m^4).  
\end{align*} 
For the complementary set, we directly have 
\[ \int_{M\setminus M_m} \left( \frac{1}{2} \right)^{2_k^*}\ud\mu_g\leq \int_{M\setminus M_m} \left| \rho_m \right|^{2_k^*}\ud\mu_g \lesssim \mathcal E(\xi_m \phi)^{2_k^*/2} + \mathrm{o}(\xi_m^{2_k^*}) = \xi_m^{2_k^*}(\mathcal E(\phi)^{{2_k^*}/{2}}
+ \mathrm{o}(1)), \]
and hence $\text{vol}_g(M\setminus M_m) \lesssim \xi_m^{2_k^*}$. Thus we can use a second-order Taylor expansion to deduce, by estimates analogous to those in \cite[p.11]{Koenig2023}, that 
\begin{align*}
     &\int_{C_m} |1 + \rho_m|^{2_k^*}\ud\mu_g\\
     &= \int_{M\setminus M_m} \ud\mu_g + 2_k^* \int_{M\setminus M_m} \rho_m\ud\mu_g + \frac{2_k^*(2_k^*-1)}{2} \int_{M\setminus M_m} \rho_m^2\ud\mu_g + \frac{2_k^*(2_k^*-1)(2_k^*-2)}{6} \int_{M\setminus M_m} \rho_m^3\ud\mu_g  \\
    &\quad + \frac{2_k^*(2_k^*-1)(2_k^*-2)(2_k^*-3)}{24} \int_{M\setminus M_m} \rho_m^4\ud\mu_g + \mathrm{o}(\xi_m^4 + \xi_m^2 \mathcal E(R_m)).
\end{align*}
Adding up the two expansions above, we obtain 
\begin{align*}
    \|u_m\|_{2_k^*}^2 &= \text{vol}_g(M)^{{2}/{2_k^*}} \\
    & \quad + \text{vol}_g(M)^{{2}/{2_k^*}-1} (2_k^*-1) \left( \int_M \rho_m^2\ud\mu_g +  \frac{2_k^*-2}{3} \int_M \rho_m^3\ud\mu_g + \frac{(2_k^*-2)(v-3)}{12} \int_M \rho_m^4\ud\mu_g
 \right) \\
 &\quad -  \text{vol}_g(M)^{{2}/{2_k^*}-2} \frac{(q-2)(2_k^*-1)^2}{4} \left( \int_M \rho_m^2 \right)^2 +  o\left((\xi_m^4 + \xi_m^2 \mathcal E(R_m)\right),
\end{align*} 
which with $\mathcal E(u_m) = \mathcal E(1) + \mathcal E(\rho_m)$, and again recalling 
\[\mathcal Y \,  \text{vol}_g(M)^{{2}/{2_k^*}-1} (2_k^*-1) = \frac{2}{n-2k}\alpha_{0,0}(2_k^*-1) = \frac{2}{n-2k} \alpha_{1,0},\] 
gives us 
\begin{align}
    \mathcal E(u_m) - \mathcal Y \|u_m\|_{2_k^*}^2 &= \mathcal E(\rho_m) - \frac{2}{n-2k} \alpha_{1,0} \left( \int_M \rho_m^2\ud\mu_g +  \frac{2_k^*-2}{3} \int_M \rho_m^3\ud\mu_g + \frac{(2_k^*-2)(2_k^*-3)}{12} \int_M \rho_m^4 \ud\mu_g\right) \nonumber \\
    &\quad + \frac{2}{n-2k} \alpha_{1,0} \text{vol}_g(M)^{-1} \frac{(2_k^*-1)(2_k^*-2)}{4}  \left( \int_M \rho_m^2 \ud\mu_g\right)^2 +  \mathrm{o}\left((\xi_m^4 + \xi_m^2 \mathcal E(R_m)\right) \label{frank3 expansion1}. 
\end{align}
Now we expand the terms on the right more precisely, according to the refined decomposition $\rho_m = \xi_m (\phi + R_m)$ from Lemma \ref{lemma frank 2}. Clearly, by orthogonality, 
\[ \mathcal E(\rho_m) - \frac{2}{n-2k} \alpha_{1,0}  \int_M \rho_m^2\ud\mu_g = \xi_m^2\mathcal E(R_m) - \frac{2}{n-2k} \alpha_{1,0} \xi_m^2 \int_M R_m^2 \ud\mu_g \]
because $\phi\in W^{k,2}(M)$ is in the kernel of this quadratic form. Moreover, since $\int_M \phi^3\ud\mu_g = 0$,
\[ \int_M \rho_m^3\ud\mu_g  = 3 \xi_m^3 \int_M \phi^2 R_m\ud\mu_g + \mathrm{o}\left((\xi_m^4 + \xi_m^2 \mathcal E(R_m)\right). \]
Finally, in the power-four terms, only the leading term $\xi_m \phi$ is relevant in the sense that 
\[ \int_M \rho_m^4\ud\mu_g = \xi_m^4 \int_M \phi^4\ud\mu_g + \mathrm{o}(\xi_m^4) \quad \text{ and } \quad \left( \int_M \rho_m^2\ud\mu_g \right)^2 
= \xi_m^4 \left( \int_M \phi^2 \ud\mu_g\right)^2+ \mathrm{o}(\xi_m^4).  \]
In view of 
\[ \phi^2(s)= \frac{1}{2} + \frac{1}{2} \cos\left(\frac{2s}{\tau_0}\right), \]
we further decompose
\[ R_m = b_m \cos\left(\frac{2s}{\tau_0}\right) + S_k \quad {\rm with} \quad  \int_M S_k \cos\left(\frac{2s}{\tau_0}\right)\ud\mu_g = 0.\]
Then, one has
\[\int_M \phi^2 R_m\ud\mu_g = \frac{1}{2} b_m \int_M \cos^2 \left(\frac{2s}{\tau_0}\right)\ud\mu_g = b_m \frac{\text{vol}_g(M)}{4}.  \]
Moreover, we can  estimate
\begin{align*}
     &\mathcal E(\xi_m R_m) - \frac{2}{n-2k} \alpha_{1,0} \int_M (\xi_m R_m)^2\ud\mu_g + \mathrm{o}( \mathcal \xi_m^2 \mathcal E(R_m)) \\
     &= \frac{2}{n-2k}(\alpha_{2,0} - \alpha_{1,0} +\mathrm{o}(1)) \xi_m^2 b_m^2 \frac{\text{vol}_g(M)}{2} + (1 + \mathrm{o}(1)) \xi_m^2 \mathcal E(S_k) - \frac{2}{n-2k} \alpha_{1,0} \xi_m^2 \int_M S_k^2\ud\mu_g \\
     &\geq \frac{2}{n-2k} (\alpha_{2,0} - \alpha_{1,0} +\mathrm{o}(1)) \xi_m^2 b_m^2 \frac{\text{vol}_g(M)}{2}.  
\end{align*}
Finally, we can compute explicitly the numerical values of 
\[ \int_M \phi^4\ud\mu_g = \frac{3}{8} \text{vol}_g(M) \quad \text{ and } \quad \left(\int_M \phi^2\ud\mu_g \right)^2 = \frac{\text{vol}_g(M)^2}{4}. \]
Inserting all of this back into \eqref{frank3 expansion1}, we arrive at 
\begin{align}
    & \qquad \frac{n-2k}{2} \left(  \mathcal E(u_m) - \mathcal Y \|u_m\|_{2_k^*}^2 \right) \nonumber \\
    &\geq (\alpha_{2,0} - \alpha_{1,0} +\mathrm{o}(1)) \xi_m^2 b_m^2 \frac{\text{vol}_g(M)}{2}  - \alpha_{1,0} \frac{q-2}{4} \text{vol}_g(M) \xi_m^3 b_m \nonumber \\
     &+ \alpha_{1,0} \text{vol}_g(M) \xi_m^4 \left( \frac{(2_k^*-1)(2_k^*-2)}{16} - \frac{(2_k^*-2)(2_k^*-3)}{32} \right) + \mathrm{o}(\xi_m^4) \nonumber \\
     &= \frac{\text{vol}_g(M)}{8} \xi_m^2 \left( \sqrt{4 (\alpha_{2,0} - \alpha_{1,0} + \mathrm{o}(1))} b_m - \frac{\alpha_{1,0}(2_k^*-2)}{\sqrt{4 (\alpha_{2,0} - \alpha_{1,0} + \mathrm{o}(1))}} \xi_m \right)^2       -  \frac{\text{vol}_g(M)}{8} \frac{\alpha_{1,0}^2 (2_k^*-2)^2}{4 (\alpha_{2,0}-\alpha_{1,0})} \xi_m^4 \nonumber  \\
     &+ \frac{\alpha_{1,0} \text{vol}_g(M)}{32} (2_k^*-2)(2_k^*+1) \xi_m^4  + \mathrm{o}(\xi_m^4) \label{bm square} \\
     & \geq \frac{\text{vol}_g(M)}{32} (2_k^*-2)  \alpha_{1,0} \, \xi_m^4 \left(  (2_k^*+1) -  \frac{\alpha_{1,0}} {\alpha_{2,0} - \alpha_{1,0}} (2_k^*-2)  + \mathrm{o}(1) \right).  \nonumber
\end{align}
Here we completed a square in $b_m$ and simplified the occurring terms.
Since from \eqref{frank 2 expansion}, we have 
\[ \xi_m^4 = \mathcal E(\rho_m)^2( \mathcal E(\phi)^{-2} + \mathrm{o}(1)) = \mathcal E(\rho_m)^2 \left( \left( \frac{n-2k}{2} \right)^2 \frac{4}{\mathrm{vol}_g(M)^2} \alpha_{1,0}^{-2} + o(1) \right), \] 
the proof of the lemma is complete if we can show that $(2_k^*+1) -  \frac{\alpha_{1,0}} {\alpha_{2,0} - \alpha_{1,0}} (2_k^*-2)$ is strictly positive, or equivalently, 
\begin{equation}
\label{alpha 2 inequality}
\alpha_{2,0} > \frac{2\cdot 2_k^*-1}{2_k^*+1} \alpha_{1,0}. 
\end{equation} 
Now recall that $\alpha_{k,0} = \mathsf P_k(k \tau_0^{-1})$ for all $k \in \mathbb N_0$ and that $\alpha_{1,0} = (2_k^*-1) \alpha_{0,0}$ by the choice of $\tau_0$. Using that $\mathsf P_k(X)$ is a convex function of $X$ (being a sum of even-degree monomials with positive coefficients), we find 
\[ \alpha_{2,0} =  \mathsf P_k(2 \tau_0^{-1}) \geq 2 \mathsf P_k (\tau_0^{-1}) - \mathsf P_k(0) = 2 \alpha_{1,0} - \alpha_{0,0} = \left(2 - \frac{1}{2_k^*-1}\right) \alpha_{1,0}. \]
A direct computation shows $2 - \frac{1}{\nu-1} > \frac{2\nu -1}{\nu+1}$ for any $\nu>2$. In particular taking $\nu=2_k^*> 2$, we conclude that \eqref{alpha 2 inequality} is proved.
\end{proof}

    \begin{proof}
        [Proof of Theorem \ref{AS4_examples}]
Let us prove that there is $c > 0$ such that 
        \begin{equation}
            \label{E(u) quotient proof gamma=2}
            \frac{(\mathcal E(u) - \mathcal Y \|u\|_{2_k^*}^2) \mathcal E(u)}{\mathcal E(u - \overline{u})^2}\geq  c  \quad \text{ for all non-constant} \,  W^{k,2}(M).  
        \end{equation}
        We argue by contradiction and assume that there is a sequence $\{u_m\}_{m\in\mathbb N} \subset W^{k,2}(M)$ such that 
        \begin{equation}
            \label{contradiction ass proof thm gamma2}
            \frac{(\mathcal E(u_m) - \mathcal Y \|u_m\|_{2_k^*}^2) \mathcal E(u_m)}{\mathcal E(u_m - \overline{u}_m)^2}  \to 0 \quad {\rm as } m \to \infty.
        \end{equation}
        By zero homogeneity of the quotient in \eqref{contradiction ass proof thm gamma2}, we may assume without loss that $\|u_m\|_{2_k^*} = \text{vol}_g(M)^{1/2_k^*}$. Since $\mathcal E(u_m - \overline{u}_m) \leq \mathcal E(u_m)$, from \eqref{contradiction ass proof thm gamma2} it follows that 
        \[ 0 = \lim_{m \to \infty} \frac{(\mathcal E(u_m) - \mathcal Y \|u_m\|_{2_k^*}^2) \mathcal E(u_m)}{\mathcal E(u_m)^2} = 1 - \lim_{m \to \infty} \frac{\mathcal Y \|u_m\|_{2_k^*}^2}{\mathcal E(u_m)}.  \]
        Thus, $\{u_m\}_{m\in\mathbb N}\subset W^{k,2}(M)$ satisfies the assumptions of Lemma \ref{lemma frank 1}. As a consequence, we get that $\lambda_m^{-1} u_m  = 1 + \rho_m$ satisfies the assumptions of Lemmas \ref{lemma frank 2} and \ref{lemma frank 3} and we conclude
        \[ 0 < c \leq  \frac{(\mathcal E(1 + \rho_m) - \mathcal Y \|1 + \rho_m\|_{2_k^*}^2)}{\mathcal E(\rho_m)^2} = \lambda_m^2 \frac{(\mathcal E(u_m) - \mathcal Y \|u_m\|_{2_k^*}^2)}{\mathcal E(u_m - \overline{u}_m)^2} \lesssim \frac{(\mathcal E(u_m) - \mathcal Y \|u_m\|_{2_k^*}^2) \mathcal E(u_m)}{\mathcal E(u_m - \overline{u}_m)^2} . \]
        This is a contradiction to \eqref{contradiction ass proof thm gamma2} and so the proof of \eqref{E(u) quotient proof gamma=2} is complete. 

        From \eqref{E(u) quotient proof gamma=2}, it is straightforward to deduce the stability estimate \eqref{stabilityestimate}. Indeed, by \eqref{E(u) quotient proof gamma=2} we can estimate  
        \[ \mathcal Q_{g,k}(u) - \mathcal Y =  \frac{(\mathcal E(u) - \mathcal Y \|u\|_{2_k^*}^2) \mathcal E(u)}{\mathcal E(u - \overline{u})^2}  \frac{\mathcal E(u - \overline{u})^2}{\mathcal E(u) \|u\|_{2_k^*}^2} \geq c \frac{\mathcal E(u - \overline{u})^2}{\mathcal E(u) \|u\|_{2_k^*}^2}.  \]
        Using that $\|u\|_{2_k^*}^2 \leq \mathcal Y^{-1} \mathcal E(u)$, and that $\mathcal E(u) \lesssim \|u\|_{k,2}^2$ and $\mathcal E(u - \overline{u}) \gtrsim \|u - \overline{u}\|_{k,2}^2$ by norm equivalence, we get 
        \[ \mathcal Q_{g,k}(u) - \mathcal Y \geq \tilde{c} \frac{\|u - \overline{u}\|_{k,2}^4}{\|u\|_{k,2}^4} \geq  \tilde{c} \frac{ \inf_{c > 0} \|u - c\|_{k,2}^4}{\|u\|_{k,2}^4} = d(u, \mathcal M_{g,k})^4.   \]
        We recall that by Lemma \ref{lemma M gamma2 minimizers gen m} the minimizing set $\mathcal M_{g,k}$ consists precisely of the constants and that $d(u, \mathcal M_{g,k})$ is defined in \eqref{dist}. 
        
        By considering specifically the sequence $\{u_m\}_{m\in\mathbb N}\subset W^{k,2}(M)$ given by
        \[ u_m(s) = 1 + m^{-1}\left[\cos \left(\frac{s}{\tau_0}\right) + b_m \cos \left(\frac{2s}{\tau_0}\right)\right], \]
        with $b_m$ chosen so that the square in \eqref{bm square} vanishes, the inequalities in the above computations become (asymptotic) equalities, and we find
        \[ \mathcal Q_{g,k}(u_m) - \mathcal Y \sim d(u_m, \mathcal M_{g,k})^4.  \]
        Since $d(u_m, \mathcal M_{g,k}) \to 0$ as $m \to \infty$, we therefore cannot have $\mathcal Q_{g,k}(u_m) - \mathcal Y \gtrsim d(u_m, \mathcal M_{g,k})^{2 + \gamma}$ for any $\gamma < 2$. This proves the sharpness of $\gamma = 2$ in the statement of Theorem \ref{AS4_examples}. 
    \end{proof}

   \appendix

   \section{Recursive formulas for GJMS operators}\label{GJMSformulas}

   Formulas for the GJMS operator are only known in a few cases, for instance when the background manifold is Einstein \cite{MR2244375}, or more generally a special Einstein product \cite{MR2574315,Case2023}.
    Nevertheless, in \cite{MR3073887} one can find a recursion formula for these operators.
    We follow the recent construction \cite[Proposition~2.1]{mazumdar2022existence} and define this operator using the following recursion formula.
    
    Let ${\rm A}_g$ be the Schouten tensor defined as
    \begin{equation*}
        \mathrm{A}_g:=\frac{1}{n-2}\left(\operatorname{Ric}-\frac{R_g}{2(n-1)} g\right)
    \end{equation*}
    and ${\rm B}_g$ be the Bach tensor whose coordinates are given by
    \begin{equation*}
        \mathrm{B}_{ j}:=\mathrm{A}_{m\ell} \mathrm{W}_i^{m\ell}+\mathrm{P}_{ij ; m}^m-\mathrm{P}_{im;j}^m,
    \end{equation*}
    where $\mathrm{W}_{imj\ell}$, $\mathrm{A}_{m\ell}$ and $\mathrm{A}_{ij;m\ell}$ are the coordinates of $\mathrm{W}_g, \mathrm{A}_g$ and $\nabla_g^2 \mathrm{A}_g$, respectively. 
    We let $(\cdot, \cdot)$ be the multiple inner product induced by the metric $g$ for the tensors of the type $\mathfrak{T}^{r,s}(M)$.

    For every $k \in \mathbb{N}$ such that $n >2k$, we set
    \begin{align*}
        P_{g,k}= & (-\Delta_g)^k+k (-\Delta_g)^{k-1}\left(J_{g,1} \cdot\right)+k(k-1) (-\Delta_g)^{k-2}\left(J_{g,2} \cdot+\left(\mathrm{T}_{g,1}, \nabla\right)+\left(\mathrm{T}_{g,2}, \nabla_g^2\right)\right) \\
        & +k(k-1)(k-2) (-\Delta_g)^{k-3}\left(\left(\mathrm{T}_{g,3}, \nabla^2\right)+\left(\mathrm{T}_{g,4}, \nabla^3_g\right)\right) \\
        & +k(k-1)(k-2)(k-3) (-\Delta_g){k-4}\left(\mathrm{T}_{g,5}, \nabla^4_g\right)+Z,
    \end{align*}
    where $Z$ is a smooth linear operator of order less than $2k-4$ if $k \geq 3$ and $Z:=0$ if $k \leq 2$, the functions $J_{g,1},J_{g,2}\in \mathcal{C}^\infty(M)$ are defined as
    \begin{equation*}
        J_{g,1}:=\frac{n-2}{4(n-1)}R_g
    \end{equation*}
    and
    \begin{equation*}
        J_{g,2}:=\frac{1}{6}\left(\frac{3 n^2-12 n-4 k+8}{16(n-1)^2} R_g^2-(k+1)(n-4)|A_g|^2-\frac{3 n+2 k-4}{4(n-1)} (-\Delta_g) R_g\right)
    \end{equation*}
    with the tensors $\mathrm{T}_{g,1}, \mathrm{T}_{g,2}, \mathrm{T}_{g,3}, \mathrm{T}_{g,4}$ and $\mathrm{T}_{g,5}$ being defined as
    \begin{align*}
        \mathrm{T}_{g,1}:= & \frac{n-2}{4(n-1)} \nabla R_g-\frac{2}{3}(k+1) \delta_g \mathrm{A}_{g}, \\
        \mathrm{T}_{g,2}:= & \frac{2}{3}(k+1) \mathrm{A}_{g}, \\
        \mathrm{T}_{g,3}:= & \frac{n-2}{6(n-1)} \nabla^2_g R_g+\frac{(k+1)(n-2)}{6(n-1)} R_g\mathrm{A}_{g}-\frac{k+1}{3}(\delta \nabla_g \mathrm{A}_{g}+2 \nabla_g \delta_g \mathrm{A}_{g}+2 \mathrm{Rm}_g \cdot \mathrm{A}_{g}) \\
        & -\frac{2}{15}(k+1)(k+2)\left(3 \mathrm{A}_{g}^{\#} \mathrm{A}_{g}+\frac{\mathrm{B}_{g}}{n-4}\right), \\
        \mathrm{T}_{g,4}:= & \frac{2}{3}(k+1) \nabla_g \mathrm{A}_{g},
    \end{align*}
    and
    \begin{equation*}
        \mathrm{T}_{g,5}:=\frac{2}{5}(k+1)\left(\frac{5 k+7}{9} \mathrm{A}_{g} \otimes \mathrm{A}_{g}+\nabla^2_g \mathrm{A}_{g}\right).
    \end{equation*}
    Here \# stands for the musical isomorphism with respect to $g$ and $\delta_g\nabla_g \mathrm{A}_{g}$, $\nabla_g \delta_g \mathrm{A}_{g}$ and $\mathrm{Rm}_g \cdot \mathrm{A}_{g}$ stand for the covariant tensors whose coordinates are given by
    \begin{equation*}
        (\delta_g \nabla_g \mathrm{A})_{i j}:=-\mathrm{A}_{i j ; m}^m, \quad (\nabla_g \delta_g \mathrm{A}_g)_{i j}:=-\mathrm{A}_{i;mj}^m \quad \text { and } \quad (\mathrm{Rm}_g \cdot \mathrm{A}_g)_{ij}:=\mathrm{A}_{im\ell}^m \mathrm{A}_j^\ell+\mathrm{Rm}_{i \ell j m} \mathrm{A}^{m\ell},
    \end{equation*}
    where $\mathrm{Rm}_{i \ell j m}$ are the coordinates of the Riemann curvature tensor.

    \bibliography{references}
    \bibliographystyle{abbrv}

    \end{document}